\documentclass[11pt,reqno]{amsart}
\pdfoutput=1
\theoremstyle{plain}
\newtheorem{lma}{Lemma}

\newtheorem{teorema}{Theorem}
\newtheorem*{teorema*}{Theorem}
\theoremstyle{definition}
\newtheorem{definicion}{Definition}
\newtheorem{rmk}{Remark}
\usepackage{float}
\usepackage{amsmath,amssymb}
\usepackage{amsthm}
\usepackage{verbatim}
\usepackage{enumitem}
\usepackage[most]{tcolorbox}
\usepackage{tikz}
\usepackage{mathtools}
\usepackage[T1]{fontenc}
\usepackage[utf8]{inputenc}
\usepackage{microtype} 
\usepackage[margin=1in]{geometry}
\usepackage{hyperref}
\usepackage{dsfont}
\newcommand{\ind}{\mathds{1}}
\usetikzlibrary{arrows.meta,calc,decorations.pathreplacing}
\usepackage[skip=6pt plus2pt, indent=0pt]{parskip}
\numberwithin{equation}{section}
\allowdisplaybreaks
\newsavebox{\pbbox}
\newenvironment{parambracket}
  {\par\addvspace{2.5ex}\refstepcounter{equation}%
   \begin{lrbox}{\pbbox}%
   \begin{minipage}{\dimexpr\linewidth-6.5em\relax}\setlength{\parindent}{0pt}}
  {\end{minipage}\end{lrbox}%
   \noindent\hspace*{1.5em}%
   \begin{tikzpicture}[line width=0.5pt]
     \node[inner sep=0pt] (pb) {\usebox{\pbbox}};
     \draw[decorate,decoration={brace,amplitude=7pt,mirror}]
       ([xshift=-0.7em]pb.north west) -- ([xshift=-0.7em]pb.south west);
     \node[anchor=west,inner sep=0pt] at ([xshift=1em]pb.east) {$(\theequation)$};
   \end{tikzpicture}\par\addvspace{2.5ex}}
\tikzset{lbl/.style={font=\fontsize{14}{9.6}\selectfont}}
\newtcolorbox{mybox}[1]{
    colback=gray!10,
    colframe=black,
    fonttitle=\bfseries,
    title={#1}
}
\definecolor{vecverde}{RGB}{30,140,60}
\definecolor{vecrojo}{RGB}{206,88,66}
\tikzset{
  vec/.style={line width=0.75pt, -{Stealth[length=2.6mm,width=1.9mm]}},
}
\title[Finite-time blow-up for generalized SQG]{Blow-up at finite time for the generalized SQG equations in the Sobolev well-posedness regime}
\author{Diego C\'ordoba}
\address{Instituto de Ciencias Matem\'aticas CSIC-UAM-UCM-UC3M, Spain}
\email{dcg@icmat.es}

\author{\'Oscar Dom\'inguez}
\address{Departamento de Matem\'aticas, CUNEF Universidad, Madrid, Spain}
\email{oscar.dominguez@cunef.edu}

\author{Jos\'e Lucas-Manch\'on}
\address{Instituto de Ciencias Matem\'aticas CSIC-UAM-UCM-UC3M, Spain}
\email{jose.lucas@icmat.es}

\author{Luis Mart\'inez-Zoroa}
\address{Departamento de Matem\'aticas, CUNEF Universidad, Madrid, Spain}
\email{luis.martinezzoroa@cunef.edu}
\begin{document}
\begin{abstract}
We prove finite-time singularity formation for the forced generalized surface quasi-geostrophic equation in the singular velocity regime $\gamma\in(0,1)$, where $\gamma=0$ corresponds to SQG. For every such $\gamma$, we construct a smooth, compactly supported initial datum and a time-dependent force $F$ for which the corresponding solution $\theta$ is classical on $[0,1)$ with finite energy for all times, but loses Sobolev regularity at $t=1$. More precisely, there exists
\[
\kappa_0>2+\gamma+\frac{\gamma^2(1-\gamma)}{25(4+\gamma)},
\]
such that the force satisfies $F\in L^1([0,1];H^\kappa(\mathbb{R}^2))$ for every $\kappa\in[2+\gamma,\kappa_0]$, whereas
\[
\lim_{T\nearrow1}\int_0^T\|\theta(\cdot,t)\|_{H^\kappa}\,dt=\infty
\]
for every exponent in the same interval. At the same time, the solution remains uniformly bounded in $H^{\kappa_1}$ throughout its lifespan for any \[\kappa_1\in\left[0,2+\gamma-\frac{\gamma(1-\gamma)}{2(4+\gamma)}\right].\]
Then, the singularity occurs within the Sobolev well-posedness regime and cannot be attributed to insufficient regularity of the force or the initial conditions. To the best of our knowledge, this is the first finite-time blow-up result for classical finite-energy solutions of the generalized SQG equations in a well-posedness regime.
\end{abstract}
\maketitle
\tableofcontents
\section{Introduction}

We consider classical solutions of the forced generalized surface
quasi-geostrophic equation on $\mathbb{R}^2$,
\begin{equation}\label{ecuacion}
\left\{
\begin{aligned}
&\partial_t\theta + v^{\gamma}(\theta)\cdot\nabla\theta = F,
   && (x,t)\in\mathbb{R}^2\times[0,\infty),\\[2pt]
&v^{\gamma}(\theta)=\nabla^{\perp}\psi_{\gamma},
   \qquad \psi_{\gamma}=-\Lambda^{-1+\gamma}\theta,\\[2pt]
&\theta(\cdot,0)=\theta_0,
\end{aligned}
\right.
\end{equation}
where $\Lambda=(-\Delta)^{1/2}$ and
$\gamma\in(-1,1)$. Equivalently,  the velocity can be
written as
\begin{equation}\label{velocidad}
v^{\gamma}(\theta)(x,t)=c_{\gamma}\,\text{p.v.}\!\int_{\mathbb{R}^2}
   \frac{(x-y)^{\perp}}{|x-y|^{3+\gamma}}\,\theta(y,t)\,dy,
\end{equation}
for a certain nonzero normalization constant $c_{\gamma}$. The velocity is divergence
free, and the system \eqref{ecuacion} is an active scalar transport equation with
a given external force $F$.

This parametrization places the two-dimensional Euler equation in vorticity
form at $\gamma=-1$ and the inviscid SQG equation at $\gamma=0$. For
$\gamma<0$, the velocity is smoother than the transported scalar; at
$\gamma=0$, it has the same regularity; and for $\gamma>0$, which is the range
considered in this paper, the velocity is $\gamma$ derivatives more singular
than $\theta$.

The SQG equation was introduced as a model for the evolution of surface
potential temperature in rapidly rotating geophysical flows~\cite{HPGS}.
Constantin, Majda and Tabak~\cite{CMT} emphasized its close analogy with the
three-dimensional incompressible Euler equations and initiated the analytical
and numerical study of possible singularity formation (see also \cite{Const}). The generalized
family~\eqref{ecuacion} provides a natural scale of active scalar equations in
which the degree of singularity of the Biot--Savart law can be varied; see,
among others, \cite{CCW,CCCGW}.

For the unforced equation, local well-posedness in Sobolev spaces is known for
\begin{equation}\label{eq:sobolev-data}
\theta_0\in H^{\kappa}(\mathbb{R}^2),\qquad \kappa>2+\gamma,
\end{equation}
with the analogous theory for a time-dependent force
$F\in L^1_{\mathrm{loc}}H^{\kappa}$ (check Subsection \ref{locale}). The first local well-posedness result for $\gamma\in(0,1) $ was proven in~\cite{CCCGW} for initial data in $H^4$. In ~\cite{HuKukavicaZiane} this result is extended for positive $\gamma$ down to the critical Sobolev exponent. In the more regular regime $\gamma\in(-1,0]$, similar results
were obtained in~\cite{Inci} and~\cite{YuZhang}. The threshold reflects the
increasing singularity of the velocity as $\gamma$ grows. Below it, several
forms of norm inflation and instant loss of regularity are known. For SQG, these
were established in~\cite{CMZa,JK}; for the generalized family, including the
regime $\gamma>0$, see~\cite{CLMMZ,CJO,CMZb,CMZO}. These results show that the
Sobolev threshold in~\eqref{eq:sobolev-data} is a genuine dividing line. They
do not, however, produce a finite-time singularity from data lying above that
threshold.

Whether classical solutions of the SQG and gSQG equations on $\mathbb{R}^2$, in a
local well-posedness regime, develop a singularity in finite time is a major open problem.
A number of constructions and numerical experiments have revealed mechanisms
that may lead to singular behavior. Castro and Córdoba~\cite{CC} constructed
singular SQG solutions with infinite energy (see also \cite{HQSW} for a recent similar
construction in the case of $\gamma>0$), while García and
Gómez-Serrano~\cite{GGS} constructed nonradial self-similar spirals for
gSQG, again outside the finite-energy class (see also \cite{Abe}). Bronzi, Guimarães and
Mondaini~\cite{BGM} analyzed locally self-similar blow-up scenarios for gSQG and,
under suitable growth assumptions on the profile and its gradient, ruled out
broad classes of such profiles or characterized their possible asymptotic
behavior. 

A line of work closely related to blow-up studies concerns growth of high
Sobolev or Hölder norms. Rigorous
results for classical finite-energy SQG solutions exhibiting long-time growth
are due to Kiselev--Nazarov and He--Kiselev. Kiselev and Nazarov \cite{KiselevNazarov} constructed
arbitrarily small $H^{\kappa}$ initial data whose $H^{\kappa}$ norm exceeds any
prescribed size at a later time. He and Kiselev \cite{HeKiselev}
constructed solutions that either develop a finite-time singularity or, if
globally smooth, exhibit at least exponential growth of $C^{1,\alpha}$ norms for
$t\to\infty$. This last result also applies to the generalized SQG equations obtaining at least exponential growth of certain Hölder norms. 

Rigorous finite-time singularities have been obtained for a family of
weak solutions in settings whose geometry or solution class differs
essentially from the classical whole-plane problem. In particular,  Kiselev, Ryzhik, Yao and
Zlatoš~\cite{KRYZ}, and later Gancedo and Patel~\cite{GP}, proved singularity
formation for generalized SQG patches on the half-plane. Zlatoš subsequently
obtained finite-time blow-up from smooth data for the gSQG equation itself
on the half-plane~\cite{Z}; the boundary is a crucial part of that mechanism.
Miao, Tan, Xue and Xue \cite{MiaoTanXueXue} then treated a half-plane model with a general Fourier
multiplier $m(\Lambda)$ and proved finite-time singularity formation under an
Osgood-type condition and mild additional hypotheses. More
recently, Jeon and Zlatoš established singularity formation for touching
gSQG patches on the whole plane, see \cite{JZ}. None of these results gives
finite-time blow-up for a classical, finite-energy solution on $\mathbb{R}^2$ whose
initial datum lies in the standard Sobolev well-posedness regime.

The purpose of this paper is to provide such a construction for the forced
equation. For every $\gamma\in(0,1)$, we construct a smooth, compactly
supported initial datum and a force that remains integrable in time in Sobolev
spaces strictly above the local well-posedness threshold, but for which the
corresponding classical solution becomes singular in finite time. To the best
of our knowledge, this is the first finite-time blow-up construction for a
classical solution of a generalized SQG equation on $\mathbb{R}^2$ with finite energy
and with both the initial datum and the force in the Sobolev well-posedness
regime.

An important feature of the result is that the singularity is not inserted
through a force that is itself singular at the level of local well-posedness.
Indeed, the force belongs to $L^1([0,1];H^{\kappa})$ for a nonempty interval of
exponents $\kappa>2+\gamma$. Moreover, since $v^{\gamma}$ is divergence free,
the standard $L^2$ estimate gives
\begin{equation}\label{eq:energy}
\sup_{0\le t<1}\|\theta(t)\|_{L^2}
   \le \|\theta_0\|_{L^2}+\|F\|_{L^1([0,1];L^2)}<\infty.
\end{equation}
Thus the solution has finite energy throughout its lifespan, even though its
higher Sobolev norms become non-integrable at the singular time.

\subsection{Main result}

Our main theorem reads as follows.

\begin{teorema}\label{teorema}
Let $\gamma\in(0,1)$. There exists
\begin{equation}\label{eq:kappa0}
\kappa_0>2+\gamma+\frac{\gamma^2(1-\gamma)}{25(4+\gamma)},
\end{equation}
a scalar
\[
\theta:\mathbb{R}^2\times[0,1)\longrightarrow\mathbb{R},
\]
and a force
\[
F:\mathbb{R}^2\times[0,1)\longrightarrow\mathbb{R},
\]
with the following properties:
\begin{enumerate}
\item For every $T\in(0,1)$,
\[
\theta,\,F\in C^{\infty}\!\big([0,T];C^{\infty}_c(\mathbb{R}^2)\big),
\]
and $\theta$ is a classical solution of the forced $\gamma$-gSQG equation~\eqref{ecuacion}  on
$[0,1)$; namely,
\[
\partial_t\theta(x,t)+v^{\gamma}(\theta)(x,t)\cdot\nabla\theta(x,t)=F(x,t)
\]
for every $(x,t)\in\mathbb{R}^2\times[0,1)$.

\item The force remains in the Sobolev well-posedness class up to and including
the singular time:
\[
F\in L^1\!\big([0,1];H^{\kappa}(\mathbb{R}^2)\big)
\quad\text{for every} \quad \kappa\in[2+\gamma,\kappa_0].
\]

\item \label{puntodivergencia}For every $\kappa\in[2+\gamma,\kappa_0]$,
\[
\lim_{T\nearrow 1}\int_{0}^{T}\|\theta(\cdot,t)\|_{H^{\kappa}}\,dt=\infty.
\]
\item Moreover, for any $\kappa\in\left[0,2+\gamma-\frac{\gamma(1-\gamma)}{2(4+\gamma)}\right]$,
\begin{align*}
    \sup_{0\leq t< 1} \|\theta(\cdot,t)\|_{H^\kappa}<\infty.
\end{align*}
\end{enumerate}
\end{teorema}

Since the strict inequality in~\eqref{eq:kappa0} places $\kappa_0$ above
$2+\gamma$, Theorem~\ref{teorema} yields blow-up simultaneously for a
nontrivial interval of Sobolev exponents in the well-posedness regime. Notice
also that $\theta_0=\theta(\cdot,0)\in C^{\infty}_c(\mathbb{R}^2)$. The divergence in
the item \ref{puntodivergencia} shows that the classical solution cannot be continued through $t=1$
within any of the corresponding Sobolev classes, while~\eqref{eq:energy} rules
out loss of finite energy as the source of the breakdown.

    
\subsubsection{Local well-posedness with forcing} \label{locale}  The local theory for the unforced equation extends directly to external forces in
$$
F\in L^1_{\mathrm{loc}}\big([0,\infty);H^\kappa(\mathbb R^2)\big).
$$

More precisely, if $\kappa>2+\gamma$, $\theta_0\in H^\kappa(\mathbb R^2)$, and
$F\in L^1_{\mathrm{loc}}([0,\infty);H^\kappa(\mathbb R^2))$, then there exists
$T>0$ and a unique solution to \eqref{ecuacion} with 
$$
\theta\in C\big([0,T];H^\kappa(\mathbb R^2)\big).
$$
Indeed, for smooth solutions, the commutator estimate used in
\cite{HuKukavicaZiane}, together with the standard $L^2$ energy estimate, gives
\begin{align*}
\frac12\frac{d}{dt}\|\theta(t)\|_{L^2}^2
&\leq \|\theta(t)\|_{L^2}\|F(t)\|_{L^2},\\
\frac12\frac{d}{dt}\|\Lambda^\kappa\theta(t)\|_{L^2}^2
&\leq C_{\gamma,\kappa}\|\theta(t)\|_{H^\kappa}^3
+\|\Lambda^\kappa\theta(t)\|_{L^2}
\|\Lambda^\kappa F(t)\|_{L^2}.
\end{align*}
Consequently, setting
$$
Y(t):=\left(
\|\theta(t)\|_{L^2}^2+
\|\Lambda^\kappa\theta(t)\|_{L^2}^2
\right)^{1/2}
\simeq \|\theta(t)\|_{H^\kappa},
$$
we obtain, for almost every $t$,
$$
Y'(t)\leq C_{\gamma,\kappa}Y(t)^2+\|F(t)\|_{H^\kappa}.
$$
If
$$
A_T:=Y(0)+\|F\|_{L^1([0,T];H^\kappa)}
$$
and $T>0$ is small enough such that $C_{\gamma,\kappa}TA_T<1$, the nonlinear Grönwall inequality yields
$$
\sup_{0\leq t\leq T}Y(t)
\leq \frac{A_T}{1-C_{\gamma,\kappa}TA_T}.
$$
The standard approximation and the corresponding difference estimates then give existence, uniqueness, and continuous dependence in $H^\kappa$. Accordingly, 
$L^1_{\mathrm{loc}}H^\kappa$ is the natural forcing class for the local
$H^\kappa$ theory.

\subsubsection{Layered structure of the solution} The proof of Theorem \ref{teorema} relies on a multiscale construction in which the solution is expressed as an infinite superposition of localized layers, activated at times $t_n \nearrow 1$ and concentrated near the origin at progressively finer spatial scales. The linear part of the velocity generated by the previously activated layers transports and shears each newly introduced layer, thereby increasing its anisotropy and its high order Sobolev norms. Once deformed, that layer contributes a stronger linear velocity, which drives the same mechanism at the next, finer scale. At each stage, we also introduce a finite hierarchy of correction terms that successively cancel the leading self-interaction errors of the new layer, so that the residual force is summable in the required Sobolev norms. Iterating this corrected scheme produces a cascade of anisotropy and Sobolev growth that accumulates at $t=1$, making $\|\theta(\cdot,t)\|_{H^\kappa}$ non-integrable in time for the range of exponents stated in the main theorem, while the force remains integrable in the corresponding Sobolev well-posedness spaces.

   \subsection{Comparison with other finite time singularities for 2D incompressible flows}
   
   Multilayer methods have very recently been applied to some other 2D models for incompressible flows, including IPM \cite{CMZc}, the Boussinesq equations \cite{CLSMZa} and the inhomogeneous Euler equations \cite{CLSMZb}. However, gSQG presents a very unique set of difficulties that makes the problem especially challenging. In this subsection we will discuss some of these difficulties.

    First, and maybe the most important difference between gSQG and the models mentioned previously is the lack of a clear and strong growth mechanism. For IPM, perturbation of a linear density $-Ax_{2}$ can  grow exponentially since 
   $$u_{IPM}(\rho_{pert})\cdot\nabla (-Ax_{2})\approx A\rho_{pert}$$
    for $\rho_{pert}$ with the right properties. This is a very robust growth mechanism, which is also present when adding a perturbation to a solution with negative derivative in the $x_2$ direction at some point. Furthermore, the growth is (exponentially) fast  and all the relevant norms of the perturbation grow (both subcritical and supercritical). 
    For Boussinesq, the presence of the term $\partial_{x_2}\rho$ in the equation for the time derivative also gives a very clear mechanism to get growth of the vorticity. In fact, one can again consider perturbations of linear densities and obtain exponential growth. Finally, for inhomogeneous Euler, one can use the fact that, under certain conditions, Boussinesq is a good approximation of inhomogeneous Euler, and thus the growth mechanism from Boussinesq can be exploited.

    If one tries to use a similar idea for gSQG perturbing some carefully chosen background and  hoping to obtain some robust growth, we immediately run into some essential obstructions. If we ignore the self interactions of the perturbation, we would obtain the evolution equation
    $$\partial_{t}\theta_{pert}+u(\theta_{pert})\cdot\nabla \theta_{bg}+u(\theta_{bg})\cdot\nabla \theta_{pert}=0,$$

    where $\theta_{bg}$ is the background solution we are perturbing and $\theta_{pert}$ is the perturbation we want to grow. The term $u(\theta_{bg})\cdot\nabla \theta_{pert}$ is a transport term, and it therefore cannot make any $L^{p}$  norm grow, so we would like to use $u(\theta_{pert})\cdot\nabla \theta_{bg}$ for our growth. When considering concentrated perturbations, the main order of this term will be
    $$u(\theta_{pert})\cdot\nabla \theta_{bg}(x=0,t)=a(t)u_1(\theta_{pert})+b(t)u_2(\theta_{pert})$$
    where $a(t)$ and $b(t)$ come from the derivatives of the background at the origin.
    However, due to the odd parity of the velocity operator, these terms are orthogonal in $H^s$ to $\theta_{pert}$, and therefore cannot produce growth in any $H^s$ norm. This is especially restrictive in the gSQG case since (as we will discuss later), for the equations we consider, we only have local well-posedness in $H^s$, and not in other spaces like $C^{k,\alpha}$ where $u(\theta_{pert})$ could give growth.

    This makes it very hard to obtain the term $u(\theta_{pert})\cdot\nabla \theta_{bg}$ for growth, and leaves us with the transport term. This already  creates a new challenge since, as we mentioned before, the transport term will not produce any growth in the $L^p$ norms, and therefore it is a more subtle growth mechanism. One can, however, still hope for a growth mechanism that is fast, ideally exponentially fast, like the ones we described earlier.

    The obvious way to attempt this would be as follows: We consider a background that generates a hyperbolic velocity at the origin, add a perturbation at the origin, and the hyperbolic flow will deform  the perturbation exponentially fast. The perturbation grows quickly due to this exponential deformation, and after some time it generates its own strong hyperbolic flow. 
    Here we encounter another problem: If we want the perturbation to create a strong hyperbolic flow, we need
    $\partial_{1}u_{1}(\theta_{pert})$ to be big, which means $\partial_{x_{1}}\partial_{x_{2}}(-\Delta)^{\frac{-1+\gamma}{2}}\theta_{pert}$ must be big. However, the hyperbolic flow actually keeps $\partial_{x_{1}}\partial_{x_{2}}$ constant (since compression in $x_{1}$ gets exactly compensated by the decompression in $x_{2}$). This leaves $(-\Delta)^{\frac{-1+\gamma}{2}}$ as the only part that could grow. But, in general, the deformation will send our perturbation to higher frequencies, and therefore $(-\Delta)^{\frac{-1+\gamma}{2}}$ will make us smaller, not bigger, if we are considering $\gamma<1$, which is the case of interest.

    More precisely, one can for example consider sinusoidal perturbations $\sin(ax_1)\sin(bx_{2})$, and observe that $\partial_{x_1}u_{1}(\theta_{pert})$ can only become big after we deform the perturbation with a hyperbolic flow if it was already outside of the well-posedness regime before any deformation happened.

    Therefore, using an exponential deformation scenario seems also out of reach. This leaves us with one last option: Considering growth via linear transport. In particular, we can consider a background with $\partial_{x_2}u_{1}(\theta_{bg})$ big, and use it to deform a perturbation so that $\partial_{x_2}u_1(\theta_{pert})$ becomes big. In this case, we do not run into the same problem as before, since in this case we deal with $\partial_{x_2}\partial_{x_{2}}(-\Delta)^{\frac{-1+\gamma}{2}}$, which can grow as we make the $x_{2}$ derivative bigger as long as $\gamma>-1$. As we will see, this is in fact too optimistic, and we will only obtain blow-up for $\gamma>0$, since, at $\gamma=0$ the growth becomes too slow to give finite time blow-up.
    Furthermore, when $\gamma$ approaches one, as the velocity operator becomes more and more singular, the number of corrections required in our construction grows, in such a way that it is convenient to construct a method that allow us to introduce an arbitrary finite number of corrections, making it very technically challenging to prove the result for the whole range $\gamma\in(0,1)$.

    Finally, another important issue with gSQG (for $\gamma>0$) is the fact that we only have well-posedness in $H^{s}$, and thus we cannot directly work in $L^{\infty}$ based spaces such as $C^{k, \alpha}$. Note that it is not a matter of the well-posedness not being proved yet in $C^{k,\alpha}$: We know that the equations are ill-posed in those spaces (see \cite{CMZb}). To understand why this is restrictive, note that, in order for our solution to blow-up, we need  $\|\theta\|_{C^{1,\gamma}}$ norm to blow up. If consider solutions in $C^{1,\gamma+\varepsilon}$, this would basically mean lose $\varepsilon$ derivatives, without having to worry about things like the support of our forcing or the solution. However, when working in Sobolev spaces, the support of our solutions and forcing suddenly becomes relevant: $f(x)\sin(Nx)/N^{1+\gamma+\varepsilon}$ for example has the right size if we focus on the $L^{\infty}$ spaces (that is,  it is of order one in $C^{1,\gamma+\varepsilon}$) but it is comfortably supercritical when looking at Sobolev spaces, since it it is also of order one in $H^{1+\gamma+\varepsilon}$. This is not a purely technical issue, since many of the properties of our building blocks depend heavily on the size of their support. To put more emphasis on this point, when considering for example the blow-up for Boussinesq \cite{CLSMZa} and inhomogeneous Euler \cite{CLSMZb}, the solutions are well inside the well-posedness regime when considering $L^{\infty}$ based spaces, but none of them live in the well-posedness regime when considering Sobolev spaces.

\subsection{Organization of the paper}
Section \ref{Sketch} described the mechanism heuristically; the rest of the paper makes it rigorous.
One of the main features that we exploit to make the construction work is the fact that the building blocks
\[
  \mathcal{C}\phi\big(r\,\textbf{n}(\alpha)\cdot x\big)\phi\big(s\,\textbf{n}(\beta)\cdot x\big)
\]
form a class that is stable under everything the approximated equations do to them. Similar thing happens with the corrections of each layer. 

Sections \ref{SectionSolution}, \ref{SF}, \ref{SectionResult} verify
this stability and exploit it, as follows.

In Section \ref{SectionSolution} we introduce the class itself (Definition \ref{sod}: \emph{sequences of derivatives},
parametrized by an amplitude $\mathcal{C}$, two principal directions $(\alpha,\beta)$ and two
principal frequencies $(r,s)$) and we show that it is preserved by the two operations that generate the corrections of our
solution: taking a self-interaction $v^{\gamma}(f)\cdot\nabla f$ (Lemma \ref{nuevaperturbacion}) and solving a linear transport equation along the flow of the linearized velocity field generated by the previous layers forced with the self-interaction term (Lemma \ref{nuevaperturbacion2}).
Section \ref{sec:velocitylayers} computes the velocity generated by such a function (Lemma \ref{lemavelocidad}): its Taylor expansion
at the origin consists of a linear part, with three explicit coefficients
$a^{(1,1)},a^{(1,2)},a^{(2,1)}$, plus a cubic remainder with coefficients $b^{(j,k)}$. Section \ref{S3.2} shows that transporting a building block by a
linear velocity field keeps it in the class and reduces its evolution to a closed system of
four ODEs for $(r,\alpha)$ and $(s,\beta)$ (Lemma \ref{dinamica}), and then extracts the quantitative
behavior  of those ODEs in the regime of interest (Lemma \ref{lemaedos}): the new layer is essentially
subject to a pure shear of intensity $N_n^{\varepsilon}$, so that its long frequency grows
linearly in time while one of its principal directions rotates towards the principals directions of the previous layer.
Section \ref{construccion} assembles the induction (Lemma \ref{Lemma6}): assuming the description of the first $n$
layers, we construct the $(n+1)$-th, fix the frequency $N_{n+1}$, and recover the same
description one step further.

Sections \ref{SF} and \ref{SectionResult} are then estimates rather than construction: Section \ref{SF} bounds in $H^{\kappa}$
each of the four error terms left by the previous scheme, and Section \ref{SectionResult} checks that the four
bounds are summable and that the principal parts of the layers produce a finite time blow-up of the $H^{\kappa}$ norm.

\subsection{Notation}
\begin{itemize}
    \item Given two positive quantities $A$ and $B$, we write $A\lesssim B$ when $A\leq CB$ for some constant $C>0$ that is independent of all important quantities. To stress that $C$ depends on a certain constant $c > 0$, we use the notation $A \lesssim_c B$. If $A\lesssim B$ and $B\lesssim A$, then we write $A\simeq B$. 
    
    On the other hand, by $A \approx B$ we mean that $A$ is approximately close to $B$ (but not necessarily $A \simeq B$). Since the symbol $\approx$ lacks a rigorous mathematical definition, we will only use it in heuristics. 

    \item We write $a^\perp:=(-a_2,a_1)$ for $a= (a_1, a_2) \in \mathbb{R}^2$ and $\nabla^\perp:=(-\partial_2,\partial_1)$.
    \item Let $s\in (-1,0)$. The fractional Laplacian of a compactly supported smooth function $f:\mathbb{R}^2\rightarrow \mathbb{R}$  is defined by
    \begin{align}
    \label{laplacianofraccionario}
        (-\Delta)^{s}f(x):=\frac{4^s \Gamma (1+s)}{\pi \Gamma (-s)}\int_{\mathbb{R}^2}\frac{f(x-y)}{|y|^{2+2s}} \, dy, \qquad \forall x\in \mathbb{R}^2.
    \end{align}
    \item We denote the unit vector in the direction $\alpha\in \mathbb{R}$ as $\textbf{n}(\alpha):=(\cos \alpha,\sin \alpha)$.
    \item The (counterclockwise) rotation of angle $\alpha\in \mathbb{R}$ is defined as 
    \begin{align*}
        R_\alpha x:=\left(\begin{matrix}
            \cos \alpha & -\sin \alpha\\
            \sin \alpha & \cos \alpha
        \end{matrix}\right)x, \qquad \forall x\in \mathbb{R}^2.
    \end{align*}
    \item Given a function $f:\mathbb{R}^2\rightarrow \mathbb{R}$, its $\alpha$-rotated function is
    \begin{align*}
        R_\alpha f(x):=f(R_{-\alpha}x), \qquad \forall x \in \mathbb{R}^2.
    \end{align*}
    \item Let $\kappa_1,\kappa_2\in \mathbb{N}_0$. We define the differential operator $
        D^{(\kappa_1,\kappa_2)}:=\frac{\partial^{\kappa_1+\kappa_2}}{\partial x_1^{\kappa_1}\partial x_2^{\kappa_2}}.$
    \item The indicator function of the set $A\subset \mathbb{R}^2$ is denoted by $\ind_A$.  
    \item The open ball of radius $R>0$ centered at $x\in \mathbb{R}^2$ is denoted by $B_x(R)$.
\end{itemize}
  \section{Sketch of the proof}\label{Sketch}
   In this section, we outline the key ideas involved in the proof of Theorem \ref{teorema}. 
   
    Take $\phi:\mathbb{R}\rightarrow\mathbb{R}$ an even, compactly supported, smooth bump around the origin. Let $\varepsilon,\delta>0$ be two small parameters to be specified later, subject to $\gamma\varepsilon>>\delta$. Then, consider the stationary function
   \begin{align}\label{P16.8.2}
      \theta_0(x, t)=N_0^{-1-\gamma+\varepsilon}\phi(N_0x_1).
   \end{align}
   Discarding  multiplicative constants for simplicity and Taylor expanding $y_1 \mapsto \phi'(N_0 (x_1-y_1))$ around $y_1=0$, we obtain
   \begin{align*}
       v^\gamma(\theta_0)(x, t)&=-\int_{\mathbb{R}^2}\frac{\nabla^\perp \theta_0(x-y)}{|y|^{1+\gamma}} \, dy =-N_0^{-\gamma+\varepsilon}\int_{\mathbb{R}^2}\left(\begin{matrix}
           0\\
           \phi'(N_0(x_1-y_1))
       \end{matrix}\right) \, \frac{dy}{|y|^{1+\gamma}} \\
       &=-N_0^{1-\gamma+\varepsilon}x_1\int_{\mathbb{R}^2}\left(\begin{matrix}
           0\\
           \phi''(N_0y_1)
       \end{matrix}\right)\frac{dy}{|y|^{1+\gamma}}+O_{N_0}(|x|^3)\\
       &=\left(\begin{matrix}
           0\\
           N_0^\varepsilon x_1\int_{\mathbb{R}^2}\frac{(-\phi''(y_1))}{|y|^{1+\gamma}} \, dy
       \end{matrix}\right)+O_{N_0}(|x|^3).
   \end{align*}
  Here,  the zeroth and second order terms of the Taylor expansion vanish because of the parity of $\phi$. Hence, the first order of the Taylor expansion around the origin of $v^\gamma(\theta_0)$ (say, $\overline{v}^\gamma (\theta_0)$) is of the form
   \begin{align}
   \label{velocidadbarra}
       \overline{v}^\gamma (\theta_0)(x, t)=\left(
\begin{matrix}
    0\\ N_0^\varepsilon x_1
\end{matrix}
       \right),
   \end{align}
   where we have omitted the multiplicative constant depending on $\phi$ (it is easy to see that $\int_{\mathbb{R}^2}\frac{(-\phi''(y_1))}{|y|^{1+\gamma}} \, dy > 0$, apply integration by parts together with the fact that $\phi$ is non-increasing on $(0, \infty)$).

   
   Notice that the function $\theta_0$ introduced in \eqref{P16.8.2} does not have finite energy. To overcome this, let us redefine $\theta_0$ as 
   \begin{align}\label{P16.8.1}
       \theta_0(x, t)=N_0^{-1-\gamma+\varepsilon}\phi(N_0x_1)\phi(x_2).
   \end{align}
   Clearly, this function has finite energy. Furthermore, applying the same reasoning as before with $N_0 \gg 1$, one can check  that the main linear term of the Taylor expansion of $v^\gamma(\theta_0)$ with \eqref{P16.8.1} around the origin is still of type \eqref{velocidadbarra}. 

   Now, consider that the approximate velocity $\overline{v}^\gamma(\theta_0)$  transports a new layer $\theta_1$,  which is initially defined by 
   \begin{align}
   \label{condicioninicial1}
       \theta_1(x,t_1)=N_1^{-1-\gamma+\varepsilon}\phi\Big(N_1^{1-\frac{\varepsilon+\delta}{1+\gamma+\delta}}x_1\Big)\phi\Big(N_1^{1-\frac{\varepsilon+\delta}{1+\gamma+\delta}}x_2\Big),
   \end{align}
   where the initial time $t_1$ and $N_1$ will be adequately chosen.  
The choice of the exponent $1-\frac{\varepsilon+\delta}{1+\gamma+\delta}$ in the frequency of $\theta_1$ is not arbitrary and it is  justified by the fact that, by homogeneity,   $\|\theta_1(\cdot,t_1)\|_{H^{2+\gamma+\delta}}\approx 1$ (at least if $N_1 \gg 1$). In other words, the new layer $\theta_1$ becomes critical within the Sobolev regularity $2+\gamma+\delta$, while its Sobolev norms with exponents strictly smaller than $2+\gamma+\delta$ will tend to zero as $N_1 \to \infty$. Since the forcing term is responsible for generating the successive layers, controlling the forcing requires controlling the Sobolev size of each newly generated layer. The previous reasoning shows that $2+\gamma+\delta$ naturally appears as the threshold Sobolev exponent for which the construction can be expected to remain controlled.
   
  As mentioned above, assume that $\theta_1(x,t)$ solves 
   \begin{align}
   \label{pdeaprox}
       \partial_t\theta_1+\overline{v}^\gamma(\theta_0)\cdot \nabla \theta_1=0, \qquad \forall t\in [t_{1},1],
   \end{align}
   with the initial condition \eqref{condicioninicial1}. The unique solution to this transport equation is 
   \begin{align}
   \label{soluciontheta1}
       \theta_1(x,t)=N_1^{-1-\gamma+\varepsilon}\phi\Big(N_1^{1-\frac{\varepsilon+\delta}{1+\gamma+\delta}}x_1\Big)\phi\Big(N_1^{1-\frac{\varepsilon+\delta}{1+\gamma+\delta}}\left(x_2-N_0^\varepsilon (t-t_1)x_1\right)\Big).
   \end{align}
   One can check that, for $t_2>t_1$ sufficiently large, the supports of the layers $\theta_0$ and $\theta_1$ evolve as depicted in Figure \ref{fig:capa1} (in particular, we note that the size of the  support of $\theta_1$ is preserved in time since $\nabla \cdot \overline{v}^\gamma(\theta_0) = 0$). 
\hspace{2cm}
\begin{figure}[h]
\centering
\begin{tikzpicture}[scale=0.6, transform shape, line width=0.6pt]
  \draw (0,0) rectangle (7,15.44);
  \draw (2.40,6.62) rectangle (4.60,8.82);   
  \node[lbl, anchor=north west] at (0.15,15.28) {$\text{supp }\theta_0(\cdot,t_1)$};
  \node[lbl, anchor=south]      at (3.5,8.9)    {$\text{supp }\theta_1(\cdot,t_1)$};
\end{tikzpicture}
\hspace{3cm}
\begin{tikzpicture}[scale=0.6, transform shape, line width=0.6pt]
  \draw (0,0) rectangle (7,15.44);
  \draw (4.60,12.40) -- (4.60,11.08) -- (2.40,3.04) -- (2.40,4.36) -- cycle;
  \node[lbl, anchor=north west] at (0.15,15.28) {$\text{supp } \theta_0(\cdot,t_2)$};
  \node[lbl, anchor=east]       at (3.75,10.5)  {$\text{supp }\theta_1(\cdot,t_2)$};
\end{tikzpicture}
 \caption{Evolution of the supports of $\theta_0$ and $\theta_1$.}
  \label{fig:capa1}
\end{figure}
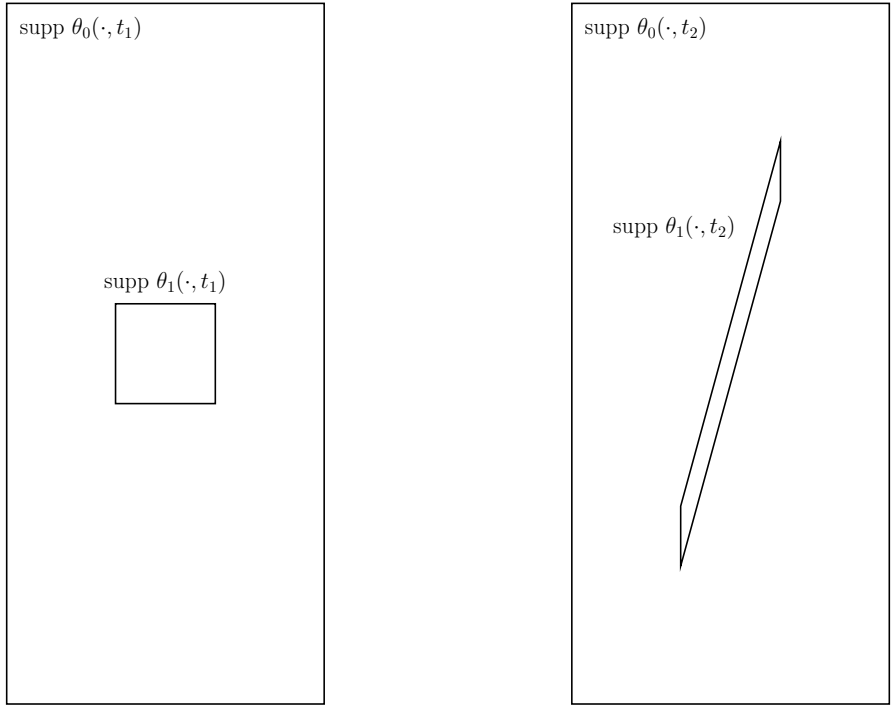

   If we estimate the norm $\|\theta_1(\cdot,t)\|_{H^{\kappa}}$ for $\kappa \geq 0$, namely, 
   \begin{align}
       \|\theta_1(\cdot,t)\|_{H^\kappa}&\approx N_1^{-1-\gamma+\varepsilon} \Big(N_1^{1-\frac{\varepsilon+\delta}{1+\gamma+\delta}}(1+N_0^\varepsilon(t-t_1))\Big)^{\kappa} |\text{supp } \theta_1(\cdot,t_1)|^{\frac{1}{2}} \nonumber \\
       &=N_1^{-1-\gamma+\varepsilon+(\kappa-1)(1-\frac{\varepsilon+\delta}{1+\gamma+\delta})}(1+ N_0^{ \varepsilon}(t-t_1))^\kappa. \label{P16.8.5}
   \end{align}
we observe that if $\kappa<2+\gamma+\delta$ then initially (that is, $t=t_1$),
\begin{align*}
    \|\theta_1(\cdot,t_1)\|_{H^{\kappa}}\approx N_1^{-1/C},
\end{align*}
that is,  we introduce a perturbation whose $H^{\kappa}$ norm with  $\kappa<2+\gamma+\delta$ may be sufficiently small. Note that the choice of $N_1$ in \eqref{condicioninicial1} is available, so   if we choose 
\begin{align}\label{P16.8.3}
    N_1\approx N_0^{(1+\gamma+\delta)\frac{\varepsilon}{\varepsilon+\delta}}
\end{align} 
and consider the critical case $t-t_1\approx1$, then
\begin{align*}
\|\theta_1(\cdot,t)\|_{H^{2+\gamma}}\approx N_1^{\varepsilon-\frac{1+\gamma}{1+\gamma+\delta}(\varepsilon+\delta)}N_0^{(2+\gamma)\varepsilon}\geq  N_0^{(2+\gamma)\varepsilon-(1+\gamma)\delta}\geq N_0^\varepsilon.
\end{align*}
As a byproduct, we have found a growth  mechanism in the well-posedness regime $H^{\kappa}$ with $\kappa\in (2+\gamma,2+\gamma+\delta)$. 

To obtain a blow-up in finite time, we would like to accumulate this growth. For this aim, we need to be able to iterate the previous process in a systematic way. In this regard, a basic observation is  that the main feature of $\theta_0$ that allowed us to obtain a velocity of the form \eqref{velocidadbarra} relies on anisotropy: the existence of a direction in which the derivative of $\theta_0$ is very large compared with the derivative in the perpendicular one. But notice that this property also holds  for $\theta_1$ (see \eqref{soluciontheta1}). In fact, if one considers the derivative of $\theta_1$ along the direction 
\begin{align*}
    \textbf{n}=\frac{1}{\sqrt{1+N_0^{2\varepsilon}(t-t_1)^2}} \, (N_0^\varepsilon(t-t_1),-1)
\end{align*}
will be much bigger than the corresponding one along $\textbf{n}^\perp$, so that the most important terms of the first order of the velocity (encoded in $\overline{v}^\gamma(\theta_1)$) verify  (see \eqref{P16.8.3})
\begin{align*}
    \overline{v}^\gamma(\theta_1)(x,t)\cdot \textbf{n}^\perp \approx N_1^{-1-\gamma+\varepsilon}\Big( N_1^{1-\frac{\varepsilon+\delta}{1+\gamma+\delta}}N_0^\varepsilon \Big)^{1+\gamma} \textbf{n}\cdot x=N_1^{\varepsilon}\textbf{n}\cdot x
\end{align*}
in the critical case $t-t_1 \approx 1$. This gives the same critical  behavior of $\overline{v}^\gamma(\theta_1)$ as the corresponding one exhibited by $\overline{v}^\gamma(\theta_0)$, see \eqref{velocidadbarra}. 

Then it is natural to expect  that at time $t_2$ (very close to $1$), a new perturbation $\theta_2$ could appear in the same fashion as $\theta_1(x,t_1)$ but with higher frequencies (that is, $N_2$ instead of $N_1$ in \eqref{condicioninicial1}) and rotated by a small angle (of order $N_0^{-\varepsilon}$) in such a way that two of the sides of its support are parallel to the long sides of the support of $\theta_1(\cdot,t_2)$, as shown in Figure \ref{fig:capa2}. Looking at the zoomed image in Figure \ref{fig:capa2}, one can see that we recover a scenario similar to the one observed before for $\theta_0(\cdot, t_1)$ and $\theta_1(\cdot, t_1)$ but now involving $\theta_1(\cdot, t_2)$ and $\theta_2(\cdot, t_2)$. Therefore we anticipate that the support of $\theta_2$ will elongate as the one exhibited by $\theta_1$.

\begin{figure}[h]
    \centering
\hspace{4cm}
\begin{tikzpicture}[scale=0.6, transform shape, line width=0.6pt]
  \draw (0,0) rectangle (7,15.44);

  \draw (4.60,12.40) -- (4.60,11.08) -- (2.40,3.04) -- (2.40,4.36) -- cycle;

  \draw (3.45,7.81) -- (3.59,7.77) -- (3.55,7.63) -- (3.41,7.67) -- cycle;

  \draw[blue,line width=0.48pt,rotate around={-15.3:(3.5,7.72)}]
        (3.5,7.72) ellipse (0.5 and 0.85);

  \draw[blue,line width=0.48pt,-{Stealth}] (4.15,7.50) to[out=-20,in=200] (9.20,7.90);

  \draw[blue,line width=0.48pt] (11.5,7.72) ellipse (2.2 and 3.6);

  \begin{scope}
    \clip (11.5,7.72) ellipse (2.2 and 3.6);
    \draw (11.54,12.41) -- (9.16,3.73);
    \draw (13.84,11.71) -- (11.46,3.03);
  \end{scope}

  \draw (11.22,8.21) -- (11.99,8.00) -- (11.78,7.23) -- (11.01,7.44) -- cycle;

  \node[lbl, anchor=north west] at (0.15,15.28) {$\text{supp } \theta_0(\cdot,t_2)$};
  \node[lbl, anchor=east]       at (3.75,10.5)  {$\text{supp }\theta_1(\cdot ,t_2)$};
  \node[lbl, anchor=south, fill=white, inner sep=1pt]
        at (11.77,8.55) {$\text{supp }\theta_2(\cdot,t_2)$};
\end{tikzpicture}  
    \caption{Activation of the layer $\theta_2$ at time $t_2$.}
    \label{fig:capa2}
\end{figure}
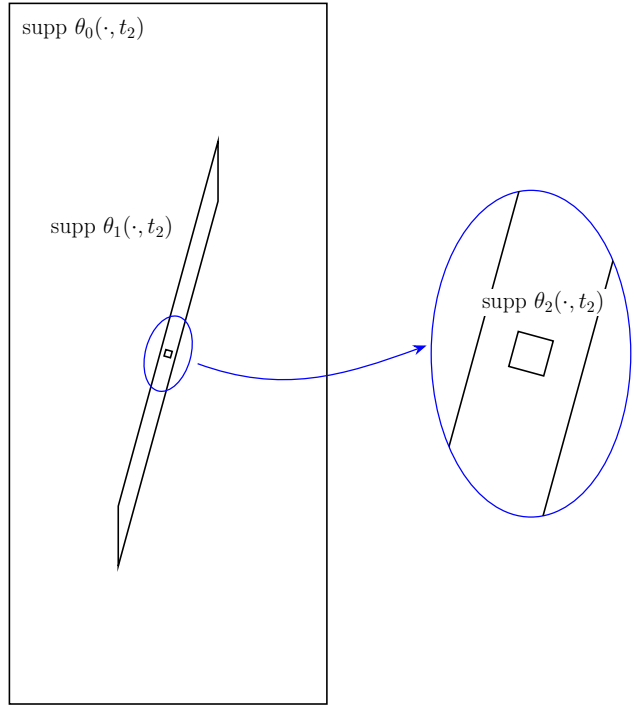

This accumulation of anisotropy around the origin generates at the same time bigger velocities (which help us to iterate the process) and bigger Sobolev norms (which will blow up at time $t=1$). In this way, we  arrive at the sum of layers 
\begin{align}
    \theta(x,t)=\sum_{n=0}^\infty h_n(t) \theta_n(x,t)
\end{align}
as a natural candidate to be a solution. Here,  $h_n$ are smooth non-decreasing functions used to introduce the layers at time $t_n$, more precisely, 
\begin{align}\label{P16.8.4}
\left\{
\begin{aligned}
    h_n(t)&=0,\qquad  \forall t\in [0,t_n),\\
    h_n(t)&=1, \qquad  \text{soon after }  t_n,
    \end{aligned} 
\right.
\end{align}
and where $t_n \nearrow 1$ with $t_0=0$.

Next, let us see how  this mechanism can be reconciled  with solving the forced generalized SQG equation with a regular forcing term. Although we need to iterate this mechanism an infinite number of times, it is enough to consider two layers to see how we are going to proceed. Define
\begin{align*}
    \theta(x,t)=\theta_0(x)+h_1(t)\theta_1(x,t), \qquad \forall (x,t)\in \mathbb{R}^2\times[0,1],
\end{align*}
where $h_1$ a smooth non-decreasing function such that
\begin{align*}
\left\{
\begin{aligned}
    h_1(t)&=0, \qquad \forall t\in [0,t_1], \\
    h_1(t)&=1, \qquad \forall t\in [t_2,1],
    \end{aligned}
\right.
\end{align*}
with $h_1$ to be chosen in $(t_1,t_2)$; see \eqref{P16.8.4}. Consider $F$ such that 
\begin{align*}
    \partial_t \theta(x,t)+v^\gamma(\theta)(x,t)\cdot \nabla \theta(x,t)=F(x,t).
\end{align*}
Then, if we want $\theta_1$ to solve \eqref{pdeaprox},  one has that
\begin{align*}
&\partial_t(\theta_0+h_1\theta_1)+v^\gamma(\theta_0+h_1\theta_1)\cdot \nabla (\theta_0+h_1\theta_1)\\
    &=h_1'\theta_1 +h_1\partial_t\theta_1 + v^\gamma(\theta_0)\cdot \nabla \theta_0 + h_1 v^\gamma(\theta_0)\cdot \nabla \theta_1 +h_1 v^\gamma(\theta_1)\cdot \nabla\theta_0+h_1^2 v^\gamma(\theta_1)\cdot \nabla\theta_1
    \\&=h_1'\theta_1 + v^\gamma(\theta_0)\cdot \nabla \theta_0 + h_1 (v^\gamma(\theta_0)-\overline{v}^\gamma(\theta_0))\cdot \nabla \theta_1 +h_1 v^\gamma(\theta_1)\cdot \nabla\theta_0+h_1^2 v^\gamma(\theta_1)\cdot \nabla\theta_1,
\end{align*}
and then 
\begin{align*}
    F(x,t)=h_1'\theta_1 + v^\gamma(\theta_0)\cdot \nabla \theta_0 + h_1 (v^\gamma(\theta_0)-\overline{v}^\gamma(\theta_0))\cdot \nabla \theta_1 +h_1 v^\gamma(\theta_1)\cdot \nabla\theta_0+h_1^2 v^\gamma(\theta_1)\cdot \nabla\theta_1.
\end{align*}
Accordingly,  we are going to separate the forces into four types:
\begin{itemize}
    \item The activation terms of the type
    \begin{align*}
        h_1'(t)\theta_1(x,t).
    \end{align*}
    This term of the force is connected with the creation of the new layer. 
    \item The cubic transport of the new layer: 
    \begin{align*}
        (v^\gamma(\theta_0)-\overline{v}^\gamma(\theta_0))\cdot \nabla \theta_1.
    \end{align*}
    Since the new layer is highly concentrated around the origin, the transport it undergoes due to the previous layers is well approximated by its linearization. Notice that, since $\theta_0$ is even with respect to the origin, then $v^\gamma(\theta_0)$ is odd and hence the next term in the Taylor expansion is the cubic one.
    \item The transport of the previous layers:
    \begin{align*}
        v^\gamma(\theta_1) \cdot \nabla \theta_0.
    \end{align*}
 In the forcing term, we also neglect the transport that each new layer induces on all the previous ones.
     \item The self-interaction terms:
    \begin{align*}
         v^\gamma(\theta_0)\cdot \nabla \theta_0 , \hspace{5mm}  v^\gamma(\theta_1)\cdot \nabla \theta_1.
    \end{align*}
    That is, we neglect the transport that each layer induces on itself. 
\end{itemize}
Notice that we have only kept the presence of $h_1$ in the activation term, where its derivative appears. The reason behind is that $h_1$ will only be relevant in the activation terms. For the remainder three terms, it suffices to apply that $h_1 \lesssim 1$. Now, we should estimate all these forcing terms in the well-posedness space $L^1([0,1],H^\kappa)$ with $\kappa>2+\gamma$.

The activation term can easily be bounded if $h_1$ grows sufficiently fast. That is, if 
\begin{align*}
    h_1(t)=1, \quad  \forall t\in [t_1+N_0^{-10\varepsilon},1], \qquad \text{and} \qquad  \|h'_1\|_{L^\infty} \lesssim  N_0^{10\varepsilon}
\end{align*}
then $h_1'(t)=0$ for $t\in [t_1+N_0^{-10\varepsilon},1]$ and, in particular, the activation term vanishes when $\|\theta_1(\cdot,t)\|_{H^\kappa}$ is big. In other words, we have to include the layer before its Sobolev norm has grown too much. Under this assumption, we have (see \eqref{P16.8.5})
\begin{align*}
    &\int_0^1 h_1'(t)\|\theta_1(\cdot,t)\|_{H^\kappa} \, dt
    =\int_{t_1}^{t_1+N_0^{-10\varepsilon}} h_1'(t)\|\theta_1(\cdot,t)\|_{H^\kappa} \, dt \\& \hspace{1cm} \lesssim \|h_1'\|_{L^\infty} \|\theta_1(\cdot,t_1+N_0^{-10\varepsilon})\|_{H^\kappa} N_0^{-10\varepsilon}
    \approx  \|\theta_1(\cdot,t_1)\|_{H^\kappa}\simeq N_1^{-1-\gamma+\varepsilon+(\kappa-1)(1-\frac{\varepsilon+\delta}{1+\gamma+\delta})}.
\end{align*}
In particular,  the exponent of $N_1$ is negative provided that $\kappa<2+\gamma+\delta$.

The cubic transport can be estimated by taking advantage of the fact that, although the support of $\theta_1$ elongates, it will always remain close enough to the origin, where the linear part of the velocity gives a good enough approximation of the velocity itself. Then, noticing that the largest frequency is associated with   $\theta_1$ (and hence the dominating term is the one with all the derivatives acting on $\theta_1$), 
\begin{align}
    &\int_0^1  \|(v^\gamma(\theta_0)-\overline{v}^\gamma(\theta_0))\cdot \nabla \theta_1\|_{H^\kappa}\, dt \nonumber \\
    & \hspace{.75cm}\lesssim |\text{supp } \theta_1|^{\frac{1}{2}} \sup_{t\in[0,1]}\|(v^\gamma(\theta_0)-\overline{v}^\gamma(\theta_0))(\cdot,t)\|_{L^\infty (\text{supp } \theta_1)} N_1^{-1-\gamma+\varepsilon}\Big(N_1^{1-\frac{\varepsilon+\delta}{1+\gamma+\delta}}N_0^\varepsilon\Big)^{1+\kappa}  \nonumber \\
    & \hspace{.75cm}\lesssim N_1^{-1+\frac{\varepsilon+\delta}{1+\gamma+\delta}} \Big(N_0^{-1-\gamma+\varepsilon} N_0^{\gamma} \sup_{t\in [0,1]}\sup_{x\in \text{supp }\theta_1}|N_0x|^3\Big) N_1^{-1-\gamma+\varepsilon}\Big(N_1^{1-\frac{\varepsilon+\delta}{1+\gamma+\delta}}N_0^\varepsilon\Big)^{1+\kappa}  \nonumber\\
    &\hspace{.75cm} \lesssim N_1^{-1+\frac{\varepsilon+\delta}{1+\gamma+\delta}} N_0^{2+\varepsilon}\Big( N_1^{-1+\frac{\varepsilon+\delta}{1+\gamma+\delta}}N_0^\varepsilon\Big)^3 N_1^{-1-\gamma+\varepsilon}\Big(N_1^{1-\frac{\varepsilon+\delta}{1+\gamma+\delta}}N_0^\varepsilon\Big)^{1+\kappa}
     \lesssim N_1^{-4-\gamma+\frac{2}{1+\gamma}+\kappa +O(\varepsilon)},\label{17.8.2}
\end{align}
where we have also used that the area of the support of $\theta_1$ is invariant as a consequence of the incompressibility and that the term of order 3 of the velocity consists heuristically in taking out $N_0^\gamma$ (because $v^\gamma$ is an operator of order gamma) and then the order three of the argument of the function, that is, $|N_0x|^3$.

The term corresponding to the transport of the previous layers can be estimated by splitting the norm into two parts: close to the support of $\theta_1$ (say, $\frac{\text{supp } \theta_0}{10}$)  and far from it. On the one hand, on the support of $\theta_1$  we  use that $\theta_0$ is constant (due to the choice of $\phi$ that we will specify) and, on the other hand,  outside  we use the decay of the velocity $|v^\gamma(\theta_1)(x)| \lesssim \frac{\|\theta_1\|_{L^1}}{|x|^{1+\gamma}}$ for $x$ sufficiently far from $\text{supp } \theta_1$, the fact that $|\nabla \theta_0(x)| \lesssim |x| \|D^2 \theta_0\|_{L^\infty}$. Since we want to locate the Sobolev norm into two pieces, it is convenient to deal  with local norms, so let us assume momentarily that  $\kappa \in \mathbb{N}$. Then
\begin{align}
    &\int_0^1\|v^\gamma(\theta_1)\cdot \nabla \theta_0\|_{H^\kappa}\, dt =   \int_0^1\|v^\gamma(\theta_1)\cdot \nabla \theta_0\|_{H^\kappa(\frac{\text{supp }\theta_0}{10})}\, dt+  \int_0^1\|v^\gamma(\theta_1)\cdot \nabla \theta_0\|_{H^\kappa(\mathbb{R}^2\setminus \frac{\text{supp }\theta_0}{10})}\, dt \nonumber\\
    &\hspace{.75cm}\lesssim 
    \left(\int_{\mathbb{R}^2\setminus \frac{\text{supp }\theta_0}{10}} \left(\Big(N_1^{1-\frac{\varepsilon+\delta}{1+\gamma+\delta}}N_0^\varepsilon\Big)^{\kappa}\frac{\|\theta_1\|_{L^1}}{|x|^{2+\gamma}} N_0^{-1-\gamma+\varepsilon}N_0^2|x| \right)^2 dx\right)^{\frac{1}{2}}
    \lesssim N_1^{-3-\gamma+\frac{1}{1+\gamma}+\kappa+O(\varepsilon)}. \label{P17.8.1}
\end{align}
 Moreover, the previous computations may be extended to  the case $\kappa \not \in \mathbb{N}$ via standard interpolation-type arguments. As a consequence, if $\kappa\in (2+\gamma,2+\gamma+\delta)$, where $\delta$ is small enough (in terms of $\gamma$), then the exponent of $N_1$ achieved in \eqref{P17.8.1} is negative. 

We wish to repeat the above process infinitely many times. As illustrated by the example associated with the two first layers of the construction,  one can see that the infinite sum of the forces will be bounded by a convergent infinite sum and hence the total force will remain in the desired well-posedness Sobolev space during the whole time, including the time of the blow up. However, we have  not yet estimated the self-interaction term in $F$. In fact, one can show that 
\begin{align*}
    \|v^\gamma(\theta_1)\cdot \nabla\theta_1\|_{H^\kappa}\geq N_1^{O(\varepsilon)},
\end{align*}
so the proposed strategy seems to fail dramatically.  Indeed, fixing this issue has been one of the most challenging aspects of our  construction.

At this point, we have no other option but to change the candidate for a solution. That is, we need to construct a better approximation of a solution of gSQG (without force) so that the errors accumulating in the force are summable. Nevertheless, we would like to change it in such a way that the mechanism of growth and the bounds of all the other terms of the force that are not coming from self-interactions remain the same. 

Motivated by the above discussion, we redefine $\theta_1$ as follows:
\begin{align*}
    \theta_1(x,t)=f_{1,0}(x,t)+f_{1,1}(x,t),
\end{align*}
where $f_{1,0}(x,t)$ is the previous $\theta_1$, that is, $f_{1,0}(x,t)$ solves \eqref{pdeaprox} with initial conditions \eqref{condicioninicial1}. On the other hand, $f_{1,1}$ solves the Cauchy problem: 
\begin{align*}
\left\{
\begin{aligned}
   & \partial_t f_{1,1}+\overline{v}^\gamma(\theta_0)\cdot \nabla f_{1,1}=-h_1v^\gamma(f_{1,0})\cdot \nabla f_{1,0}, \qquad \forall (x,t)\in \mathbb{R}^2\times [t_1,1],\\
   & f_{1,1}(x,t_1)=0. 
   \end{aligned}
   \right.
\end{align*}
With this new strategy, the problematic self-interaction term $v^\gamma(f_{1,0})\cdot \nabla f_{1,0}$ in the force  disappears. Of course, there is a price to pay: There will appear the terms 
\begin{align}
\label{fuerzaextra}
    v^\gamma(f_{1,0})\cdot \nabla f_{1,1}+
    v^\gamma(f_{1,1})\cdot \nabla f_{1,0}+
    v^\gamma(f_{1,1})\cdot \nabla f_{1,1}.
\end{align}
First, notice that $f_{1,1}$ will retain the most important aspects that were necessary to prove all the previous computations. Indeed, 
\begin{itemize}
    \item It is even with respect to the origin: Since $f_{1,0}$ is even with respect to the origin, then $v^\gamma(f_{1,0})$ and $\nabla f_{1,0}$ are odd, and hence their scalar product is even again. Moreover, $\overline{v}^\gamma(\theta_0)$ is odd for the same reason. Hence, a solution of a transport PDE with even initial condition, even force and odd velocity remains even for all time. It will therefore generate an odd velocity.
    \item Since we start with a zero initial condition, the force is supported in $\text{supp }f_{1,0}$ and $f_{1,0}$, $f_{1,1}$ are both transported by the same velocity, then
    \begin{align*}
        \text{supp }f_{1,1}\subset 
        \text{supp }f_{1,0}. 
    \end{align*}
    Hence we can use the same estimate on the support when estimating  the force terms.
    \item It has the same anisotropy as $f_{1,0}$: the large derivative will be in the same direction as it was for $f_{1,0}$ and as large as the one of $f_{1,0}$. This is useful in the computation of the velocity $\overline{v}^\gamma(\theta_1)$ and to obtain some extra cancellations in the force terms.
\end{itemize}
But now, if we want the force to be smaller, the terms in \eqref{fuerzaextra} must be smaller than the one that we canceled. As we just said, the frequencies of $f_{1,1}$ are comparable to the ones of $f_{1,0}$, so
\begin{align*}
    \frac{\|f_{1,0}\|_{H^\kappa}}{\|f_{1,0}\|_{L^2}}\approx \frac{\|f_{1,1}\|_{H^\kappa}}{\|f_{1,1}\|_{L^2}}. 
\end{align*}
Since the supports are comparable, the gain will instead come from  the $L^\infty$ norm, or as we shall call it in the rest of the paper, the amplitude. As a result,
\begin{align*}
    \|f_{1,1}\|_{H^\kappa} \ll \|f_{1,0}\|_{H^\kappa}.
\end{align*}
 To check this, one has to look at the $L^\infty$ norm of the force term and take into account the orthogonality cancellation. That is, if we define
\begin{align*}
    f_{1,0}^{(1,0)}(x,t)=N_1^{-1-\gamma+\varepsilon}\phi'\Big(N_1^{1-\frac{\varepsilon+\delta}{1+\gamma+\delta}}x_1\Big)\phi\Big(N_1^{1-\frac{\varepsilon+\delta}{1+\gamma+\delta}}\left(x_2-N_0^\varepsilon (t-t_1)x_1\right)\Big),\\
    f_{1,0}^{(0,1)}(x,t)=N_1^{-1-\gamma+\varepsilon}\phi\Big(N_1^{1-\frac{\varepsilon+\delta}{1+\gamma+\delta}}x_1\Big)\phi'\Big(N_1^{1-\frac{\varepsilon+\delta}{1+\gamma+\delta}}\left(x_2-N_0^\varepsilon (t-t_1)x_1\right)\Big),
\end{align*}
then
\begin{align*}
    v^\gamma(f_{1,0})=  
    N_1^{1-\frac{\varepsilon+\delta}{1+\gamma+\delta}}\left[
\begin{pmatrix}
    1\\ N_0^\varepsilon(t-t_1)
\end{pmatrix}(-\Delta)^{-\frac{1-\gamma}{2}} f_{1,0}^{(0,1)}-\begin{pmatrix}
    0\\1
\end{pmatrix}(-\Delta)^{-\frac{1-\gamma}{2}} f_{1,0}^{(1,0)}
    \right],
\end{align*}
and
\begin{align*}
\nabla f_{1,0}=
    N_1^{1-\frac{\varepsilon+\delta}{1+\gamma+\delta}} \left(
\begin{pmatrix}
    -N_0^\varepsilon(t-t_1)\\1
\end{pmatrix}f_{1,0}^{(0,1)}
+\begin{pmatrix}
    1\\0
\end{pmatrix} f_{1,0}^{(1,0)}
    \right),
\end{align*}
so that, by orthogonality, the terms carrying two derivatives of $\phi$ in the large directions cancel and we are left with
\begin{align*}
     |v^\gamma(f_{1,0})\cdot \nabla f_{1,0}|&=  \left|N_1^{2\left(1-\frac{\varepsilon+\delta}{1+\gamma+\delta}\right)} \left((-\Delta)^{-\frac{1-\gamma}{2}} f_{1,0}^{(0,1)} f_{1,0}^{(1,0)}
     -(-\Delta)^{-\frac{1-\gamma}{2}} f_{1,0}^{(1,0)} f_{1,0}^{(0,1)}
\right)\right|\\
&\lesssim 
N_1^{2\left(1-\frac{\varepsilon+\delta}{1+\gamma+\delta}\right)} N_1^{2(-1-\gamma+\varepsilon)} N_1^{-1+\gamma}=N_1^{-1-\gamma+\varepsilon}N_1^{\varepsilon-2\frac{\varepsilon+\delta}{1+\gamma+\delta}},
\end{align*}
and, taking $\delta$ small enough so that $1+\gamma+\delta<2$, then 
\begin{align*}
   \varepsilon-2\frac{\varepsilon+\delta}{1+\gamma+\delta}<0 
\end{align*}
is the gain that we obtain in the exponent of the $L^\infty$ norm and hence also in the exponent of the force. It may happen that this is not enough, but we can iterate this process a finite number of times (depending only on $\gamma$) until we obtain a negative exponent in the self-interaction term of the force. 

We are now ready to present the complete infinite-layers scheme: We construct a solution $\theta$ that solves \eqref{ecuacion} with a force $F$. The solution will be a sum of layers 
\begin{align}
    \theta(x,t)=\sum_{n=0}^\infty h_n(t) \theta_n(x,t),
\end{align}
where $h_n$ are given by \eqref{P16.8.4}. 
The sequence of times is defined by
\begin{align}
    \label{tiempos}
    t_0=0,\qquad  t_n=1- N_{n-1}^{-\mu}, \qquad  \forall n\in \mathbb{N},
\end{align}
 where $\mu>0$ is a small parameter to be fixed.

Each of the layers is formed by a principal part $f_{n,0}$ and smaller corrections $f_{n,i}$ for $1\leq i\leq I_\gamma$. That is, 
\begin{align}
    \theta_n(x,t)=\sum_{i=0}^{I_\gamma} f_{n,i}(x,t), \qquad n \in \mathbb{N}, 
\end{align}
with $I_\gamma>0$ to be specified. The zeroth layer is stationary. More specifically, 
\begin{align*}
    \theta_0(x,t)=N_0^{-1-\gamma+\varepsilon}\phi(N_0x_1)\phi(x_2), \qquad h_0(t)=1.
\end{align*}
Consider 
\begin{align*}
    \Theta_n=\sum_{j=0}^n h_j \theta_j.
\end{align*}
For $n \in \mathbb{N}$, the principal part of the layers  verify
\begin{align}
\label{pdeprincipal}
\left\{
\begin{aligned}
    &\partial_t f_{n,0}+\overline{v}^\gamma \left(\Theta_{n-1}\right) \cdot \nabla f_{n,0}=0, \qquad  \forall t\in [t_n,1],\\
 &f_{n,0}(x,t_n)=N_n^{-1-\gamma+\varepsilon}\phi(r_n(t_n)\textbf{n}(\alpha_n(t_n))\cdot x)\phi(s_n(t_n)\textbf{n}(\beta_n(t_n))\cdot x). 
\end{aligned}
\right.
\end{align}
Here, $\overline{v}^\gamma \left(\Theta_{n-1}\right)$ is defined  as the linear part of the spatial Taylor expansion of $v^\gamma \left(\Theta_{n-1}\right)$.  Accordingly,  there exist functions $r_n, s_n, \alpha_n, \beta_n$ such that
\begin{align*}
    f_{n,0}(x,t)=N_n^{-1-\gamma+\varepsilon}\phi(r_n(t)\textbf{n}(\alpha_n(t))\cdot x)\phi(s_n(t)\textbf{n}(\beta_n(t))\cdot x).
\end{align*}
As we will see in Lemma \ref{dinamica}, one can prove  that \eqref{pdeprincipal} is equivalent to a system of ODEs for $r_n, s_n, \alpha_n, \beta_n$, where the coefficients of those ODEs will be given in terms of $r_j, s_j, \alpha_j, \beta_j$ for $j\leq n-1$. A precise estimation of the solutions of those ODEs is crucial to describe the dynamics of the solution. 

The functions $r_n$, $s_n$ are called principal frequencies and $\alpha_n$, $\beta_n$ are called principal directions. Notice that $\textbf{n}(\alpha_n)$ and $\textbf{n}(\beta_n)$ will be perpendicular to two of the lines of the parallelograms that confine the supports of the layers. As illustrated in Figure \ref{fig:vectors} below,  the green vector remains perpendicular to the vertical lines and does not change its size (since the vertical lines do not change their relative distance nor their directions). On the other hand, the orange vector grows (since the long sides of the quadrangle tend to approach each other) and rotates towards the green one (since the long sides tend to become vertical).

\begin{figure}[h]
\centering
\begin{tikzpicture}[scale=0.6, transform shape, line width=0.6pt]
  \filldraw[fill=white] (0,0) rectangle (7,15.44);
  \filldraw[fill=white] (2.40,6.62) rectangle (4.60,8.82);   
  \draw[vec, vecverde] (3.5,7.72) -- (5.5,7.72);
  \draw[vec, vecrojo] (3.5,7.72) -- (3.5,5.72);

  \node[lbl, anchor=north west] at (0.15,15.28) {$\text{supp }\theta_0(\cdot,t_1)$};
  \node[lbl, anchor=south]      at (3.5,8.9)    {$\text{supp }\theta_1(\cdot,t_1)$};
  \node[lbl, vecverde, anchor=south] at (5,7.90)
        {$r_1(t_1)\mathbf{n}(\alpha_1(t_1))$};
  \node[lbl, vecrojo, anchor=north] at (4.20,5.50)
        {$s_1(t_1)\mathbf{n}(\beta_1(t_1))$};
\end{tikzpicture} 
\hspace{20mm}
\begin{tikzpicture}[scale=0.6, transform shape, line width=0.6pt]
  \filldraw[fill=white] (0,0) rectangle (7,15.44);
  \filldraw[fill=white] (4.60,12.40) -- (4.60,11.08) -- (2.40,3.04) -- (2.40,4.36) -- cycle;
  \draw[vec, vecverde] (3.5,7.72) -- (5.5,7.72);
  \draw[vec, vecrojo] (3.5,7.72) -- (6.394,6.928);

  \node[lbl, anchor=north west] at (0.15,15.28) {$\text{supp }\theta_0(\cdot,t_2)$};
  \node[lbl, anchor=east]       at (3.75,10.5)  {$\text{supp }\theta_1(\cdot,t_2)$};
  \node[lbl, vecverde, anchor=south] at (5.5,7.90)
        {$r_1(t_2)\mathbf{n}(\alpha_1(t_2))$};
  \node[lbl, vecrojo, anchor=north] at (5.20,6.50)
        {$s_1(t_2)\mathbf{n}(\beta_1(t_2))$};
\end{tikzpicture}
\caption{Evolution of the vectors $r_1\textbf{n}(\alpha_1)$ and $s_1\textbf{n}(\beta_1)$.}
\label{fig:vectors}
\end{figure}
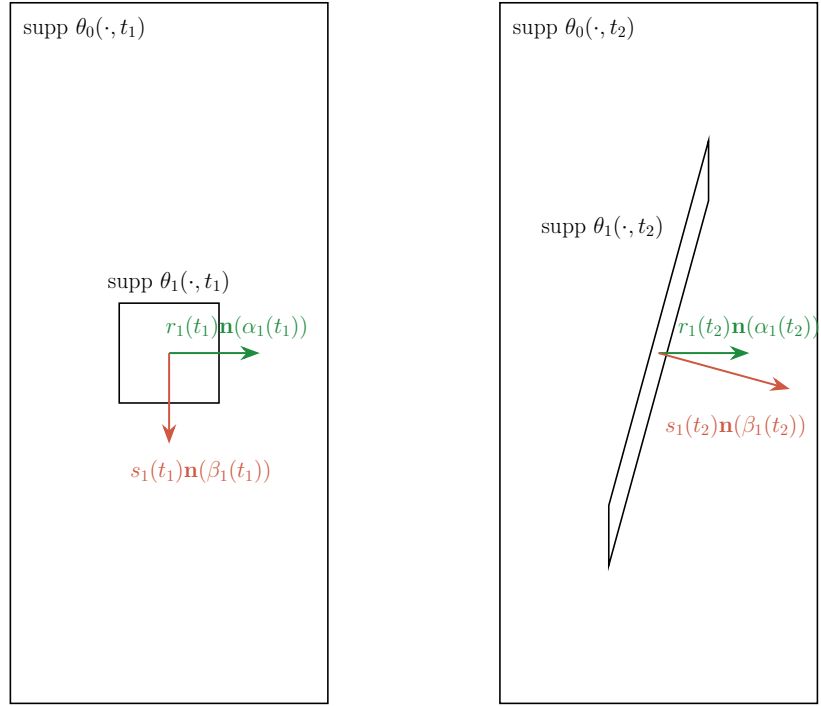 

The initial conditions for the principal frequencies and directions will be chosen as follows
 \begin{align*}
            r_{n}(t_{n})=s_n(t_n)=N_{n}^{1-\frac{\varepsilon+\delta}{1+\gamma+\delta}}, \qquad \alpha_{n}(t_{n})=\beta_{n-1}(t_{n}),  \qquad \beta_{n}(t_{n})=\beta_{n-1}(t_{n})-\frac{\pi}{2}.  
    \end{align*}
    Graphically, one can see this choice in  Figure \ref{fig:ic}.

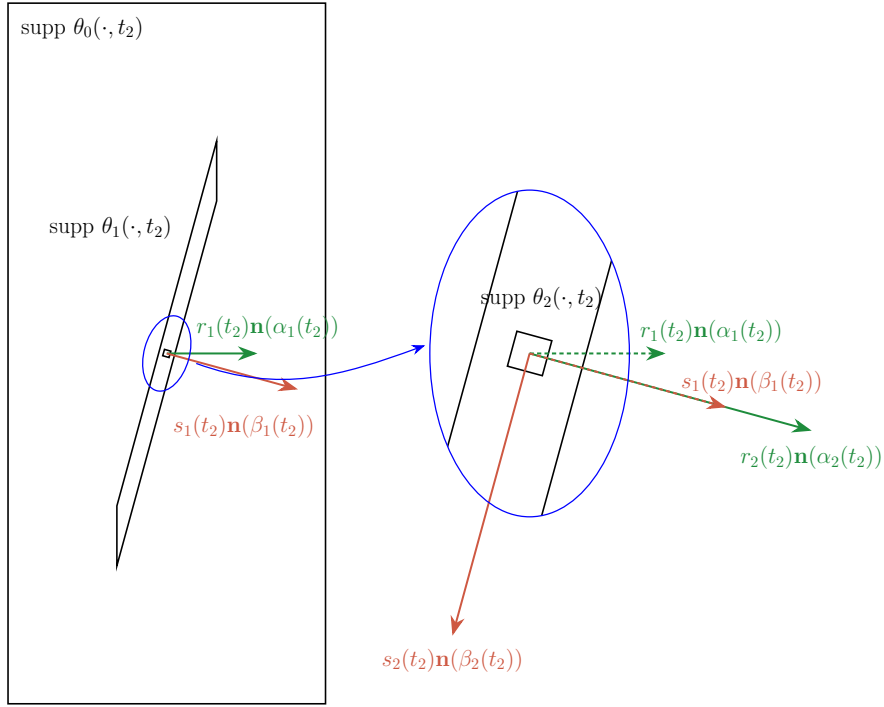
\begin{figure}[h]
\centering
\begin{tikzpicture}[scale=0.6, transform shape, line width=0.6pt]
  \filldraw[fill=white] (0,0) rectangle (7,15.44);

  \filldraw[fill=white] (4.60,12.40) -- (4.60,11.08) -- (2.40,3.04) -- (2.40,4.36) -- cycle;

  \filldraw[fill=white] (3.45,7.81) -- (3.59,7.77) -- (3.55,7.63) -- (3.41,7.67) -- cycle;

  \draw[vec, vecverde] (3.5,7.72) -- (5.5,7.72);
  \draw[vec, vecrojo]  (3.5,7.72) -- (6.394,6.928);

  \draw[blue,line width=0.48pt,rotate around={-15.3:(3.5,7.72)}]
        (3.5,7.72) ellipse (0.5 and 0.85);

  \draw[blue,line width=0.48pt,-{Stealth}] (4.15,7.50) to[out=-20,in=200] (9.20,7.90);

  \begin{scope}
    \clip (11.5,7.72) ellipse (2.2 and 3.6);
    \fill[white] (11.54,12.41) -- (13.84,11.71) -- (11.46,3.03) -- (9.16,3.73) -- cycle;
    \draw (11.54,12.41) -- (9.16,3.73);
    \draw (13.84,11.71) -- (11.46,3.03);
  \end{scope}

  \filldraw[fill=white] (11.22,8.21) -- (11.99,8.00) -- (11.78,7.23) -- (11.01,7.44) -- cycle;

  \draw[vec, vecverde] (11.5,7.72) -- (17.721,6.018);
  \draw[vec, vecrojo]  (11.5,7.72) -- (9.798,1.499);
  \draw[vec, vecverde, dash pattern=on 1.6pt off 1.2pt] (11.5,7.72) -- (14.500,7.720);
  \draw[vec, vecrojo,  dash pattern=on 1.6pt off 1.2pt] (11.5,7.72) -- (15.841,6.533);

  \draw[blue,line width=0.48pt] (11.5,7.72) ellipse (2.2 and 3.6);

  \node[lbl, anchor=north west] at (0.15,15.28) {$\text{supp }\theta_0(\cdot,t_2)$};
  \node[lbl, anchor=east]       at (3.75,10.5)  {$\text{supp }\theta_1(\cdot,t_2)$};
  \node[lbl, anchor=south]      at (11.77,8.55) {$\text{supp }\theta_2(\cdot,t_2)$};
  \node[lbl, vecverde, anchor=south] at (5.73,7.90) {$r_1(t_2)\mathbf{n}(\alpha_1(t_2))$};
  \node[lbl, vecrojo,  anchor=north] at (5.20,6.50) {$s_1(t_2)\mathbf{n}(\beta_1(t_2))$};
  \node[lbl, vecverde, anchor=west]  at (13.8,8.2) {$r_1(t_2)\mathbf{n}(\alpha_1(t_2))$};
  \node[lbl, vecrojo,  anchor=south] at (16.4,6.7) {$s_1(t_2)\mathbf{n}(\beta_1(t_2))$};
  \node[lbl, vecverde, anchor=north] at (17.72,5.82) {$r_2(t_2)\mathbf{n}(\alpha_2(t_2))$};
  \node[lbl, vecrojo,  anchor=north] at (9.80,1.35)  {$s_2(t_2)\mathbf{n}(\beta_2(t_2))$};
\end{tikzpicture}
\caption{Initial conditions of $r_2 \textbf{n}(\alpha_2)$ and $s_2 \textbf{n}(\beta_2)$ and comparison with the saturated vectors $r_1 \textbf{n}(\alpha_1)$ and $s_1 \textbf{n}(\beta_1)$. }
\label{fig:ic}
\end{figure}

At this point, the reader may detect several things: Firstly, the length of the vectors with subindex $2$ is larger  than the corresponding  one with subindex $1$. Also, the vectors with subindex $2$ start perpendicular to each other and $r_2 \textbf{n}(\alpha_2)$ starts overlapping $s_1 \textbf{n}(\beta_1)$. As time evolves we will see how the orange vectors (that is, $s_n\textbf{n}(\beta_n)$) accumulate near the horizontal axis and grow. More specifically, at the time of saturation $s_n(t_{n+1})\approx N_n$ and the large principal frequency of each layer $N_n$ grows superexponentially in $n\nearrow\infty$.  

The orange vectors measure the approach and the verticality of the sides of the supports that are initially almost horizontal. In Figure \ref{fig:acumulado} one can see how they accumulate close to the angle $2\pi$ and how they will grow in length.

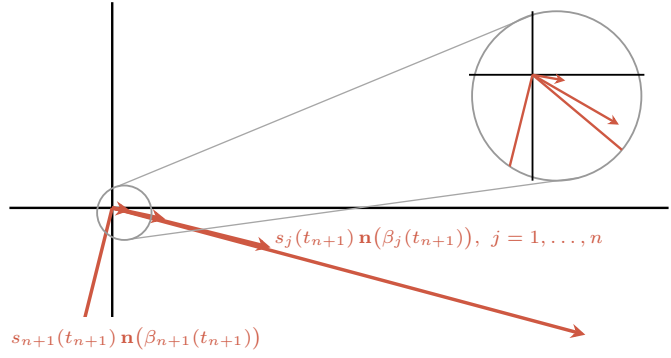
\begin{figure}[h]
    \centering
\begin{tikzpicture}[scale=0.8]
 \begin{scope}
  \clip (-1.8,-2.5) rectangle (9.3,3.6);
  \draw[black,line width=0.9pt] (-1.7,0) -- (9.2,0);
  \draw[black,line width=0.9pt] (0,-2.5) -- (0,3.4);
  \begin{scope}[vecrojo,line width=1.3pt,-{Stealth[length=1.9mm,width=1.9mm]}]
    \draw (0,0) -- ( -8.00:0.3);
    \draw (0,0) -- (-12.00:0.9);
    \draw (0,0) -- (-14.00:2.7);
    \draw (0,0) -- (-15.00:8.1);
  \end{scope}
  \draw[vecrojo,line width=1.3pt] (0,0) -- (256:3.0);
  \coordinate (L)  at (0.20,-0.08);
  \coordinate (C)  at (7.35,1.85);
  \coordinate (Oz) at (6.95,2.20);
  \draw[gray!70,line width=0.5pt] (0.083,0.354) -- (6.985,3.202);
  \draw[gray!70,line width=0.5pt] (0.317,-0.514) -- (7.715,0.498);
  \draw[gray!80,line width=0.7pt] (L) circle[radius=0.45];
  \draw[gray!80,line width=0.7pt] (C) circle[radius=1.4];
  \draw[black,line width=0.7pt] (5.90,2.20) -- (8.80,2.20);   
  \draw[black,line width=0.7pt] (6.95,0.45) -- (6.95,3.25);   
  \begin{scope}[vecrojo,line width=0.9pt,-{Stealth[length=1.4mm,width=1.4mm]}]
    \draw (Oz) -- ($(Oz)+(-10:0.55)$);
    \draw (Oz) -- ($(Oz)+(-30:1.65)$);
  \end{scope}
  \begin{scope}
    \clip (C) circle[radius=1.4];
    \draw[vecrojo,line width=0.9pt] (Oz) -- ($(Oz)+(-40:3)$);
    \draw[vecrojo,line width=0.9pt] (Oz) -- ($(Oz)+(256:3)$);
  \end{scope}
  \node[vecrojo,font=\fontsize{7}{8.4}\selectfont]
        at (5.4,-0.5) {$s_j(t_{n+1})\,\mathbf{n}\bigl(\beta_j(t_{n+1})\bigr),\ j=1,\dots,n$};
  \node[vecrojo, fill=white ,font=\fontsize{7}{8.4}\selectfont]
        at (0.4,-2.15) {$s_{n+1}(t_{n+1})\,\mathbf{n}\bigl(\beta_{n+1}(t_{n+1})\bigr)$};
 \end{scope}
\end{tikzpicture}
    \caption{Accumulation of the growing sequence of vectors $s_j \textbf{n}(\beta_j)$ close to the angle $2\pi$.}
    \label{fig:acumulado}
\end{figure}

The perturbations will satisfy, for each $n \in \mathbb{N}$ and $i= 1, \ldots, I_\gamma$, 
\begin{align*}
    \partial_t f_{n,i}+\overline{v}^\gamma \left(\Theta_{n-1}\right) \cdot \nabla f_{n,i}=-h_n\sum_{(j_1,j_2)\in\mathcal{J}_{i-1}}v^\gamma(f_{n,j_1})\cdot \nabla f_{n,j_2},
\end{align*} 
where
\begin{align*}
    \mathcal{J}_i=\{(j_1,j_2)\in \mathbb{N}_0^2: \max\{j_1,j_2\}=i\}.
\end{align*}
All this implies, for any $n\in\mathbb{N}$,
\begin{align*}
    \partial_t \Theta_n+ v^\gamma (\Theta_n) \cdot \nabla \Theta_n= \mathcal{F}_n, \qquad \forall t\in [0,t_{n+1}),
\end{align*}
where 
\begin{align*}
    \mathcal{F}_n=\sum_{j=0}^n F_j,
\end{align*}
and, for any $n\in \mathbb{N}$,
\begin{align*}
    F_n=h_n\left[(v^\gamma-\overline{v}^\gamma)\left(\sum_{j=0}^{n-1}h_j\theta_j\right)\cdot\nabla \theta_n+v^\gamma(\theta_n)\cdot \nabla\left(\sum_{j=0}^{n-1}h_j\theta_j\right)+h_n\sum_{(j_1,j_2)\in \mathcal{J}_{I_\gamma}}v^\gamma(f_{n,j_1})\cdot\nabla f_{n,j_2}\right]+h_{n}'\theta_n.
\end{align*}
The zeroth force is just defined by the self interaction term
\begin{align*}
    F_0=v^\gamma(\theta_0)\cdot\nabla \theta_0.
\end{align*}
Notice that, since $t_n\nearrow 1$, if we define (in the appropriate spaces) the limits
\begin{align}
\label{limite}
    \theta=\lim_{n\rightarrow \infty} \Theta_n,\qquad F=\lim_{n\rightarrow \infty} \mathcal{F}_n,
\end{align}
then 
\begin{align*}
    \partial_t \theta+v^\gamma(\theta)\cdot \nabla \theta=F, \qquad  \forall t \in [0,1),
\end{align*}
and we will have, as desired, a solution of the forced generalized SQG equation that blows up at time $t=1$. The most important features of this construction are:
\begin{itemize}
    \item The solution and the force are smooth in space and time before the blow up, that is, 
    \begin{align*}
        \theta ,F \in C^\infty ([0,T], C^\infty_{c}(\mathbb{R}^2)), \qquad  \forall T\in(0,1).
    \end{align*}
    This can be easily checked in the previous construction, since at any time $T < 1$ all a priori infinite sums are indeed finite sums. Then, as all the functions involved are smooth, a finite sum of them will also be smooth. In this setting, it is clear that  the limits \eqref{limite} are well-defined in a pointwise sense for any $(x,t)\in\mathbb{R}^2\times [0,1)$.
    \item The force $F$ starts and remains in a well-posedness space, more precisely, there exists $\kappa_0>2+\gamma$ such that 
    \begin{align*}
         F\in L^1([0,1],H^\kappa(\mathbb{R}^2)), \qquad  \forall \kappa\in [2+\gamma,\kappa_0].
    \end{align*}
    This is achieved by estimating the Sobolev norm of $F$ by the infinite sum of the Sobolev estimates obtained for each $F_n$. Although we have only outlined how to compute the estimates for  $F_1$, the argument is analogous for the rest of the $F_n$'s, and then we arrive at 
    \begin{align*}
       \int_0^1 \|F_n(\cdot,t)\|_{H^\kappa} \, dt\leq N_n^{-C(\gamma)},
    \end{align*}
    where $C(\gamma)>0$ depends only on $\gamma$. Clearly,  the sum of these integrals is  finite. Then, the limit \eqref{limite} makes sense in the whole time interval in $L^1_tH_x^\kappa$ as
    \begin{align*}
        \lim_{n\rightarrow \infty}\int_0^1\|F (\cdot, t)-\mathcal{F}_n(\cdot, t)\|_{H^\kappa} \, dt=0.
    \end{align*}
    \item The solution $\theta$ escapes from the well-posedness space at time $t=1$, that is,  
    \begin{align*}
         \int_0^1\|\theta(\cdot, t)\|_{H^\kappa} \, dt=\infty \qquad  \text{for} \qquad  \kappa\in [2+\gamma,\kappa_0].
    \end{align*}
    This is due to the fact that, at the saturation time $t=t_{n+1}$, $s_n(t_{n+1})\approx N_n$ and then one can show that
    \begin{align*}
        \|\theta_{n}(\cdot,t_{n+1})\|_{H^\kappa}\approx\|f_{n,0}(\cdot,t_{n+1})\|_{H^\kappa}\approx N_{n}^{-2-\gamma+\varepsilon+\frac{\varepsilon+\delta}{1+\gamma+\delta}+\kappa}. 
    \end{align*}
    Nevertheless, using this last estimate one can see that the solution remains in some Sobolev space. In particular, 
    \begin{align*}
        \sup_{0\leq t\leq 1}\|\theta(\cdot,t)\|_{H^\kappa}<\infty \qquad  \text{for all} \qquad  \kappa\in [0,2+\gamma-2\varepsilon],
    \end{align*}
    and  the convergence \eqref{limite} makes sense in the whole time interval in the $L^1_tH_x^\kappa$ sense (for $\kappa\in[0,2+\gamma-2\varepsilon]$):
     \begin{align*}
        \lim_{n\rightarrow \infty}\int_0^1\|\theta-\Theta_n\|_{H^\kappa} \, dt=0.
    \end{align*}
    \item Lastly, one can see in the above heuristics why SQG (that is, $\gamma=0$) is out of the scope of our approach. Looking at the bounds of the forces (see \eqref{17.8.2} and \eqref{P17.8.1}) one can see that we cannot make them small for $\gamma=0$ and $\kappa>2$ (which corresponds to the  well-posedness regime for SQG). The same comment also applies to $\gamma\in (-1,0)$. In other words, the cubic remainder of the transport affecting the new layers, as well as the transport of the previous layers, is no longer negligible in this regime.
\end{itemize}

        \section{The solution}\label{SectionSolution}
The whole construction takes place inside a class of functions closed under the operations we shall iterate to construct the corrections (see Lemmas \ref{nuevaperturbacion}, \ref{nuevaperturbacion2}). A member of the class is not a single function but a family $\{f^{(j,k)}\}$ indexed by the number of derivatives taken in each of two distinguished directions: $\textbf{n}(\alpha)$ and $\textbf{n}(\beta)$, the normals to the two pairs of sides of the parallelogram supporting $f^{(j,k)}$. Notice that differentiating in the direction $\textbf{n}(\alpha)$ costs a factor $r$ and differentiating in the direction $\textbf{n}(\beta)$ costs a factor $s$.

\begin{definicion} 
\label{sod}
Let $t_\ast\in [0,1)$ and let $\alpha,\beta,r,s\in C^\infty([t_\ast,1])$ with $r,s>0$ and 
\begin{align}
\label{noparalelos}
    \sin(\alpha(t)-\beta(t))\neq 0, \hspace{3mm} \forall t\in [t_\ast,1].
\end{align}
    We say that $\{f^{(j,k)}\}\subset L^\infty ([t_\ast,1],C^\infty_c(\mathbb{R}^2))$ is a sequence of derivatives of non-trivial amplitude $\mathcal{C}(t)\in [0,\infty)$, principal directions $(\alpha(t),\beta(t))\in \mathbb{R}^2$ and principal frequencies $(r(t),s(t))\in (0,\infty)^2$ if the following properties hold  for every $t\in[t_\ast,1]$:
    \begin{enumerate}[label=(P\arabic*), ref=(P\arabic*)]
        \item\label{paridad} If $j+k$ is even (odd), then $f^{(j,k)}(\cdot,t)$ is even (odd) with respect to the origin.
        \item\label{amplitud} $\|f^{(j,k)}(\cdot,t)\|_{L^\infty}\lesssim_{j,k,\gamma}\mathcal{C}(t)$,  $\forall (j,k)\in \mathbb{N}_0^2$.
        \item\label{soporte} $\text{supp} f^{(j,k)}(\cdot,t) \subset \mathcal{S}(\alpha(t),\beta (t),r(t),s(t))$, $\forall (j,k)\in \mathbb{N}_0^2$, where
        \begin{align*}
            \mathcal{S}(\alpha(t),\beta (t),r(t),s(t)):=\{x\in \mathbb{R}^2:|r(t)\textbf{n}(\alpha(t))\cdot x|\leq 1, |s(t)\textbf{n}(\beta(t))\cdot x|\leq 1 \}.
        \end{align*}
        \item\label{derivadas} $\nabla f^{(j,k)}(x,t)=f^{(j+1,k)}(x,t)r(t)\textbf{n}(\alpha(t))+f^{(j,k+1)}(x,t)s(t)\textbf{n}(\beta(t)),  \qquad \forall (j,k)\in \mathbb{N}_0^2$.
    \end{enumerate}
\end{definicion}
\begin{rmk}
    Since principal amplitudes, principal directions and frequencies depend just on time, in what follows, when there is no danger of confusion, we may omit time dependence in the notation. 
\end{rmk}
\begin{rmk}
    The $t_\ast$ in the time interval refers to the initial time of the layer. In the construction this role will be played by the sequence $t_n$ indicated in \eqref{tiempos}.
\end{rmk}
\begin{rmk}
    The terminology sequence of derivatives is justified by Property \ref{derivadas}.  
\end{rmk}

\begin{rmk}
    An important fact to keep in mind for the future construction is that, when considering several layers indexed in $n$, there will exist $C_{j,k,\gamma}>0$ depending on $j,k,\gamma$ (which will traduce essentially in a dependence on $\gamma$) such that 
    \begin{align*}
        \|f^{(j,k)}_n(\cdot,t)\|_{L^\infty}\leq C_{j,k,\gamma}\mathcal{C}_n(t),
    \end{align*}
    but the dependence on $n$ will be entirely encoded in $\mathcal{C}_n$.
    
\end{rmk}

\begin{definicion}\label{DefSD}
 We say that $f\in L^\infty ([t_\ast,1],C^\infty_c(\mathbb{R}^2))$ generates a sequence of derivatives of amplitude $\mathcal{C}$, principal directions $(\alpha,\beta)$ and principal frequencies $(r,s)$ if there exists a sequence of derivatives $\{f^{(j,k)}\}\subset L^\infty ([t_\ast,1],C^\infty_c(\mathbb{R}^2))$ with amplitude $\mathcal{C}$, principal directions $(\alpha,\beta)$ and principal frequencies $(r,s)$ such that 
     $f^{(0,0)} = f.$
\end{definicion}
\begin{rmk}
    Notice that, since $f$ is a smooth function in the spatial variables and, by \eqref{noparalelos}, $\{\textbf{n}(\alpha),\textbf{n}(\beta)\}$ is a basis of $\mathbb{R}^2$, Property \ref{derivadas} determines $f^{(j+1,k)}$ and $f^{(j,k+1)}$ from $f^{(j,k)}$. Consequently, the sequence of derivatives generated by $f$ is unique.
\end{rmk}
It is convenient to normalize a building block so that the second direction is horizontal and the second frequency is one; this reduces every computation of the velocity operator to a fixed geometry, and only the ratio $r/s$ and the angle $\alpha-\beta$ survive as parameters.
\begin{definicion}
    If $\{f^{(j,k)}\}\subset L^\infty ([t_\ast,1],C^\infty_c(\mathbb{R}^2))$ is a sequence of derivatives of amplitude $\mathcal{C}$, principal directions $(\alpha, \beta)$ and principal frequencies $(r,s)$, we define its aligned and rescaled sequence of derivatives as $\{\tilde{f}^{(j,k)}\}$, where
    \begin{align}\label{DefinitionAlignedRescaled}
        \tilde{f}^{(j,k)}(x,t)=R_{-\beta}f^{(j,k)} \left( \frac{x}{s},t\right)=f^{(j,k)} \left( \frac{R_{\beta}x}{s},t\right).
    \end{align}
\end{definicion}
\begin{rmk}
    Notice that the aligned and rescaled sequence of derivatives is indeed a sequence of derivatives of amplitude $\mathcal{C}$, principal directions $(\alpha-\beta,0)$ and principal frequencies $\left(\frac{r}{s},1\right)$. Properties \ref{paridad}, \ref{amplitud} and \ref{soporte} can be easily checked. Concerning Property \ref{derivadas}, we have
    \begin{align}
    \label{derivadasaligned}
    \begin{split}
           \nabla \tilde{f}^{(j,k)}(x,t)&= \nabla \left( f^{(j,k)}\left(\frac{R_{\beta}x}{s},t\right)\right)=\frac{1}{s}\left(\begin{matrix}
           \cos \beta & \sin \beta\\
           -\sin \beta & \cos \beta
       \end{matrix}\right) (\nabla f^{(j,k)})\left(\frac{R_\beta x}{s},t\right)\\
       &=\frac{1}{s}\left(\begin{matrix}
           \cos \beta & \sin \beta\\
           -\sin \beta & \cos \beta
       \end{matrix}\right) \left(f^{(j+1,k)}\left(\frac{R_\beta x}{s},t\right)r\textbf{n}(\alpha)+f^{(j,k+1)}\left(\frac{R_\beta x}{s},t\right)s\textbf{n}(\beta)\right)\\
       &=\tilde{f}^{(j+1,k)}(x,t) \frac{r}{s}\textbf{n}(\alpha-\beta)+\tilde{f}^{(j,k+1)}(x,t)\textbf{n}(0).
           \end{split}
    \end{align}
\end{rmk}
The next two lemmas are the closure properties announced above. Lemma
\ref{nuevaperturbacion} states that a product $v^\gamma(f_{i_1})\cdot\nabla f_{i_2}$ of two
members of the class with common principal directions and frequencies is again a member of the class,
with smaller amplitude $\mathcal{C}_{i_1}\mathcal{C}_{i_2}\,r s^{\gamma}|\sin(\alpha-\beta)|$. Lemma
\ref{nuevaperturbacion2} states that solving $\partial_t f_2+v\cdot\nabla f_2=-f_1$ with zero
initial datum along the flow of the linear velocity that transports $f_1$ again produces a member of the class,
with the same principal directions and frequencies as $f_1$ and amplitude $\int_{t_\ast}^{t}\mathcal{C}_1$.
Both lemmas together say that each correction $f_{n,i}$ (with $i=1,...,I_\gamma$), defined by \eqref{defpert} is again
a building block sharing the geometry of $f_{n,0}$ --same support, same anisotropy-- and
differing from it only in being smaller. The gain in the amplitudes, applied repeatedly, is what allows $f_{n,0}$ to be the leading term of the layer and what makes the self-interaction errors small.

\begin{lma}
\label{nuevaperturbacion}
    For any $i=0,\ldots,I$, let $\{f^{(j,k)}_i\}$ be a sequence of derivatives of amplitudes $\mathcal{C}_i$ and common principal directions $(\alpha,\beta)$ and principal frequencies $(r,s)$ such that $r,s\geq 1$. Let $\mathcal{J} \subset \{0,\ldots,I\}^2$. Consider 
    \begin{align}
    \label{nuevaperturbacion00}
        g^{(0,0)}=\sum_{(i_1,i_2)\in \mathcal{J}} v^\gamma (f_{i_1}^{(0,0)})\cdot \nabla f_{i_2}^{(0,0)}, 
    \end{align}
    and for any $(j,k)\in \mathbb{N}_0^2\setminus\{(0,0)\}$,
    \begin{align}
     \begin{split}
     \label{recursividad}
       g^{(j,k)}  =&\sum_{(i_1,i_2)\in \mathcal{J}}\mathcal{C}_{i_1}\mathcal{C}_{i_2}rs^\gamma \sin(\alpha-\beta)\left[\frac{s^{1-\gamma}}{\mathcal{C}_{i_1}\mathcal{C}_{i_2}}\sum_{l=0}^j\sum_{m=0}^k \left(\begin{matrix}
            j\\ l
       \end{matrix}\right)\left(\begin{matrix}
            k \\ m
       \end{matrix}\right)\left(
       f_{i_2}^{(l,m+1)} (-\Delta)^{-\frac{1-\gamma}{2}} f_{i_1}^{(j-l+1,k-m)}
       \right.\right.\\
       &\left.\left. -f_{i_2}^{(l+1,m)} (-\Delta)^{-\frac{1-\gamma}{2}} f_{i_1}^{(j-l,k-m+1)} \right)\right].
     \end{split}
   \end{align}
Then $\{g^{(j,k)}\}$ is a sequence of derivatives with amplitude $\mathcal{C}_{g}$ and the same principal directions $(\alpha,\beta)$ and principal frequencies $(r,s)$, where 
\begin{align}
\label{nuevaamplitud}
    \mathcal{C}_{g}(t)= \max_{(i_1,i_2)\in \mathcal{J} }\mathcal{C}_{i_1}(t)\mathcal{C}_{i_2}(t) r(t)s(t)^\gamma |\sin(\alpha(t)-\beta(t))|.
\end{align}
\end{lma}
\begin{rmk}
    Notice that the implicit constant in Property \ref{amplitud} for $\{g^{(j,k)}\}$ depends on $j,k,\gamma$ and on the cardinality of $\mathcal{J}$ (which in practice will depend also on $\gamma$). Yet, it does not depend on the amplitudes, principal directions or frequencies.
\end{rmk}
\begin{rmk}
    It is plain to check that the expression (\ref{recursividad}) extends (\ref{nuevaperturbacion00}). More precisely, (\ref{recursividad}) is also valid for $(j,k)=(0,0)$.
\end{rmk}
\begin{rmk}
    Notice that, in the Definition \ref{sod} we allow the amplitudes $\mathcal{C}_i$ to be zero for some times. In practice, the amplitudes of the principal layers will never vanish and the corrections will generate sequences of derivatives with amplitudes that will vanish just on the activation times $t_n$. It may seem problematic since in the expression \eqref{recursividad} we find amplitudes on the denominator. Nevertheless, notice that on those points, one can check that $g^{(j,k)}$ will also vanish: in this prove we will show that the expression between brackets in \eqref{recursividad} is always bounded, since $f_1^{l,m}$ is also bounded by $\mathcal{C}_i$, so the factor $\mathcal{C}_{i_1}\mathcal{C}_{i_2}$ makes  $g^{(j,k)}$ vanish in the points of conflict.
\end{rmk}
\begin{proof}[Proof of Lemma \ref{nuevaperturbacion}]
    We have to prove every property of Definition \ref{sod}. Property \ref{paridad} can be easily checked using that $f_i^{(j,k)}$ verifies Property \ref{paridad} for $i=0,\ldots ,I$. To be more precise, we first notice that the operator $\nabla$ and the velocity \eqref{velocidad} change the parity of a function with respect to the origin. Hence, the case $(j,k)=(0,0)$ is trivial taking into account that the product of two odd functions is even. For arbitrary $(j,k)\neq (0,0)$ such that $j+k$ is even, using that the fractional Laplacian (\ref{laplacianofraccionario}) does not change the parity, we find in (\ref{recursividad}) the product of odd functions for $l+m$ even and the product of even functions for $l+m$ odd and so, in any case, $g^{(j,k)}$ is a sum of even functions which is even. The same reasoning can be applied to conclude that $g^{(j,k)}$ is odd provided that  $j+k$ is odd.  
    
    To obtain  \ref{amplitud} with $\mathcal{C}_{g}$ satisfying (\ref{nuevaamplitud}) we first notice that, for $i\leq I$, using Properties \ref{amplitud} and \ref{soporte},
    \begin{align}
        \bigg|(-\Delta)^{-\frac{1-\gamma}{2}} f_{i}^{(j,k)}\left(x,t\right)\bigg|&=\left|c_\gamma\int_{\mathbb{R}^2} \frac{f_{i}^{(j,k)}(x-y,t)}{|y|^{1+\gamma}}dy\right| \nonumber \\
        &\lesssim_{\gamma, j, k}  \|f_{i}^{(j,k)}(\cdot,t)\|_{L^\infty} \int_{|\textbf{n}(\beta)\cdot y|\leq 2/s}  \frac{1}{|y|^{1+\gamma}}dy \lesssim_{\gamma,j, k} \mathcal{C}_is^{-1+\gamma} \label{P12.8.1}
    \end{align}
    for all $x\in \mathcal{S}(\alpha,\beta,r,s)$. 
    Using this and Property \ref{amplitud} for $f^{(j, k)}_i$ with $i=0, \ldots, I$, we are able to estimate
     \begin{align*}
        &\left|\frac{s^{1-\gamma}}{\mathcal{C}_{i_1}\mathcal{C}_{i_2}}\sum_{l=0}^j\sum_{m=0}^k \left(\begin{matrix}
            j\\ l
       \end{matrix}\right)\left(\begin{matrix}
            k \\ m
       \end{matrix}\right)\left(f_{i_2}^{(l,m+1)} (-\Delta)^{-\frac{1-\gamma}{2}} f_{i_1}^{(j-l+1,k-m)}-f_{i_2}^{(l+1,m)} (-\Delta)^{-\frac{1-\gamma}{2}} f_{i_1}^{(j-l,k-m+1)} \right)\right|
       \\& \lesssim_{j,k} \sum_{l=0}^j\sum_{m=0}^k \left(\begin{matrix}
            j\\ l
       \end{matrix}\right)\left(\begin{matrix}
            k \\ m
       \end{matrix}\right) \lesssim_{j,k} 1,
    \end{align*}
    which proves \ref{amplitud} with amplitude $\mathcal{C}_{g}$ given by  (\ref{nuevaamplitud}).
    
    Clearly Property \ref{soporte} holds because $\text{supp } f_i^{(j,k)}\subset \mathcal{S}(\alpha,\beta,r,s)$ for $i=0, \ldots, I$.
    
    Next, we prove Property \ref{derivadas}. To fix ideas, we start with the model case $(j,k)=(0,0)$. Since $g^{(0,0)}$ is a sum of functions $v^\gamma (f_{i_1}^{(0,0)}) \cdot \nabla f_{i_2}^{(0,0)}$, it is enough to prove the result for  
     \begin{align*}
        g^{(0,0)}=v^\gamma (f_{i_1}^{(0,0)}) \cdot \nabla f_{i_2}^{(0,0)}.
    \end{align*}
     Using the assumption \ref{derivadas} for $i= 0, \ldots, I$ and the definition of the velocity (\ref{velocidad}), we have that 
        \begin{align*}
            v^\gamma (f_{i}^{(0,0)}) =-(-\Delta)^{-\frac{1-\gamma}{2}} \nabla^\perp f_{i}^{(0,0)}
            =-(-\Delta)^{-\frac{1-\gamma}{2}}f_i^{(1,0)} r\textbf{n}^\perp(\alpha)-(-\Delta)^{-\frac{1-\gamma}{2}}f_i^{(0,1)} s\textbf{n}^\perp(\beta).     
        \end{align*}
    Since $\textbf{n}(\alpha)^\perp \cdot \textbf{n}(\beta)=-\sin(\alpha-\beta)$ and $\textbf{n}(\beta)^\perp \cdot \textbf{n}(\alpha)=\sin(\alpha-\beta)$, we obtain  
    \begin{align*}
    \begin{split}
       g^{(0,0)}&=v^\gamma (f_{i_1}^{(0,0)}) \cdot \nabla f_{i_2}^{(0,0)}  
        =-rs \sin (\alpha-\beta) [-f_{i_2}^{(0,1)}(-\Delta)^{-\frac{1-\gamma}{2}}f_{i_1}^{(1,0)}+f_{i_2}^{(1,0)}(-\Delta)^{-\frac{1-\gamma}{2}}f_{i_1}^{(0,1)}]\\
        &=\mathcal{C}_{i_1}\mathcal{C}_{i_2}rs^\gamma \sin(\alpha-\beta)\left[\frac{s^{1-\gamma}}{\mathcal{C}_{i_1}\mathcal{C}_{i_2}}(f_{i_2}^{(0,1)}(-\Delta)^{-\frac{1-\gamma}{2}}f_{i_1}^{(1,0)}-f_{i_2}^{(1,0)}(-\Delta)^{-\frac{1-\gamma}{2}}f_{i_1}^{(0,1)})\right].
        \end{split}
    \end{align*}
    Using also that the following identity holds, at least for smooth  functions, 
    \begin{align*}
        \nabla (-\Delta)^{-\frac{1-\gamma}{2}} f_i^{(j,k)}=  (-\Delta)^{-\frac{1-\gamma}{2}}f_i^{(j+1,k)} r\textbf{n}(\alpha)+ (-\Delta)^{-\frac{1-\gamma}{2}}f_i^{(j,k+1)} s \textbf{n}(\beta),
    \end{align*}
  the first iteration ($g^{(1,0)}$ and $g^{(0,1)}$) can be easily computed, namely,
    \begin{align*}
        \nabla g^{(0,0)}=-\mathcal{C}_{i_1}\mathcal{C}_{i_2}r s^\gamma \sin(\alpha-\beta) \left[ \frac{s^{1-\gamma}}{\mathcal{C}_{i_1}\mathcal{C}_{i_2}}(-f_{i_2}^{(1,1)}(-\Delta)^{-\frac{1-\gamma}{2}}f_{i_1}^{(1,0)}+f_{i_2}^{(2,0)}(-\Delta)^{-\frac{1-\gamma}{2}}f_{i_1}^{(0,1)}\right.\\  \left.
        -f_{i_2}^{(0,1)}(-\Delta)^{-\frac{1-\gamma}{2}}f_{i_1}^{(2,0)}+f_{i_2}^{(1,0)}(-\Delta)^{-\frac{1-\gamma}{2}}f_{i_1}^{(1,1)}) r \textbf{n}(\alpha) \right. \\ \left.
        +\frac{s^{1-\gamma}}{\mathcal{C}_{i_1}\mathcal{C}_{i_2}}(-f_{i_2}^{(0,2)}(-\Delta)^{-\frac{1-\gamma}{2}}f_{i_1}^{(1,0)}+f_{i_2}^{(1,1)}(-\Delta)^{-\frac{1-\gamma}{2}}f_{i_1}^{(0,1)}\right.\\  \left.
        -f_{i_2}^{(0,1)}(-\Delta)^{-\frac{1-\gamma}{2}}f_{i_1}^{(1,1)}+f_{i_2}^{(1,0)}(-\Delta)^{-\frac{1-\gamma}{2}}f_{i_1}^{(0,2)}) s \textbf{n}(\beta)
        \right].
    \end{align*} 
     Same ideas as employed in the case $(j,k)=(0,0)$ can in fact be extended to the general setting $(j,k) \in \mathbb{N}_0^2$. Indeed, we have
   \begin{align*}
   & \frac{\nabla  g^{(j,k)}}{\mathcal{C}_{i_1}\mathcal{C}_{i_2}rs^\gamma \sin(\alpha-\beta)}=\\
       &\hspace{5mm}=-\frac{s^{1-\gamma}}{\mathcal{C}_{i_1}\mathcal{C}_{i_2}}\sum_{l=0}^j\sum_{m=0}^k \left(\begin{matrix}
            j\\ l
       \end{matrix}\right)\left(\begin{matrix}
            k \\ m
       \end{matrix}\right) \left(f_{i_2}^{(l+2,m)} (-\Delta)^{-\frac{1-\gamma}{2}} f_{i_1}^{(j-l,k-m+1)}-f_{i_2}^{(l+1,m+1)} (-\Delta)^{-\frac{1-\gamma}{2}} f_{i_1}^{(j-l+1,k-m)}\right.\\
       &\hspace{5mm} \left. +f_{i_2}^{(l+1,m)} (-\Delta)^{-\frac{1-\gamma}{2}} f_{i_1}^{(j-l+1,k-m+1)} -f_{i_2}^{(l,m+1)} (-\Delta)^{-\frac{1-\gamma}{2}} f_{i_1}^{(j-l+2,k-m)}\right)r\textbf{n}(\alpha)\\
       &\hspace{5mm}-\frac{s^{1-\gamma}}{\mathcal{C}_{i_1}\mathcal{C}_{i_2}}\sum_{l=0}^j\sum_{m=0}^k \left(\begin{matrix}
            j\\ l
       \end{matrix}\right)\left(\begin{matrix}
            k \\ m
       \end{matrix}\right)\left(f_{i_2}^{(l+1,m+1)} (-\Delta)^{-\frac{1-\gamma}{2}} f_{i_1}^{(j-l,k-m+1)}-f_{i_2}^{(l,m+2)} (-\Delta)^{-\frac{1-\gamma}{2}} f_{i_1}^{(j-l+1,k-m)}\right.\\
       &\hspace{5mm}\left.+f_{i_2}^{(l+1,m)} (-\Delta)^{-\frac{1-\gamma}{2}} f_{i_1}^{(j-l,k-m+2)}-f_{i_2}^{(l,m+1)} (-\Delta)^{-\frac{1-\gamma}{2}} f_{i_1}^{(j-l+1,k-m+1)}  \right)s\textbf{n}(\beta)\\
       &\hspace{5mm}=-\frac{s^{1-\gamma}}{\mathcal{C}_{i_1}\mathcal{C}_{i_2}}\sum_{l=0}^{j+1}\sum_{m=0}^k \left(\begin{matrix}
            j+1\\ l
       \end{matrix}\right)\left(\begin{matrix}
            k \\ m
       \end{matrix}\right)\left(f_{i_2}^{(l+1,m)} (-\Delta)^{-\frac{1-\gamma}{2}} f_{i_1}^{(j+1-l,k-m+1)}\right.\\
       &\hspace{5mm}\left.-f_{i_2}^{(l,m+1)} (-\Delta)^{-\frac{1-\gamma}{2}} f_{i_1}^{(j+1-l+1,k-m)} \right)r\textbf{n}(\alpha)\\
       &\hspace{5mm}-\frac{s^{1-\gamma}}{\mathcal{C}_{i_1}\mathcal{C}_{i_2}}\sum_{l=0}^{j}\sum_{m=0}^{k+1} \left(\begin{matrix}
            j\\ l
       \end{matrix}\right)\left(\begin{matrix}
            k+1 \\ m
       \end{matrix}\right)\left(f_{i_2}^{(l+1,m)} (-\Delta)^{-\frac{1-\gamma}{2}} f_{i_1}^{(j-l,k+1-m+1)}\right.\\&\hspace{5mm}
       \left.-f_{i_2}^{(l,m+1)} (-\Delta)^{-\frac{1-\gamma}{2}} f_{i_1}^{(j-l+1,k+1-m)} \right)s\textbf{n}(\beta),
   \end{align*}
   where in the last step we have used a change of variables in the summation index and the property of the binomial coefficients 
   \begin{align*}
       \left(\begin{matrix}
            j\\ l
       \end{matrix}\right)+
       \left(\begin{matrix}
            j\\ l-1
       \end{matrix}\right)=
       \left(\begin{matrix}
            j+1\\ l
       \end{matrix}\right).
   \end{align*}
\end{proof}

\begin{lma}
\label{nuevaperturbacion2}
    Let $\{f_1^{(j,k)}\}\subset C^\infty ([t_\ast,1],C^\infty_c(\mathbb{R}^2))$ be a sequence of derivatives with amplitude $\mathcal{C}_1(t)$, principal directions $(\alpha(t), \beta(t))$ and principal frequencies $(r(t),s(t))$. Consider 
    \begin{align}
    \label{velocityfield}
        \overline{v}(X_t(x),t)= \partial_t X_t(x)  ,
    \end{align}
    where 
    \begin{align}
    \label{flowmap}
        \begin{split}X_t(x)&=\frac{1}{ \sin (\alpha(t)-\beta(t))} \left[\frac{s(t_\ast)}{s(t) } \left(\begin{matrix}
             \sin \alpha(t)\sin \alpha(t_\ast) & -\sin \alpha (t)\cos \alpha(t_\ast)\\
             -\cos \alpha(t) \sin \alpha (t_\ast)&  \cos \alpha( t)\cos \alpha(t_\ast)
        \end{matrix} \right)\right.\\
        &\left.+\frac{r(t_\ast)}{r(t)} \left(\begin{matrix}
             \sin \beta(t)\sin \beta(t_\ast) & -\sin \beta (t)\cos \beta(t_\ast)\\
             -\cos \beta(t) \sin \beta (t_\ast)&  \cos \beta( t)\cos \beta(t_\ast)
        \end{matrix} \right)\right]x,
        \end{split}
    \end{align}
    with initial conditions 
    \begin{align}
    \label{condicioninicialflowmap}
        r(t_\ast)=s(t_\ast), \qquad \alpha(t_\ast)=\beta(t_\ast)+\frac{\pi}{2}.
    \end{align}
    Let $\{f_2^{(j,k)}\}$ be defined by 
    \begin{align}
    \label{inttiempo}
        f_2^{(j,k)}(x,t)=-\int_{t_\ast}^t f_1^{(j,k)}(X_\tau (X_t^{-1}(x)),\tau) \, d\tau, \qquad \forall (j,k)\in \mathbb{N}_0^2.
    \end{align}
    Then $f_2^{(j,k)}$ is the unique solution to the Cauchy problem: 
    \begin{align}
    \label{cauchy}
        \begin{split}
           & \partial_t f_2^{(j,k)}+\overline{v}\cdot \nabla f_2^{(j,k)}=-f_1^{(j,k)},\\
           & f_2^{(j,k)}(x,t_\ast)=0.
        \end{split}
    \end{align}
    Moreover,  $\{f_2^{(j,k)}\}$ is a sequence of derivatives with principal directions $(\alpha(t),\beta(t))$, principal frequencies $(r(t),s(t))$ and amplitude $\mathcal{C}_2(t)$ given by
    \begin{align}
    \label{intamplitud}
        \mathcal{C}_2(t)=\int_{t_\ast}^t \mathcal{C}_1(\tau) \, d\tau.
    \end{align}
    \end{lma}
    \begin{proof}
        We start by proving that $\{f_2^{(j,k)}\}$ is a sequence of derivatives. Since $X_t$ is linear, it follows from (\ref{inttiempo}) that $f_1^{(j,k)}$ and $f_2^{(j,k)}$ have the same parity with respect to the origin. Also by (\ref{inttiempo}), it is clear that $\|f_2^{(j,k)}\|_{L^\infty}$ is bounded by the quantity given in  (\ref{intamplitud}). Then Properties \ref{paridad} and \ref{amplitud} hold. Now, let us prove Property \ref{soporte}, which holds basically because 
        \begin{align*}
            X_t(\mathcal{S}(\alpha(t_\ast),\beta(t_\ast), r(t_\ast),s(t_\ast)))=\mathcal{S}(\alpha(t),\beta(t), r(t),s(t)).
        \end{align*}
        Let $x\notin \mathcal{S}(\alpha(t),\beta(t), r(t),s(t))$. In particular,  $X_t^{-1}(x)\notin \mathcal{S}(\alpha(t_\ast),\beta(t_\ast), r(t_\ast),s(t_\ast))$. Then  $X_\tau (X_t^{-1}(x))\notin \mathcal{S}(\alpha(\tau),\beta(\tau), r(\tau),s(\tau))$ for every $\tau \in [t_\ast,1]$, and hence according to (\ref{inttiempo}), $f_2^{(j,k)}(x,t)=0$, so Property \ref{soporte} holds. Next we focus on Property  \ref{derivadas}. Let $(j,k)\in \mathbb{N}_0^2$. Differentiating under the integral sign, using the chain rule and the fact that $\{f_1^{(j,k)}\}$ verifies Property \ref{derivadas}, we infer that
        \begin{align*}
            &\nabla f^{(j,k)}_2(x,t)= -\int_{t_\ast}^t \nabla \left[f_1^{(j,k)}(X_\tau (X_t^{-1}(x)),\tau)\right] \, d\tau=-\int_{t_\ast}^t J^T(X_\tau\circ X_t^{-1})\nabla f_1^{(j,k)}(X_\tau (X_t^{-1}(x)),\tau) \, d\tau\\
            &\hspace{2mm}=-\int_{t_\ast}^t J^T(X_\tau\circ X_t^{-1}) \left(f_1^{(j+1,k)}(X_\tau (X_t^{-1}(x)),\tau)r(\tau)\textbf{n}(\alpha(\tau))+f_1^{(j,k+1)}(X_\tau (X_t^{-1}(x)),\tau)s(\tau)\textbf{n}(\beta(\tau))\right) \, d\tau\\
           &\hspace{2mm} =-\left(\int_{t_\ast}^t f_1^{(j+1,k)}(X_\tau (X_t^{-1}(x)),\tau) \, d\tau\right)r(t)\textbf{n}(\alpha(t))
            -\left(\int_{t_\ast}^t f_1^{(j,k+1)}(X_\tau (X_t^{-1}(x)),\tau) \, d\tau\right)s(t)\textbf{n}(\beta(t))\\
            &\hspace{2mm}= f_2^{(j+1,k)}(x,t)r(t)\textbf{n}(\alpha(t))+f_2^{(j,k+1)}(x,t)s(t)\textbf{n}(\beta(t)),
        \end{align*}
        where, in the penultimate step, we have also used that, if we define 
    \begin{align*}
        M(t)=&\frac{1}{ \sin (\alpha(t)-\beta(t))} \left[\frac{s(t_\ast)}{s(t) } \left(\begin{matrix}
             \sin \alpha(t)\sin \alpha(t_\ast) & -\sin \alpha (t)\cos \alpha(t_\ast)\\
             -\sin \alpha (t_\ast)\cos \alpha(t) &  \cos \alpha( t)\cos \alpha(t_\ast)
        \end{matrix} \right)\right.\\
        &\left.+\frac{r(t_\ast)}{r(t)} \left(\begin{matrix}
             \sin \beta(t)\sin \beta(t_\ast) & -\sin \beta (t)\cos \beta(t_\ast)\\
             -\sin \beta (t_\ast)\cos \beta(t) &  \cos \beta( t)\cos \beta(t_\ast)
        \end{matrix} \right)\right],
    \end{align*}
    then by \eqref{flowmap} one can prove that 
    \begin{align*}
        J^T(X_\tau\circ X_t^{-1})=\left(M(t)^{-1}\right)^T M(\tau)^T,
    \end{align*}
    and from here (using \eqref{condicioninicialflowmap}), one can deduce that
        \begin{align*}
            &J^T(X_\tau\circ X_t^{-1})r(\tau)\textbf{n}(\alpha(\tau))=r(t)\textbf{n}(\alpha(t)), \\
             &J^T(X_\tau\circ X_t^{-1})s(\tau)\textbf{n}(\beta(\tau))= s(t)\textbf{n}(\beta(t)).
        \end{align*}
    Notice that, by construction,  $X_t(x)$ is the flow map associated to the velocity field $\overline{v}$ (see (\ref{velocityfield}) and (\ref{flowmap})), so the classical Duhamel formula asserts that $f_2^{(j,k)}$ given by \eqref{inttiempo} is the unique solution to the Cauchy problem (\ref{cauchy}).
   \end{proof}

\subsection{Velocity generated by the layers}\label{sec:velocitylayers}
Now we compute the velocity that generates each layer. Notice that each layer will be a finite sum of functions that generate sequences of derivatives with common principal directions and frequencies. Hence, to compute the velocity of a layer at a linear stage (and to control the cubic reminder) we only need the following lemma.
\begin{lma}
\label{lemavelocidad}
    Let $\{f^{(j,k)}\}$ be a sequence of derivatives with amplitude $\mathcal{C}(t)$, principal directions $(\alpha(t),\beta(t))$ and principal frequencies $(r(t),s(t))$. If $x\in \mathcal{S}(\alpha(t),\beta(t),r(t),s(t))$ then
    \begin{align}
    \label{formulavelocidad}
    \begin{split}
        v^\gamma(f^{(0,0)})(x,t)=&s^{1+\gamma}  \mathcal{C}\left[ a^{(1,1)} \left(\begin{matrix}
              \cos 2\beta& \sin 2\beta \\
              \sin 2\beta  &-\cos 2 \beta
          \end{matrix}\right)+a^{(1,2)} \left(\begin{matrix}
              -\sin \beta \cos \beta&\cos^2\beta\\
              -\sin^2 \beta&\sin \beta \cos \beta 
          \end{matrix}\right)\right.\\
          &\left.
          +a^{(2,1)} \left(\begin{matrix}
              -\sin \beta \cos \beta&-\sin^2\beta\\
              \cos^2 \beta&\sin \beta \cos \beta 
          \end{matrix}\right)
          \right]x\\
           &-\textbf{n}(\alpha)^\perp \frac{r}{2 s^{1-\gamma}} \sum_{l=0}^3b^{(4-l,l)}\left(\begin{matrix}
                    3\\ l
                \end{matrix}\right)
                \left(r \textbf{n}(\alpha) \cdot x\right)^{3-l} (s\textbf{n}(\beta)\cdot x)^l \\
&- \textbf{n}(\beta)^\perp  \frac{s^\gamma}{2}\sum_{l=0}^3b^{(3-l,l+1)}\left(\begin{matrix}
                    3\\ l
                \end{matrix}\right)
                \left(r \textbf{n}(\alpha) \cdot x\right)^{3-l} (s\textbf{n}(\beta)\cdot x)^l, 
                \end{split}
    \end{align}
    where 
    \begin{align}
    \begin{split}
    \label{coeficientesvelocidad}
         a^{(1,1)}(t)&:=\frac{\sin (\alpha-\beta)}{\mathcal{C}} \left[\left(\frac{r}{s}\right)^2d^{(2,0)}\cos (\alpha-\beta)+\frac{r}{s}d^{(1,1)}\right],\\
            a^{(1,2)}(t)&:=\frac{\sin^2 (\alpha-\beta)}{\mathcal{C}}\left(\frac{r}{s}\right)^2d^{(2,0)},\\
            a^{(2,1)}(t)&:=-\frac{1}{\mathcal{C}}\left[d^{(0,2)}+2\frac{r}{s}d^{(1,1)}\cos (\alpha-\beta)+\left(\frac{r}{s}\right)^2d^{(2,0)}\cos^2(\alpha-\beta)\right],\\
            a^{(2,2)}(t)&:=-a^{(1,1)}(t),
            \end{split}
    \end{align}
    and (recall \eqref{DefinitionAlignedRescaled} for the definition of $\tilde{f}^{(j, k)}$)
    \begin{align}
        d^{(j,k)}(t)&:= \frac{2^{\gamma}\Gamma(\frac{1+\gamma}{2})}{\pi \Gamma(\frac{1-\gamma}{2})}\int_{\mathbb{R}_+\times \mathbb{R}} \frac{\tilde{f}^{(j,k)}(y,t)}{|y|^{1+\gamma}} \, dy, \label{Defdjk}\\
         b^{(j,k)}(x,t)&:= \frac{2^{\gamma-1}\Gamma(\frac{1+\gamma}{2})}{\pi \Gamma(\frac{1-\gamma}{2})} \,  \int_{\mathbb{R}_+\times \mathbb{R}} \left(\int_0^1(1-\eta)^2\frac{\tilde{f}^{(j,k)}(\eta sR_{-\beta}x+y,t)+\tilde{f}^{(j,k)}(\eta sR_{-\beta}x-y,t)}{|y|^{1+\gamma}}\, d\eta \right)\, dy \label{Defbjk}.
    \end{align}
    In addition,
    \begin{align}
    \label{cotadb}
        |d^{(j,k)}(t)|\lesssim_{\gamma, j, k} \mathcal{C}(t),\hspace{7mm}
         |b^{(j,k)}(x,t)|\lesssim_{\gamma, j, k} \mathcal{C}(t).
    \end{align}
    \end{lma}
        \begin{rmk}
            Notice that $a^{(j,k)}(t), d^{(j,k)}(t), b^{(j,k)}(x,t)$ depend also on the parameter $\gamma$. Since this will not play any role, we omit such a dependence in the notation.

            Also, we omitted the spatial dependence of $b^{(j,k)}$. In the statement, $b^{(j,k)}=b^{(j,k)}(x,t)$.
        \end{rmk}
        \begin{rmk}
    Notice that, in the Definition \ref{sod} we allow the amplitudes $\mathcal{C}$ to be zero for some times. In practice, the amplitudes of the principal layers will never vanish and the corrections will generate sequences of derivatives with amplitudes that will vanish just on the activation times $t_n$. It may seem problematic since in the expression \eqref{coeficientesvelocidad} we find amplitudes on the denominator. On those points, one can just consider that $a^{(j,k)}$ is just the limit when $\mathcal{C}$ goes to zero. This limit exists due to how $\tilde{f}^{(j,k)}$ is in the definition of $d^{(j,k)}$ and using that $\tilde{f}^{(j,k)}$ is also bounded by $\mathcal{C}(t)$. Nevertheless, is is enough to prove that $a^{(j,k)}$ is bounded, since the factor $\mathcal{C}$ in the beginning of \eqref{formulavelocidad} makes the velocity vanish in the points of conflict.
\end{rmk}
    \begin{proof}[Proof of Lemma \ref{lemavelocidad}]
        For technical reasons, it is more convenient to deal first with the velocity generated by the aligned and rescaled function $\tilde{f}^{(0, 0)}$. Using the definition of the velocity operator (\ref{velocidad}), the formula (\ref{derivadasaligned}), and taking into account that the perpendicular gradient and the fractional Laplacian commute when applied to regular enough functions, we have
        \begin{align}
        \begin{split}
        \label{formulavelocidad2}
            v^\gamma(\tilde{f}^{(0,0)})(x,t)&= -(-\Delta)^{-\frac{1-\gamma}{2}}\nabla^\perp \tilde{f}^{(0,0)} (x,t)\\
            &=-(-\Delta)^{-\frac{1-\gamma}{2}}\tilde{f}^{(1,0)}(x,t) \frac{r(t)}{s(t)} \textbf{n}(\alpha(t)-\beta(t))^\perp- (-\Delta)^{-\frac{1-\gamma}{2}}\tilde{f}^{(0,1)} (x,t)\textbf{n}(0)^\perp.
        \end{split}
        \end{align}
        Now, let us deal with $(-\Delta)^{-\frac{1-\gamma}{2}}\tilde{f}^{(1,0)}$ and $ (-\Delta)^{-\frac{1-\gamma}{2}}\tilde{f}^{(0,1)}$. Let $x\in \mathcal{S}(\alpha-\beta,0,\frac{r}{s},1)$. By virtue of  (\ref{laplacianofraccionario}) and the fact that $\tilde{f}^{1, 0}$ is odd with respect to the origin  (see Property \ref{paridad}), we can write  
        \begin{align}
           \frac{\pi \Gamma(\frac{1-\gamma}{2})}{2^{\gamma-1}\Gamma(\frac{1+\gamma}{2})}  (-\Delta)^{-\frac{1-\gamma}{2}}\tilde{f}^{(1,0)} (x, t) & = \int_{\mathbb{R}^2}\frac{\tilde{f}^{(1,0)}\left(x-y,t\right)}{|y|^{1+\gamma}} \, dy \nonumber \\
           &= \int_{\mathbb{R}_+\times \mathbb{R}}\frac{\tilde{f}^{(1,0)}\left(y+x,t\right)-\tilde{f}^{(1,0)}\left(y-x,t\right)}{|y|^{1+\gamma}} \, dy. \label{formulavelocidad2aa}
           \end{align}
    Furthermore, applying the Taylor expansion of the function $$x \mapsto \tilde{f}^{(1, 0)}(y+x, t)- \tilde{f}^{(1, 0)}(y-x, t)$$  and (\ref{derivadasaligned}), we can compute
           \begin{align}
            \int_{\mathbb{R}_+\times \mathbb{R}}\frac{\tilde{f}^{(1,0)}\left(y+x,t\right)-\tilde{f}^{(1,0)}\left(y-x,t\right)}{|y|^{1+\gamma}} \, dy & \nonumber \\
            & \hspace{-6cm}=\int_{\mathbb{R}_+\times \mathbb{R}}\left(\frac{2\nabla \tilde{f}^{(1,0)}(y,t)}{|y|^{1+\gamma}}\cdot x+ \frac{1}{2} 
            \frac{1}{|y|^{1+\gamma}}\left(\int_0^1(1-\eta)^2(x\cdot\nabla)^3\xi^{(1,0)}_y(\eta x)d\eta\right) \right)\, dy \nonumber  \\
           &\hspace{-6cm} =\int_{\mathbb{R}_+\times \mathbb{R}} \frac{2}{|y|^{1+\gamma}}\left(
\tilde{f}^{(2,0)}(y,t) \frac{r(t)}{s(t)}\textbf{n}(\alpha(t)-\beta(t))+\tilde{f}^{(1,1)}(y,t)\textbf{n}(0)
            \right)\cdot x \, dy \nonumber\\
            & \hspace{-5cm}+  \frac{1}{2}\int_{\mathbb{R}_+\times \mathbb{R}} 
            \frac{1}{|y|^{1+\gamma}}\left(\int_0^1(1-\eta)^2(x\cdot\nabla)^3\xi^{(1,0)}_y(\eta x,t)d\eta\right) \, dy, \label{formulavelocidad2aaa}
            \end{align}
            where
            \begin{align*}
                \xi^{(j,k)}_y(\eta x,t):=\tilde{f}^{(j,k)}(\eta x+y,t)+\tilde{f}^{(j,k)}(\eta x-y,t)
            \end{align*}
            for any $\eta \in [0,1]$.
            
            Moreover, elementary computations show that (see \eqref{derivadasaligned})
            \begin{align*}
                (x \cdot \nabla)^3 \xi^{(1,0)}_y(\eta x,t)=\sum_{l=0}^3\left(\begin{matrix}
                    3\\ l
                \end{matrix}\right)
                \left(\frac{r}{s} \textbf{n}(\alpha-\beta) \cdot x\right)^{3-l} (\textbf{n}(0)\cdot x)^l \xi_y^{(4-l,l)}(\eta x,t). 
            \end{align*}
            Combining this with \eqref{formulavelocidad2aa} and \eqref{formulavelocidad2aaa}, we achieve  
            \begin{align}
                &\frac{\pi \Gamma(\frac{1-\gamma}{2})}{2^{\gamma-1}\Gamma(\frac{1+\gamma}{2})}  (-\Delta)^{-\frac{1-\gamma}{2}}\tilde{f}^{(1,0)}(x,t) \nonumber \\
                 &\hspace{6mm}=2 \frac{r}{s}\textbf{n}(\alpha-\beta)\cdot x\int_{\mathbb{R}_+\times \mathbb{R}} \frac{\tilde{f}^{(2,0)}(y,t)}{|y|^{1+\gamma}} \, dy+ 2 \textbf{n}(0)\cdot x\int_{\mathbb{R}_+\times \mathbb{R}} \frac{\tilde{f}^{(1,1)}(y,t)}{|y|^{1+\gamma}} \, 
             dy \nonumber
             \\& \hspace{6mm}+\frac{1}{2}\sum_{l=0}^3\left(\begin{matrix}
                    3\\ l
                \end{matrix}\right)
                \left(\frac{r}{s} \textbf{n}(\alpha-\beta) \cdot x\right)^{3-l} (\textbf{n}(0)\cdot x)^l \int_{\mathbb{R}_+\times \mathbb{R}} \left(\int_0^1 (1-\eta)^2\frac{\xi_y^{(4-l,l)}(\eta x,t)}{|y|^{1+\gamma}}d\eta \right) dy. \label{P26.6.1}
            \end{align}
    Furthermore, reasoning similarly for $\tilde{f}^{(0, 1)}$, one can also check that
             \begin{align}
            &\frac{\pi \Gamma(\frac{1-\gamma}{2})}{2^{\gamma-1}\Gamma(\frac{1+\gamma}{2})} 
                (-\Delta)^{-\frac{1-\gamma}{2}}\tilde{f}^{(0,1)}(x,t) = \nonumber \\
                 &\hspace{6mm}=2 \frac{r}{s}\textbf{n}(\alpha-\beta)\cdot x\int_{\mathbb{R}_+\times \mathbb{R}} \frac{\tilde{f}^{(1,1)}(y,t)}{|y|^{1+\gamma}} \, dy+2 \textbf{n}(0)\cdot x\int_{\mathbb{R}_+\times \mathbb{R}} \frac{\tilde{f}^{(0,2)}(y,t)}{|y|^{1+\gamma}} \,
             dy \nonumber
             \\&\hspace{6mm}+\frac{1}{2}\sum_{l=0}^3\left(\begin{matrix}
                    3\\ l
                \end{matrix}\right)
               \left(\frac{r}{s} \textbf{n}(\alpha-\beta) \cdot x\right)^{3-l} (\textbf{n}(0)\cdot x)^l \int_{\mathbb{R}_+\times \mathbb{R}} \left(\int_0^1 (1-\eta)^2\frac{\xi_y^{(3-l,l+1)}(\eta x,t)}{|y|^{1+\gamma}} \, d\eta \right)dy.\label{P26.6.2}
            \end{align}

        
        Plugging \eqref{P26.6.1} and \eqref{P26.6.2} into (\ref{formulavelocidad2}), 
        \begin{align}
             v^\gamma(\tilde{f}^{(0,0)})(x,t)=&- 
\left(
\frac{r}{s}d^{(2,0)}\textbf{n}(\alpha-\beta)\cdot x+d^{(1,1)}\textbf{n}(0)\cdot x\right. \nonumber
             \\ &\left.+\frac{1}{2}\sum_{l=0}^3b^{(4-l,l)}\left(\frac{R_\beta x}{s},t\right)\left(\begin{matrix}
                    3\\ l
                \end{matrix}\right)
                \left(\frac{r}{s} \textbf{n}(\alpha-\beta) \cdot x\right)^{3-l} (\textbf{n}(0)\cdot x)^l 
\right)
            \frac{r}{s} \textbf{n}(\alpha-\beta)^\perp \nonumber
            \\&- \left(
\frac{r}{s}d^{(1,1)}\textbf{n}(\alpha-\beta)\cdot x+d^{(0,2)}\textbf{n}(0)\cdot x\right. \nonumber
             \\ &\left.+\frac{1}{2}\sum_{l=0}^3b^{(3-l,l+1)}\left(\frac{R_\beta x}{s},t\right)\left(\begin{matrix}
                    3\\ l
                \end{matrix}\right)
                \left(\frac{r}{s} \textbf{n}(\alpha-\beta) \cdot x\right)^{3-l} (\textbf{n}(0)\cdot x)^l 
\right)
           \textbf{n}(0)^\perp \nonumber\\
           =&-\left(\frac{r}{s}\right)^2 d^{(2,0)}\left(\begin{matrix}
               -\sin (\alpha-\beta) \cos (\alpha-\beta)& -\sin^2 (\alpha-\beta)\\
               \cos^2 (\alpha-\beta)& \sin (\alpha-\beta)\cos (\alpha-\beta)
           \end{matrix}\right)x\nonumber\\
           &-\frac{r}{s} d^{(1,1)}\left(\begin{matrix}
               -\sin (\alpha-\beta) & 0\\
               \cos(\alpha-\beta)& 0
           \end{matrix}\right)x
           -\frac{r}{s} d^{(1,1)}\left(\begin{matrix}
               0 & 0\\
               \cos(\alpha-\beta)& \sin(\alpha-\beta)
           \end{matrix}\right)x \nonumber
           \\&- d^{(0,2)}\left(\begin{matrix}
               0 & 0\\
               1& 0
           \end{matrix}\right)x \nonumber \\&
           -\frac{1}{2}\textbf{n}(\alpha-\beta)^\perp \frac{r}{s} \sum_{l=0}^3b^{(4-l,l)}\left(\frac{R_\beta x}{s},t\right)\left(\begin{matrix}
                    3\\ l
                \end{matrix}\right)
                \left(\frac{r}{s} \textbf{n}(\alpha-\beta) \cdot x\right)^{3-l} (\textbf{n}(0)\cdot x)^l \nonumber  \\&
- \frac{1}{2} \textbf{n}(0)^\perp\sum_{l=0}^3b^{(3-l,l+1)}\left(\frac{R_\beta x}{s},t\right)\left(\begin{matrix}
                    3 \\ l
                \end{matrix}\right)
                \left(\frac{r}{s} \textbf{n}(\alpha-\beta) \cdot x\right)^{3-l} (\textbf{n}(0)\cdot x)^l.  \label{P26.6.4}         
        \end{align}

        Now we wish to switch the role of $\tilde{f}^{(0, 0)}$ in the above computations to the original one $f^{(0, 0)}$. To do this, one can use  (\ref{velocidad}) and \eqref{DefinitionAlignedRescaled} together with a simple change of variables in such a way that 
        \begin{align*}
            v^\gamma(\tilde{f}^{(0,0)})(x,t)&=c_\gamma \int_{\mathbb{R}^2} \frac{(x-y)^\perp}{|x-y|^{3+\gamma}} \, \tilde{f}^{(0,0)} (y,t) \, dy\\
           & =c_\gamma \int_{\mathbb{R}^2} \frac{(x-y)^\perp}{|x-y|^{3+\gamma}} \, f^{(0,0)} \left(\frac{R_\beta y}{s},t\right) \, dy
            \\&=s^{-\gamma} c_\gamma \int_{\mathbb{R}^2} \frac{R_{-\beta}(\frac{R_\beta x}{s}-y)^\perp}{|\frac{R_\beta x}{s}-y|^{3+\gamma}} \,  f^{(0,0)} \left(y,t\right) \, dy=s^{-\gamma}R_{-\beta} v^\gamma(f^{(0,0)})\left(\frac{R_\beta x}{s},t\right)
            \end{align*}
            or in other words,
            \begin{align}\label{P26.6.3}
            v^\gamma(f^{(0,0)})(x,t)=s^\gamma R_\beta  v^\gamma(\tilde{f}^{(0,0)}) (sR_{-\beta}x,t). 
        \end{align}
       
        Putting together \eqref{P26.6.4} and \eqref{P26.6.3} (and after some elementary manipulations), we can write  
        \begin{align*}
          v^\gamma(f^{(0,0)})(x,t)=&s^{1+\gamma}  \mathcal{C}\left[ a^{(1,1)} \left(\begin{matrix}
              \cos 2\beta& \sin 2\beta \\
              \sin 2\beta  &-\cos 2 \beta
          \end{matrix}\right)+a^{(1,2)} \left(\begin{matrix}
              -\sin \beta \cos \beta&\cos^2\beta\\
              -\sin^2 \beta&\sin \beta \cos \beta 
          \end{matrix}\right)\right.\\
          &\left.
          +a^{(2,1)} \left(\begin{matrix}
              -\sin \beta \cos \beta&-\sin^2\beta\\
              \cos^2 \beta&\sin \beta \cos \beta 
          \end{matrix}\right)
          \right] x\\
           &-\textbf{n}(\alpha)^\perp \frac{r}{2 s^{1-\gamma}} \sum_{l=0}^3b^{(4-l,l)}\left(\begin{matrix}
                    3\\ l
                \end{matrix}\right)
                \left(r \textbf{n}(\alpha) \cdot x\right)^{3-l} (s\textbf{n}(\beta)\cdot x)^l \\
&- \textbf{n}(\beta)^\perp  \frac{s^\gamma}{2}\sum_{l=0}^3b^{(3-l,l+1)}\left(\begin{matrix}
                    3\\ l
                \end{matrix}\right)
                \left(r \textbf{n}(\alpha) \cdot x\right)^{3-l} (s\textbf{n}(\beta)\cdot x)^l, 
         \end{align*}
         where the coefficients $a^{(1, 1)}, a^{(1, 2)}$ and $a^{(2, 1)}$ are given by \eqref{coeficientesvelocidad} and $b^{(l,m)}=b^{(l,m)}(x,t)$ are evaluated in $x$. This shows the validity of the desired formula \eqref{formulavelocidad}. 

                    Next we turn our attention to the  quantity $d^{(j, k)}$ introduced in \eqref{Defdjk}. One can easily check that 
            \begin{align*}
             |d^{(j, k)}(t)| = \frac{2^{\gamma}\Gamma(\frac{1+\gamma}{2})}{\pi \Gamma(\frac{1-\gamma}{2})} \, \left| \int_{\mathbb{R}_+\times \mathbb{R}} \frac{\tilde{f}^{(j,k)}(y,t)}{|y|^{1+\gamma}} \, dy\right|\lesssim_{\gamma, j, k} \mathcal{C}\int_{0\leq y_1\leq 1}  \frac{dy}{|y|^{1+\gamma}} \lesssim \mathcal{C}
            \end{align*}
            since $\gamma \in (0, 1)$. This proves the first estimate in \eqref{cotadb}.

        On the other hand, for any $x\in \mathcal{S}(\alpha,\beta,r, s)$, one obtains that $sR_{-\beta}x\in \mathcal{S}(\alpha-\beta,0,\frac{r}{s}, 1)\supset \text{supp}(\tilde{f}^{(j,k)})$. Using this, the terms $b^{(j, k)}$ given by \eqref{Defbjk} are bounded by
        \begin{align}
           \left|b^{(j,k)}(x,t)\right|=&\frac{2^{\gamma-1}\Gamma(\frac{1+\gamma}{2})}{\pi \Gamma(\frac{1-\gamma}{2})} \, \left| \int_{\mathbb{R}_+\times \mathbb{R}} \left(\int_0^1(1-\eta)^2\frac{\tilde{f}^{(j,k)}(\eta sR_{-\beta} x+y,t)+\tilde{f}^{(j,k)}(\eta sR_{-\beta}x-y,t)}{|y|^{1+\gamma}}\, d\eta \right)\, dy \right| \nonumber \\
           & \lesssim_{\gamma, j, k} \mathcal{C} \int_{0 \leq y_1 \leq 2} \frac{dy}{|y|^{1+\gamma}} \lesssim \mathcal{C}, \label{P11.8.5}
          \end{align}
          where the assumption $\gamma \in (0, 1)$ was used in the last step. This leads to the second estimate of \eqref{cotadb}. 
    \end{proof}

\subsection{Dynamics of the layers}\label{S3.2}
    Let $\phi_0:\mathbb{R}\rightarrow [0, \infty)$ be an even $C^\infty$ function, non-increasing in $[0,\infty)$, with compact support in $(-1,1)$, such that 
    $$
        \phi_0 \equiv 1 \qquad \text{on} \qquad  \Big(-\frac{1}{10},\frac{1}{10} \Big).
    $$
    Now, consider $\phi=c\phi_0$,  where $c > 0$ is determined by
\begin{align}\label{P13.7.2}
   c^2 \, \frac{2^{\gamma}\Gamma(\frac{1+\gamma}{2})}{\pi \Gamma(\frac{1-\gamma}{2})} \int_{\mathbb{R}_+\times \mathbb{R}}\frac{\phi_0''(y_1)}{|y|^{1+\gamma}} \, dy= c \, \frac{2^{\gamma}\Gamma(\frac{1+\gamma}{2})}{\pi \Gamma(\frac{1-\gamma}{2})} \int_{\mathbb{R}_+\times \mathbb{R}}\frac{\phi''(y_1)}{|y|^{1+\gamma}} \, dy=-1.
\end{align}
As we explained, the dynamics of each layer are described by the linear part of the velocity generated by the previous layers. Hence, the goal of the Lemma \ref{dinamica} is to understand how a linear velocity affects to a layer, encoding the dynamics in a system of ODEs that describe the behavior   of the principal directions and frequencies.
         \begin{lma}
         \label{dinamica}
         Let
         \begin{align*}
            \overline{A}(t)=\left(\begin{matrix}
                \overline{a}^{(1,1)}&\overline{a}^{(1,2)}\\\overline{a}^{(2,1)}&\overline{a}^{(2,2)}
            \end{matrix}\right)(t), \qquad  \forall t\in[t_\ast,1], 
        \end{align*}
        where the coefficients $\overline{a}^{(i, j)} \in C^\infty$ for $i, j \in \{1, 2\}$, 
        and $\overline{v}(x, t)=\overline{A}(t)x$. Consider the Cauchy problem for the transport equation associated with $\overline{v}$:  
        \begin{align} \label{cauchyf}
        \begin{split}
            &\partial_t f (x,t) +\overline{v}(x,t)\cdot \nabla f (x,t)=0, \\
            &f(x,t_\ast)=\mathcal{C}\phi(r_\ast\textbf{n}(\alpha_\ast)\cdot x)\phi(s_\ast\textbf{n}(\beta_\ast)\cdot x),     
        \end{split}
        \end{align}
         where  $\mathcal{C} \neq 0$ and $r_\ast, s_\ast \in (0, \infty)$. Then there exists $t^\ast \in (t_\ast, 1]$ such that a unique solution $f$ to \eqref{cauchyf} exists and is given by
        \begin{align}\label{P28.6.1}
            f(x,t)= \mathcal{C} \phi (r(t)\textbf{n}(\alpha(t))\cdot x)\phi (s(t)\textbf{n}(\beta(t))\cdot x), \qquad  \forall t\in [t_\ast,t^\ast],
        \end{align}
        where $r, s, \alpha$ and $\beta$  solve the following ODEs: 
        \begin{align}
        \label{dinamicaalpha}
        \begin{split}
            r'(t)+r(t) (\overline{a}^{(1,1)}(t) \cos ^2 \alpha (t)+\overline{a}^{(2,2)}(t) \sin^2 \alpha(t)+(\overline{a}^{(1,2)}(t)+\overline{a}^{(2,1)}(t))\cos \alpha(t) \sin \alpha (t))=0,\\
           \alpha '(t)+\overline{a}^{(1,2)}(t)\cos^2 \alpha (t) -  \overline{a}^{(2,1)}(t) \sin^2\alpha (t)+ (\overline{a}^{(2,2)}(t)-\overline{a}^{(1,1)}(t)) \cos \alpha(t) \sin \alpha (t)=0,\\
            r(t_\ast)=r_\ast, \hspace{4mm} \alpha(t_\ast)=\alpha_\ast.
        \end{split}
        \end{align}
      and 
       \begin{align}
        \label{dinamicabeta}
        \begin{split}
            s'(t)+s(t) (\overline{a}^{(1,1)}(t) \cos ^2 \beta (t)+\overline{a}^{(2,2)}(t) \sin^2 \beta(t)+(\overline{a}^{(1,2)}(t)+\overline{a}^{(2,1)}(t))\cos \beta(t) \sin \beta (t))=0,\\
           \beta '(t)+\overline{a}^{(1,2)}(t)\cos^2 \beta (t) -  \overline{a}^{(2,1)}(t) \sin^2\beta (t)+ (\overline{a}^{(2,2)}(t)-\overline{a}^{(1,1)}(t)) \cos \beta(t) \sin \beta (t)=0,\\
            s(t_\ast)=s_\ast, \hspace{4mm} \beta(t_\ast)=\beta_\ast.
        \end{split}
        \end{align}
        Moreover, the time of existence $t^\ast$ can be extended as much as the time of existence of the solutions of (\ref{dinamicaalpha}) and (\ref{dinamicabeta}).
    \end{lma}
    \begin{proof}
        Assume that $f$ is given by (\ref{P28.6.1}) and let us compute its derivatives. The time derivative is
        \begin{align*}
            \partial_t f (x,t)&=\mathcal{C} ( r'(\cos \alpha x_1+\sin \alpha x_2)
            \\&+r\alpha'(-\sin \alpha x_1+\cos \alpha  x_2)) \phi '(r(\cos \alpha x_1+\sin \alpha  x_2)) \phi (s(\cos \beta x_1+\sin \beta x_2)) \\
            &+\mathcal{C} ( s'(\cos \beta x_1+\sin \beta x_2)
            \\&+s\beta'(-\sin \beta x_1+\cos \beta  x_2)) \phi (r(\cos \alpha x_1+\sin \alpha  x_2)) \phi '(s(\cos \beta x_1+\sin \beta x_2)) \\
            &=\mathcal{C}(r' \cos \alpha - r \alpha'\sin\alpha)x_1 \phi '(r \textbf{n}(\alpha) \cdot x) \phi (s \textbf{n}(\beta) \cdot x)\\
            &+\mathcal{C}(r' \sin \alpha + r\alpha' \cos\alpha)x_2 \phi '(r \textbf{n}(\alpha) \cdot x) \phi (s \textbf{n}(\beta) \cdot x)\\
            &+\mathcal{C}(s' \cos \beta - s \beta'\sin\beta )x_1 \phi (r \textbf{n}(\alpha) \cdot x) \phi' (s \textbf{n}(\beta) \cdot x)\\
            &+\mathcal{C}(s' \sin \beta + s\beta' \cos\beta)x_2 \phi (r \textbf{n}(\alpha) \cdot x) \phi' (s \textbf{n}(\beta) \cdot x).
        \end{align*}
        On the other hand, concerning the spatial derivatives, we have
        \begin{align*}
           (\overline{a}^{(1,1)}x_1+\overline{a}^{(1,2)}x_2) \partial_{x_1}f(x,t)&= \mathcal{C}\overline{a}^{(1,1)} r \cos \alpha x_1\phi' (r \textbf{n}(\alpha) \cdot x)\phi (s \textbf{n}(\beta) \cdot x)\\
           &+\mathcal{C}\overline{a}^{(1,2)} r \cos \alpha x_2\phi' (r \textbf{n}(\alpha) \cdot x)\phi (s \textbf{n}(\beta) \cdot x)\\
            &+\mathcal{C}\overline{a}^{(1,1)} s \cos \beta x_1\phi (r \textbf{n}(\alpha) \cdot x) \phi '(s \textbf{n}(\beta) \cdot x)\\
            &+\mathcal{C}\overline{a}^{(1,2)} s \cos \beta x_2\phi (r \textbf{n}(\alpha) \cdot x) \phi '(s \textbf{n}(\beta) \cdot x),
        \end{align*}
        and similarly
        \begin{align*}
     (\overline{a}^{(2,1)}x_1+\overline{a}^{(2,2)}x_2)\partial_{x_2}f(x,t)&= \mathcal{C} \overline{a}^{(2,1)} r \sin \alpha x_1\phi' (r \textbf{n}(\alpha) \cdot x) \phi (s \textbf{n}(\beta) \cdot x)\\
            &+\mathcal{C}\overline{a}^{(2,2)} r \sin \alpha x_2\phi' (r \textbf{n}(\alpha) \cdot x) \phi (s \textbf{n}(\beta) \cdot x)\\
            &+\mathcal{C}\overline{a}^{(2,1)} s \sin \beta x_1\phi (r \textbf{n}(\alpha) \cdot x) \phi '(s \textbf{n}(\beta) \cdot x)\\
            &+\mathcal{C}\overline{a}^{(2,2)} s \sin \beta x_2\phi (r \textbf{n}(\alpha) \cdot x) \phi '(s \textbf{n}(\beta) \cdot x).
        \end{align*}
        
        Notice that, if the pair $(r, \alpha)$ satisfies the following equations
        \begin{align}\label{P29.6.1}
    \left(\begin{matrix}
           \cos \alpha & -r \sin \alpha\\
           \sin \alpha & r \cos \alpha
       \end{matrix}\right) \begin{pmatrix} r' \\ \alpha' \end{pmatrix} = - r \begin{pmatrix} 
        \overline{a}^{(1,1)} \cos \alpha + \overline{a}^{(2,1)}  \sin \alpha \\
        \overline{a}^{(1,2)} \cos \alpha+\overline{a}^{(2,2)} \sin \alpha 
       \end{pmatrix}
        \end{align}
        and, similarly for $(s, \beta)$, 
              \begin{align}\label{P29.6.2}
    \left(\begin{matrix}
           \cos \beta & -s \sin \beta\\
           \sin \beta & s \cos \beta
       \end{matrix}\right) \begin{pmatrix} s' \\ \beta' \end{pmatrix} = - s \begin{pmatrix} 
        \overline{a}^{(1,1)} \cos \beta + \overline{a}^{(2,1)}  \sin \beta \\
        \overline{a}^{(1,2)} \cos \beta+\overline{a}^{(2,2)} \sin \beta 
       \end{pmatrix}
        \end{align}
        then (\ref{cauchyf}) is verified with $f$ given by (\ref{P28.6.1}). Moreover 
        $$
\det
\begin{pmatrix}
\cos\alpha & -r\sin\alpha\\
\sin\alpha & r\cos\alpha
\end{pmatrix}
=r \qquad \text{and} \qquad \det \begin{pmatrix}
\cos\beta & -s\sin\beta\\
\sin\beta & s\cos\beta
\end{pmatrix} = s.
        $$
        Assume momentarily that $r \neq 0$ and $s \neq 0$. Then 
        the systems \eqref{P29.6.1} and \eqref{P29.6.2} can be rewritten as 
        \begin{align}\label{P29.6.3}
            \begin{pmatrix} r' \\ \alpha' \end{pmatrix} = - \begin{pmatrix}
                r \cos \alpha (\overline{a}^{(1, 1)} \cos \alpha + \overline{a}^{(2, 1)} \sin \alpha) + r \sin \alpha (\overline{a}^{(1, 2)} \cos \alpha + \overline{a}^{(2, 2)} \sin \alpha) \\
                - \sin \alpha (\overline{a}^{(1, 1)} \cos \alpha + \overline{a}^{(2, 1)} \sin \alpha) + \cos \alpha (\overline{a}^{(1, 2)} \cos \alpha + \overline{a}^{(2, 2)} \sin \alpha)
            \end{pmatrix}
        \end{align}
        and
            \begin{align*}
            \begin{pmatrix} s' \\ \beta' \end{pmatrix} = - \begin{pmatrix}
                s \cos \beta (\overline{a}^{(1, 1)} \cos \beta + \overline{a}^{(2, 1)} \sin \beta) + s \sin \beta (\overline{a}^{(1, 2)} \cos \beta + \overline{a}^{(2, 2)} \sin \beta) \\
                - \sin \beta (\overline{a}^{(1, 1)} \cos \beta + \overline{a}^{(2, 1)} \sin \beta) + \cos \beta (\overline{a}^{(1, 2)} \cos \beta + \overline{a}^{(2, 2)} \sin \beta),
            \end{pmatrix}
        \end{align*}
        respectively. Adding to these  ODEs the initial conditions 
        \begin{align*}
            r(t_\ast)=r_\ast, \qquad  \alpha(t_\ast)=\alpha_\ast,
        \end{align*}
        and
        \begin{align*}
            s(t_\ast) = s_\ast, \qquad \beta(t_\ast) = \beta_\ast, 
        \end{align*}
        we can guarantee local existence of a unique solution $(r(t), \alpha(t))$ for some time interval $t\in [t_\ast,t^\ast]$ and, in particular, \eqref{dinamicaalpha} is satisfied. We can argue similarly for $(s(t), \beta(t))$ so that \eqref{dinamicabeta} also holds. Taking into account that the Cauchy problem (\ref{cauchyf}) has a unique local solution, we conclude that $f(x,t)$ given by (\ref{P28.6.1}) solves indeed (\ref{cauchyf}). Observe that the time of existence of the solution $f$ can be extended as the minimum of the time of existence of $r,\alpha$ and $s,\beta$.

        It remains to show that $r \neq 0$. Indeed,     multiplying the first equation inherited to \eqref{P29.6.1} by $\cos \alpha(t)$ and the second one by $\sin \alpha (t)$ and then summing both equations, we arrive at
        \begin{align*}
            r'+r(\overline{a}^{(1,1)}\cos^2 \alpha +\overline{a}^{(2,2)} \sin ^2 \alpha + (\overline{a}^{(1,2)}+\overline{a}^{(2,1)})\cos \alpha  \sin \alpha)=0, 
        \end{align*}
        which coincides with the first equation in \eqref{P29.6.3}. This gives an exponential type behavior for $r(t)$. Analogously, one can check that $s \neq 0$. 
    \end{proof}

	In Lemma \ref{lemaedos} below, we provide the analytic core of the construction. Its hypotheses
	\eqref{cotascoeficientes}--\eqref{sestatico} are assumptions on how the first $n$ layers should behave: each of them generates a linear velocity which is, up to relative errors, the shear
	$\tilde{a}^{(2,1)}_j\approx N_j^{\varepsilon}$ in the direction $\textbf{n}(\beta_j)^\perp$,
	and the directions $\beta_j$ are ordered, decreasing and increasingly close to one another.
	Its conclusions \eqref{estimaalpha2}--\eqref{sestatico2} describe the layer created at time
	$t_{n+1}$: principal frequencies are initially equal and principal directions initially perpendicular (see \eqref{IC}). As time evolves, its short frequency $r_{n+1}$ barely moves, as well as the principal direction $\alpha_{n+1}$. On the other hand, the long frequency $s_{n+1}$ grows like \eqref{estimas2} and the principal direction $\beta_{n+1}$ rotates from $\beta_n-\frac{\pi}{2}$ towards $\beta_n$ according to \eqref{estimabeta2}. The complete inductive proof in Lemma \ref{Lemma6} will be strongly based on  Lemma \ref{lemaedos}.

    To close some of the estimates in Lemma \ref{lemaedos}, we make the following assumptions on  the involved  parameters $\sigma, \mu, \delta$ and $\varepsilon$.

    \begin{parambracket}
    \label{parametros1}%
    \textbf{Assumptions on the parameters $\sigma, \mu, \delta$ and $\varepsilon$.} The parameters $\sigma<\mu<\delta<\varepsilon$ are chosen so that
 \begin{align*}
     \gamma \varepsilon>5\delta, \qquad \varepsilon<1-\frac{10\sigma}{\gamma}, \qquad  100\sigma <\min\{\gamma^2(1-\gamma), \mu\gamma\}, \qquad 
     10\mu<\delta. 
 \end{align*}
 We also assume that 
 \begin{align*}
    0< \delta <\frac{(1-\gamma)\gamma}{2}, \qquad  4\sigma<\frac{\delta}{1+\gamma+\delta}. 
 \end{align*}
 In particular, 
 \begin{align*}
     \varepsilon+4\sigma<\frac{2}{1+\gamma+\delta}(\varepsilon+\delta).
 \end{align*}
    \end{parambracket}
  
\begin{rmk}\label{Remark7}
        From now on,  we will use the letter $C$ to denote  an absolute constant $C > 1$ that may vary from line to line.  In fact, this constant will arise in a finite number of occasions and it is independent of $n$, but it may depend on the parameters $\varepsilon, \sigma, \delta, \mu$ and $\gamma$.  Moreover, choosing $N_0$ large enough, it may be possible to get rid of $C>1$ but paying a price (as small as we want) in the exponents of the frequencies. For example, if we obtain that a certain quantity is less or equal than $CN_n^{-\varepsilon}$, where  $C$ does not depend on $\{N_n\}$, then  one also has that 
            $CN_n^{-\varepsilon}\leq CN_0^{-\sigma}N_n^{-\varepsilon+\sigma}\leq N_n^{-\varepsilon+\sigma}$, 
        provided that $N_0$ is sufficiently large.
    \end{rmk}

\begin{lma}
\label{lemaedos}
Let $\{N_n\}_{n \in \mathbb{N}_0}$ be such that
\begin{align}
\label{des}
    N_{n-1}^{1+\frac{\gamma}{2}}\leq N_n\leq N_{n-1}^{1+\frac{3\gamma}{2}},\qquad  \forall n\in \mathbb{N}.
\end{align}
    Let 
    \begin{equation}\label{P14.7.2}
    r_0=1, \qquad \alpha_0=\frac{\pi}{2}, \qquad s_0=N_0, \qquad \beta_0=2\pi,
    \end{equation}
    and for $k=1,\ldots, n$  let $r_k, \alpha_k, s_k, \beta_k$ be  functions defined on the interval $[t_k,1]$, with $t_k$ determined by (\ref{tiempos}), such that  $r_k, \alpha_k$ verify the system of ODEs given by (\ref{dinamicaalpha}) and $s_k, \beta_k$ verify the system of ODEs given by (\ref{dinamicabeta}) with corresponding coefficients 
     \begin{align}
     \label{defcoeficientes}
     \begin{split}
            \overline{a}^{(1,1)}_k
            &= \sum_{j=0}^{k-1}  \bigg(\tilde{a}_j^{(1,1)}\cos 2\beta_j-\frac{\tilde{a}_j^{(1,2)}+\tilde{a}_j^{(2,1)}}{2}\sin 2\beta_j\bigg),\\
            \overline{a}^{(1,2)}_k&=\sum_{j=0}^{k-1}  \left(\tilde{a}_j^{(1,1)}\sin 2\beta_j+\tilde{a}_j^{(1,2)}\cos^2 \beta_j-\tilde{a}_j^{(2,1)}\sin^2\beta_j \right),\\
            \overline{a}^{(2,1)}_k&=\sum_{j=0}^{k-1}  \left(\tilde{a}_j^{(1,1)}\sin 2\beta_j-\tilde{a}_j^{(1,2)}\sin^2 \beta_j+\tilde{a}_j^{(2,1)}\cos^2\beta_j \right),\\
            \overline{a}^{(2,2)}_k&=- \overline{a}^{(1,1)}_k,
     \end{split}
        \end{align}
    where 
    \begin{align}
    \label{cotascoeficientes}
    \begin{split}
       |\tilde{a}_j^{(1,1)}(t)|\leq N_{j}^{\varepsilon-2\frac{\varepsilon+\delta}{1+\gamma+\delta}+\sigma}, \hspace{5mm}  |\tilde{a}_j^{(1,2)}(t)|\leq  N_{j}^{\varepsilon-4\frac{\varepsilon+\delta}{1+\gamma+\delta}+\sigma},\hspace{5mm}  
       |\tilde{a}_j^{(2,1)}(t)-N_j^{\varepsilon}|\leq  N_j^{\varepsilon-\sigma} ,
    \end{split}
    \end{align}
    for all $t\geq t_{j+1}$ and $j=0,\ldots, n$. 
    Assume also that for $k=1,\ldots, n$ and $t\geq t_{n+1}$, 
    \begin{align}
    \label{betaanterior}
          |\beta_{k}(t)-\beta_{k-1}(t)|&\leq N_{k-1}^{-\varepsilon+\mu}+N_{k-1}^{-\varepsilon+\mu-\sigma}, \hspace{5mm} 2\pi > \beta_1(t)>\cdots >\beta_n(t),\\
        \label{alphaanterior}
        |\alpha_{k}(t)-\beta_{k-1}(t)|&\leq  N_{k-1}^{\varepsilon-4\frac{\varepsilon+\delta}{1+\gamma+\delta}+\sigma},\\
        \label{sestatico}
        0\leq\frac{s'_n(t)}{s_n(t)}&\leq N_{n-1}^{\mu+\sigma}.
\end{align}
    Let $r_{n+1}, \alpha_{n+1}$ solve (\ref{dinamicaalpha}) and $s_{n+1}, \beta_{n+1}$ solve (\ref{dinamicabeta}) with coefficients $\overline{a}^{(1,1)}_{n+1},\overline{a}^{(1,2)}_{n+1},\overline{a}^{(2,1)}_{n+1},\overline{a}^{(2,2)}_{n+1}$ given by (\ref{defcoeficientes}). Assume the initial conditions
    \begin{align}\label{IC}
            r_{n+1}(t_{n+1})=s_{n+1}(t_{n+1})=N_{n+1}^{1-\frac{\varepsilon+\delta}{1+\gamma+\delta}}, \qquad \alpha_{n+1}(t_{n+1})= \beta_{n+1}(t_{n+1}) + \frac{\pi}{2} =\beta_n(t_{n+1}).  
    \end{align}
    Then, for every $t\in [t_{n+1},1]$,
    \begin{align}
    \label{estimaalpha2}
        |\alpha_{n+1}(t)-\beta_{n}(t)|&\leq  N_n^{\varepsilon-4\frac{\varepsilon+\delta}{1+\gamma+\delta}},\\
        \label{estimar2}
        \left|
        \frac{r_{n+1}(t)}{r_{n+1}(t_{n+1})}-1\right|&\leq N_n^{-\frac{\gamma \mu}{10}},\\
        \label{estimabeta2}
       \left| \beta_{n+1}(t)-\beta_{n}(t)+\arctan \bigg(\frac{1}{\int_{t_{n+1}}^t\mathcal{A}_n(\tau) \, d\tau}\bigg)\right|&\leq N_n^{-\varepsilon+\mu-\frac{\gamma \mu}{20}},\\
       \label{estimas2}
       \left|\frac{s_{n+1}(t)}{s_{n+1}(t_{n+1})}-\sqrt{1+\left(\int_{t_{n+1}}^t\mathcal{A}_n(\tau) \, d\tau\right)^2}\right|&\leq N_{n}^{-\frac{\gamma\mu}{40}}\sqrt{1+\left(\int_{t_{n+1}}^t\mathcal{A}_n(\tau) \, d\tau\right)^2},
    \end{align}
    where
    \begin{align}
    \label{acal}
        \mathcal{A}_n(t):=\sum_{j=0}^n (\tilde{a}_j^{(2,1)}+\tilde{a}_j^{(1,2)})(t). 
    \end{align}
    Moreover, 
    \begin{align}
    \label{sestatico2}
      0\leq  \frac{s'_{n+1}(t)}{s_{n+1}(t)}\leq N_n^{\mu+\sigma}, \qquad \forall t\in [t_{n+2},1].
    \end{align}
    \end{lma}
    \begin{proof}
    We need to control $\alpha_{n+1}, \beta_{n+1}, r_{n+1}$ and $ s_{n+1}$. Accordingly, we shall divide the proof in four steps, one for each quantity.
    \begin{enumerate}
        \item \textbf{Control of $\alpha_{n+1}$}. 
        We start by computing the ODE associated to $\alpha_{n+1}-\beta_n$. To get this, we use (\ref{dinamicaalpha}),  (\ref{dinamicabeta}) and \eqref{defcoeficientes}, 
        \begin{align*}
           &-( \alpha_{n+1}-\beta_n)'=\\&(\overline{a}_{n+1}^{(1,2)}-\overline{a}_{n}^{(1,2)}) \cos^2\alpha_{n+1}-(\overline{a}_{n+1}^{(2,1)}-\overline{a}_{n}^{(2,1)}) \sin^2\alpha_{n+1}-2(\overline{a}_{n+1}^{(1,1)}-\overline{a}_{n}^{(1,1)})\cos \alpha_{n+1}\sin\alpha_{n+1}\\
          &+\overline{a}_{n}^{(1,2)}(\cos^2 \alpha_{n+1}-\cos^2 \beta_n)- \overline{a}_{n}^{(2,1)}(\sin^2 \alpha_{n+1}-\sin^2 \beta_n)
           -2\overline{a}_{n}^{(1,1)}(\cos \alpha_{n+1}\sin\alpha_{n+1}-\cos \beta_{n}\sin\beta_{n})\\&=(\overline{a}_{n+1}^{(1,2)}-\overline{a}_{n}^{(1,2)}) \cos^2\alpha_{n+1}-(\overline{a}_{n+1}^{(2,1)}-\overline{a}_{n}^{(2,1)}) \sin^2\alpha_{n+1}-2(\overline{a}_{n+1}^{(1,1)}-\overline{a}_{n}^{(1,1)})\cos \alpha_{n+1}\sin\alpha_{n+1}\\
         &- (\overline{a}_{n}^{(2,1)}+\overline{a}_{n}^{(1,2)})\sin( \alpha_{n+1}+\beta_n)\sin( \alpha_{n+1}-\beta_n)
           -2\overline{a}_{n}^{(1,1)}\cos (\alpha_{n+1}+\beta_n)\sin(\alpha_{n+1}-\beta_n)\\
           &=(\overline{a}_{n+1}^{(1,2)}-\overline{a}_{n}^{(1,2)}) \cos^2\alpha_{n+1}-(\overline{a}_{n+1}^{(2,1)}-\overline{a}_{n}^{(2,1)}) \sin^2\alpha_{n+1}-2(\overline{a}_{n+1}^{(1,1)}-\overline{a}_{n}^{(1,1)})\cos \alpha_{n+1}\sin\alpha_{n+1}\\&
         -\sin( \alpha_{n+1}+\beta_n)\sin( \alpha_{n+1}-\beta_n) \sum_{j=0}^{n-1}  (2\tilde{a}_j^{(1,1)} \sin 2\beta_j+(\tilde{a}^{(1,2)}_j+\tilde{a}^{(2,1)}_j)\cos 2\beta_j)\\
         &-\cos (\alpha_{n+1}+\beta_n)\sin (\alpha_{n+1}-\beta_n) \sum_{j=0}^{n-1}  (2\tilde{a}_j^{(1,1)} \cos 2\beta_j-(\tilde{a}^{(1,2)}_j+\tilde{a}^{(2,1)}_j)\sin 2\beta_j)\\
        & =(\overline{a}_{n+1}^{(1,2)}-\overline{a}_{n}^{(1,2)}) \cos^2\alpha_{n+1}-(\overline{a}_{n+1}^{(2,1)}-\overline{a}_{n}^{(2,1)}) \sin^2\alpha_{n+1}-(\overline{a}_{n+1}^{(1,1)}-\overline{a}_{n}^{(1,1)})\sin 2\alpha_{n+1}\\
         &-\sin( \alpha_{n+1}-\beta_n) \sum_{j=0}^{n-1}  \left(2\tilde{a}_j^{(1,1)} \cos (\alpha_{n+1}+\beta_n-2\beta_j)+(\tilde{a}_j^{(1,2)}+\tilde{a}_j^{(2,1)})\sin (\alpha_{n+1}+\beta_n-2\beta_j)
         \right)\\
         &=\tilde{a}_n^{(1,2)}(\cos \alpha_{n+1}\cos \beta_n+\sin \alpha_{n+1}\sin\beta_n)^2\\
         &-\tilde{a}_n^{(2,1)}(\cos \alpha_{n+1}\sin \beta_n-\sin\alpha_{n+1}\cos \beta_n)^2-\tilde{a}_n^{(1,1)}\sin 2(\alpha_{n+1}-\beta_n)\\
          & -\sin( \alpha_{n+1}-\beta_n) \sum_{j=0}^{n-1}  \left(2\tilde{a}_j^{(1,1)} \cos (\alpha_{n+1}+\beta_n-2\beta_j)+(\tilde{a}_j^{(1,2)}+\tilde{a}_j^{(2,1)})\sin (\alpha_{n+1}+\beta_n-2\beta_j)
         \right)\\
         &=\tilde{a}_n^{(1,2)}\cos^2 (\alpha_{n+1}-\beta_n)-\tilde{a}_n^{(2,1)}\sin^2(\alpha_{n+1}- \beta_n)-\tilde{a}_n^{(1,1)}\sin 2(\alpha_{n+1}-\beta_n)\\
          & -\sin( \alpha_{n+1}-\beta_n) \sum_{j=0}^{n-1}  \left(2\tilde{a}_j^{(1,1)} \cos (\alpha_{n+1}+\beta_n-2\beta_j)+(\tilde{a}_j^{(1,2)}+\tilde{a}_j^{(2,1)})\sin (\alpha_{n+1}+\beta_n-2\beta_j)
         \right). 
        \end{align*}
    Hence, 
    \begin{align}
    \label{edoalpha}
    \begin{split}
        &(\alpha_{n+1}-\beta_n)'=\tilde{a}_n^{(2,1)}\sin^2(\alpha_{n+1}- \beta_n)-\tilde{a}_n^{(1,2)}\cos^2 (\alpha_{n+1}-\beta_n)+\tilde{a}_n^{(1,1)}\sin 2(\alpha_{n+1}-\beta_n)\\
           &+\sin( \alpha_{n+1}-\beta_n) \sum_{j=0}^{n-1}  \left(2\tilde{a}_j^{(1,1)} \cos (\alpha_{n+1}+\beta_n-2\beta_j)+(\tilde{a}_j^{(1,2)}+\tilde{a}_j^{(2,1)})\sin (\alpha_{n+1}+\beta_n-2\beta_j)
         \right)
         ,\\
         &(\alpha_{n+1}-\beta_n)(t_{n+1})=0,
         \end{split}
         \end{align}
         where we have also used \eqref{IC}.

   Since $(\alpha_{n+1}-\beta_n)(t_{n+1})=0$, by continuity,  there exists  $t_\ast>0$ such that
    \begin{align}\label{P30.6.1}
        |(\alpha_{n+1}-\beta_n)(t)|\leq  N_n^{\varepsilon-4\frac{\varepsilon+\delta}{1+\gamma+\delta}+\sigma}\log N_n
    \end{align}
    for all $t\in [t_{n+1},t_{n+1}+t_\ast]$. Then we are in a position to estimate $(\alpha_{n+1}-\beta_n)'$ on the time interval $[t_{n+1},t_{n+1}+t_\ast]$. To be more precise, it follows from \eqref{edoalpha} and \eqref{P30.6.1} that
    \begin{align*}
      &|(\alpha_{n+1}-\beta_n)'| =   \left|\tilde{a}_n^{(2,1)}\sin^2(\alpha_{n+1}- \beta_n)-\tilde{a}_n^{(1,2)}\cos^2 (\alpha_{n+1}-\beta_n)+\tilde{a}_n^{(1,1)}\sin 2(\alpha_{n+1}-\beta_n)\right.
         \nonumber  \\& \left.+\sin( \alpha_{n+1}-\beta_n) \sum_{j=0}^{n-1}  \left(2\tilde{a}_j^{(1,1)} \cos (\alpha_{n+1}+\beta_n-2\beta_j)+(\tilde{a}_j^{(1,2)}+\tilde{a}_j^{(2,1)})\sin (\alpha_{n+1}+\beta_n-2\beta_j)
         \right)
         \right| \nonumber\\
        &\leq |\tilde{a}_n^{(2,1)}| (N_n^{\varepsilon-4\frac{\varepsilon+\delta}{1+\gamma+\delta}+\sigma}\log N_n)^2+|\tilde{a}_n^{(1,2)}|+2|\tilde{a}_n^{(1,1)}|N_n^{\varepsilon-4\frac{\varepsilon+\delta}{1+\gamma+\delta}+\sigma}\log N_n \nonumber
        \\
       & +N_n^{\varepsilon-4\frac{\varepsilon+\delta}{1+\gamma+\delta}+\sigma}\log N_n\sum_{j=0}^{n-1}\left(      2|\tilde{a}_j^{(1,1)}|+|\tilde{a}_j^{(1,2)}+\tilde{a}_j^{(2,1)}|\left(|\alpha_{n+1}-\beta_n|+2\sum_{l=j}^{n-1}|\beta_{l+1}-\beta_l|\right)
        \right). 
\end{align*}
Note that, by \eqref{des}, 
\begin{equation}\label{P30.6.3}
    \sum_{l=j}^{\infty} N_l^{-\varepsilon + \mu} \leq C N_j^{-\varepsilon + \mu} \qquad \text{and} \qquad  \sum_{l=0}^j N_l^{\mu} \leq C N_j^{\mu}
\end{equation}
(see also Remark \ref{Remark7}) and then we can make use of the assumptions stated in \eqref{cotascoeficientes} and \eqref{betaanterior}, together with \eqref{P30.6.1}, in order  to get 
\begin{align*}
        |(\alpha_{n+1}-\beta_n)'| &\leq C\left(N_n^{3\varepsilon-8\frac{\varepsilon+\delta}{1+\gamma+\delta}+2\sigma}(\log N_n)^2+N_n^{\varepsilon-4\frac{\varepsilon+\delta}{1+\gamma+\delta}+\sigma}+N_n^{2\varepsilon-6\frac{\varepsilon+\delta}{1+\gamma+\delta}+2\sigma} \log N_n\right)\\
        & \hspace{.5cm}+CN_n^{\varepsilon-4\frac{\varepsilon+\delta}{1+\gamma+\delta}+\sigma}\log N_n \sum_{j=0}^{n-1} N_j^{\varepsilon} \left(N_j^{ - 2 \frac{\varepsilon + \delta}{1 + \gamma + \delta} + \sigma} + N_n^{\varepsilon - 4 \frac{\varepsilon + \delta}{1 + \gamma + \delta} + \sigma} \log N_n + N_j^{-\varepsilon + \mu}  \right) \\
        & \leq C N_n^{\varepsilon-4\frac{\varepsilon+\delta}{1+\gamma+\delta} + \sigma} \bigg(1 +\log N_n \sum_{j=0}^{n-1} N_j^{\mu} \bigg) \\
        &\leq CN_n^{\varepsilon-4\frac{\varepsilon+\delta}{1+\gamma+\delta}+\sigma}N_{n-1}^{\mu}\log N_n. 
         \end{align*}
Now, integrating the last estimate over all $t \in [t_{n+1}, t_{n+1}+t_\ast]$, using the definition of $t_{n+1}$ given in  \eqref{tiempos} and the relations  \eqref{P30.6.3}, we achieve
    \begin{align*}
        |(\alpha_{n+1}-\beta_n)(t)|&\leq CN_n^{\varepsilon-4\frac{\varepsilon+\delta}{1+\gamma+\delta}+\sigma}N_{n-1}^{\mu}\log N_n(t-t_{n+1})\\
        &\leq CN_n^{\varepsilon-4\frac{\varepsilon+\delta}{1+\gamma+\delta}+\sigma}N_{n-1}^{\mu}\log N_n(1-t_{n+1})\\
        &\leq CN_n^{\varepsilon-4\frac{\varepsilon+\delta}{1+\gamma+\delta}+\sigma}\left(\frac{N_{n-1}}{N_n}\right)^{\mu}\log N_n \\
        &\leq (CN_{n-1}^{-\frac{\gamma \mu}{2}+3\sigma}\log N_{n-1}) N_n^{\varepsilon-4\frac{\varepsilon+\delta}{1+\gamma+\delta}}
        \\
        &\leq (CN_{0}^{-\frac{\gamma \mu}{2}+3\sigma}\log N_{0})N_n^{\varepsilon-4\frac{\varepsilon+\delta}{1+\gamma+\delta}}. 
    \end{align*}
  Hence, taking $N_0$ sufficiently large (depending only on the parameters $\gamma, \delta, \varepsilon, \mu$ and $\sigma$), we can conclude that
    $$
        |(\alpha_{n+1}-\beta_n)(t)| \leq N_n^{\varepsilon-4\frac{\varepsilon+\delta}{1+\gamma+\delta}}, \qquad \forall t \in [t_{n+1}, t_{n+1}+t_\ast]. 
    $$
    Note that the latter consists of a  self-improvement of the a priori estimate \eqref{P30.6.1}. As a byproduct, it is possible to implement a standard bootstrapping argument so that 
    \begin{align*}
        |(\alpha_{n+1}-\beta_n)(t)|\leq N_n^{\varepsilon-4\frac{\varepsilon+\delta}{1+\gamma+\delta}}, \qquad  \forall t\in [t_{n+1},1],
    \end{align*}
    that is, \eqref{estimaalpha2} holds. 
    
    \item \textbf{Control of $\beta_{n+1}$.} Following the same approach as for (\ref{edoalpha}), but now replacing the role played by $\alpha_{n+1}$ in (\ref{edoalpha}) by $\beta_{n+1}$, one can derive the ODE associated to  $\beta_{n+1}-\beta_n$. Namely
         \begin{align*}
         \begin{split}
         &(\beta_{n+1}-\beta_n)'=\tilde{a}_n^{(2,1)}\sin^2(\beta_{n+1}- \beta_n)-\tilde{a}_n^{(1,2)}\cos^2 (\beta_{n+1}-\beta_n)+\tilde{a}_n^{(1,1)}\sin 2(\beta_{n+1}-\beta_n)\\
           &+\sin( \beta_{n+1}-\beta_n) \sum_{j=0}^{n-1}  \left(2\tilde{a}_j^{(1,1)} \cos (\beta_{n+1}+\beta_n-2\beta_j)+(\tilde{a}_j^{(1,2)}+\tilde{a}_j^{(2,1)})\sin (\beta_{n+1}+\beta_n-2\beta_j)
         \right), 
         \end{split}
    \end{align*}
    or equivalently, 
    \begin{align}\label{edobeta}
    (\beta_{n+1}-\beta_n)' &= \sin^2(\beta_{n+1}- \beta_n)\sum_{j=0}^n (\tilde{a}_j^{(2,1)}+\tilde{a}_j^{(1,2)})+\left(\tilde{a}_n^{(1,1)}\sin 2(\beta_{n+1}-\beta_n)-\tilde{a}_n^{(1,2)}\right) \nonumber \\
         &  +\sin( \beta_{n+1}-\beta_n) \sum_{j=0}^{n-1}  2\tilde{a}_j^{(1,1)} \cos (\beta_{n+1}+\beta_n-2\beta_j)
          \nonumber
         \\
         &+\sin(\beta_{n+1}-\beta_n)\sum_{j=0}^{n-1}
            (\tilde{a}_j^{(1,2)}+\tilde{a}_j^{(2,1)})\left(\sin (\beta_{n+1}-\beta_n+2(\beta_n-\beta_j))-\sin (\beta_{n+1}-\beta_n)
         \right)
    \end{align}
    together with the initial condition $(\beta_{n+1}-\beta_n)(t_{n+1})=-\frac{\pi}{2}$ (see \eqref{IC}).

    Next we define  some terms that will be useful during the proof,
    \begin{align}
    \label{defaux}
    \begin{split}
        &g_1(t,\beta_{n+1}):=\sin (\beta_{n+1}-\beta_n)\tilde{g}_1(t,\beta_{n+1}):=\sin (\beta_{n+1}-\beta_n)\left[ \sum_{j=0}^{n-1}  2\tilde{a}_j^{(1,1)} \cos (\beta_{n+1}+\beta_n-2\beta_j)
         \right.
         \\
         &+\left.2\tilde{a}_n^{(1,1)}\cos(\beta_{n+1}-\beta_n)+\sum_{j=0}^{n-1}
            (\tilde{a}_j^{(1,2)}+\tilde{a}_j^{(2,1)})\left(\sin (\beta_{n+1}-\beta_n+2(\beta_n-\beta_j))-\sin (\beta_{n+1}-\beta_n)
         \right)\right],\\
         &g_2(t,\beta_{n+1}):=-\tilde{a}_n^{(1,2)},\\
        &\mathcal{B}(t):=-\arctan \left(\frac{1}{\int_{t_{n+1}}^t\mathcal{A}_n(\tau)\, d\tau}\right), 
            \end{split}
    \end{align}
    where $\mathcal{A}_n$ is given by (\ref{acal}). For convenience, we rewrite the solution as
    \begin{align}\label{P7.7.1}
        (\beta_{n+1}-\beta_n)(t)=\mathcal{B}(t)+\mathcal{E}(t),
    \end{align}
    where $\mathcal{B}$ is defined in (\ref{defaux}) and $\mathcal{E}$ will be determined by (\ref{edobeta}). Noticing that $\mathcal{B}$ satisfies the ODE (apply the formula $\sin^2 (-\arctan x) = \frac{x^2}{1+x^2}$)
    \begin{align}
        \mathcal{B}'(t)&= \mathcal{A}_n(t)\sin^2 \mathcal{B}(t), \nonumber\\
        \mathcal{B}(t_{n+1})&= -\frac{\pi}{2}, \label{CIB}
    \end{align}
    plugging $\beta_{n+1}-\beta_n=\mathcal{B}+\mathcal{E}$ in (\ref{edobeta}), we obtain 
    \begin{align*}
      &\mathcal{A}_n \sin^2\mathcal{B}+\mathcal{E}'=\mathcal{A}_n\sin^2(\mathcal{B}+\mathcal{E})+g_{1}+g_2\\
      &=\mathcal{A}_n \sin^2\mathcal{B}+\mathcal{A}_n\sin (2\mathcal{B})\mathcal{E}+\mathcal{A}_n\cos(2\eta_1)\mathcal{E}^2
      +\sin(\mathcal{B}+\mathcal{E})\tilde{g}_{1}+g_2\\
      &=\mathcal{A}_n \sin^2\mathcal{B}+2\mathcal{A}_n \mathcal{E} \sin \mathcal{B}\cos \mathcal{B}+\mathcal{A}_n\cos(2\eta_1)\mathcal{E}^2
      +\sin\mathcal{B}\tilde{g}_{1}+\cos (\eta_2)\mathcal{E}\tilde{g}_{1}+g_2,
    \end{align*}
    where we did Taylor expansions (so $|\eta_i(t)-\mathcal{B}(t)|\leq |\mathcal{E}(t)|$ for $i=1, 2$). Then
    \begin{align}
    \label{edoE}
    \begin{split}
       & \mathcal{E}'=(2\mathcal{A}_n\mathcal{E}\cos \mathcal{B}+\tilde{g}_1)\sin \mathcal{B}+
        \mathcal{A}_n\cos(2\eta_1)\mathcal{E}^2+\cos (\eta_2)\mathcal{E}\tilde{g}_{1}+g_2,\\
       & \mathcal{E}(t_{n+1})=0.
        \end{split}
    \end{align}
    Next we prove that $\mathcal{E}$ will be negligible. To establish this, we first estimate $\tilde{g}_1$ and $g_2$, as defined in (\ref{defaux}). From (\ref{cotascoeficientes}), (\ref{betaanterior}),  the basic trigonometric identity $\sin A - \sin B = 2 \cos \frac{A+B}{2} \sin \frac{A-B}{2}$ and a standard telescopic argument for sums, we derive
    \begin{align}
    \label{cotasg}
    \begin{split}
        |\tilde{g}_1|=&\left|2\tilde{a}_n^{(1,1)}\cos(\beta_{n+1}-\beta_n)+ \sum_{j=0}^{n-1} \left[2\tilde{a}_j^{(1,1)} \cos (\beta_{n+1}+\beta_n-2\beta_j)\right.\right.\\
        &\left.\left.
         +
            (\tilde{a}_j^{(1,2)}+\tilde{a}_j^{(2,1)})\left(\sin (\beta_{n+1}-\beta_n+2(\beta_n-\beta_j))-\sin (\beta_{n+1}-\beta_n)
         \right)\right]\right|\\
         &\leq 2N_n^{\varepsilon-2\frac{\varepsilon+\delta}{1+\gamma+\delta}+\sigma}+C+C\sum_{j=0}^{n-1}N_j^\varepsilon |\beta_n-\beta_j|\leq C+C\sum_{j=0}^{n-1}N_j^\varepsilon \sum_{l=j}^{n-1}|\beta_{l+1}-\beta_l|\\
        & \leq C+C\sum_{j=0}^{n-1}N_j^\varepsilon \sum_{l=j}^{n-1} N_l^{-\varepsilon+\mu} \leq C + C \sum_{j=0}^{n-1}  N_j^{ \mu}\leq C N_{n-1}^{\mu},\\
         |g_2| &\leq  N_n^{\varepsilon-4\frac{\varepsilon+\delta}{1+\gamma+\delta}+\sigma}. 
    \end{split}
         \end{align}
          
          On the other hand, we claim that 
         \begin{align}
             \label{cotaA}
         N_{n}^{\varepsilon}-N_{n}^{\varepsilon-\frac{\sigma}{2}}\leq \mathcal{A}_n\leq N_{n}^{\varepsilon}+N_{n}^{\varepsilon-\frac{\sigma}{2}}.
    \end{align}
    Indeed, by \eqref{acal}, (\ref{cotascoeficientes}), and \eqref{des}, 
    \begin{align*}
        |\mathcal{A}_n - N_n^{\varepsilon}| & = \bigg| \sum_{j=0}^n (\tilde{a}_j^{(2, 1)} + \tilde{a}_j^{(1, 2)}) - N_n^{\varepsilon} \bigg| \\
        & \leq |\tilde{a}_n^{(2, 1)}-N_n^{\varepsilon}| + \sum_{j=0}^{n-1} |\tilde{a}_j^{(2, 1)}| + \sum_{j=0}^n |\tilde{a}_j^{(1, 2)}| \\
        & \leq N_n^{\varepsilon-\sigma} + C \sum_{j=0}^{n-1} N_j^{\varepsilon} + \sum_{j=0}^n N_j^{\varepsilon - 4 \frac{\varepsilon + \delta}{1+\gamma + \delta} + \sigma} \\
        & \leq N_n^{\varepsilon-\sigma} + C N_{n-1}^{\varepsilon} \leq  N_n^{\varepsilon-\sigma} + C N_n^{\frac{\varepsilon}{1+\frac{\gamma}{2}}} \\
        & \leq C  N_n^{\varepsilon-\sigma} \leq N_n^{\varepsilon - \frac{\sigma}{2}}. 
    \end{align*}
    As a consequence, the desired estimates \eqref{cotaA} hold.

    Since $\mathcal{E}(t_{n+1})=0$ (see \eqref{edoE}), there exists $t_\ast>0$ such that, 
    \begin{align}
    \label{cotaerror}
        |\mathcal{E}(t)|\leq N_{n}^{-\varepsilon+\mu}N_{n-1}^{-\frac{\gamma}{10}\mu}, \hspace{3mm}\forall t\in[t_{n+1},t_{n+1}+t_\ast].
    \end{align}
   In fact, take $t_\ast>0$ such that $t_{n+1}+t_\ast$ is the supremum in $(t_{n+1},1]$ verifying (\ref{cotaerror}). 
    

    In light  of the ODE associated to $\mathcal{E}$ given by (\ref{edoE}) and applying  (\ref{cotasg}), \eqref{cotaerror}, \eqref{CIB} (in particular, $\sin (2 \mathcal{B}) < 0$ provided that $t > t_{n+1}$ is sufficiently close to $t_{n+1}$) and \eqref{cotaA} (in particular, $\mathcal{A}_n > 0$), we infer that
    \begin{align*}
        |\mathcal{E}'|& \leq -\mathcal{A}_n\sin(2\mathcal{B})|\mathcal{E}|+CN_{n-1}^{\mu}+CN_n^\varepsilon(N_{n}^{-\varepsilon+\mu}N_{n-1}^{-\frac{\gamma}{10}\mu})^2\\
        &\hspace{1cm} +CN_{n-1}^{\mu}N_{n}^{-\varepsilon+\mu}N_{n-1}^{-\frac{\gamma}{10}\mu}+N_n^{\varepsilon-4\frac{\varepsilon+\delta}{1+\gamma+\delta}+\sigma} \\
        & \leq -\mathcal{A}_n\sin(2\mathcal{B})|\mathcal{E}|+CN_{n-1}^{\mu}. 
    \end{align*}
    We are now in a position to invoke the integral form of Gronwall's  inequality and then
    \begin{align}\label{P1.7.1}
        |\mathcal{E}(t)|\leq CN_{n-1}^{\mu} \int_{t_{n+1}}^{t}  \exp \left(-\int_{\tau}^{t} \mathcal{A}_n(\overline{\tau})\sin(2\mathcal{B}(\overline{\tau})) \, d\overline{\tau}\right)\, d\tau.
    \end{align}
    Let
    \begin{equation}\label{P2.7.4}
            \mathcal{I} (\tau) := \int_{t_{n+1}}^\tau \mathcal{A}_n(\xi) \, d\xi. 
    \end{equation}
    According to the definition of $\mathcal{B}$ given in \eqref{defaux} and applying the basic trigonometric formula 
    $
        \sin (\arctan x) \cos(\arctan x) = \frac{x}{1+x^2}
    $
    with $x = 1/\mathcal{I}$, we can write 
    \begin{align*}
       - \int_{\tau}^{t} \mathcal{A}_n(\overline{\tau})\sin(2\mathcal{B}(\overline{\tau})) \, d\overline{\tau} & =   \int_{\tau}^{t} \mathcal{A}_n(\overline{\tau})\sin\bigg(2 \arctan \frac{1}{\mathcal{I}(\overline{\tau})} \bigg) \, d\overline{\tau} \\
       & \hspace{-3cm}= 2 \int_{\tau}^{t} \mathcal{A}_n(\overline{\tau})\sin\bigg( \arctan \frac{1}{\mathcal{I}(\overline{\tau})} \bigg) \cos\bigg( \arctan \frac{1}{\mathcal{I}(\overline{\tau})} \bigg) \, d\overline{\tau} \\
       & \hspace{-3cm}= 2 \int_{\tau}^{t} \mathcal{A}_n(\overline{\tau}) \,  \frac{\mathcal{I}(\overline{\tau})}{1 + \mathcal{I}(\overline{\tau})^2} \, d\overline{\tau}.
    \end{align*}
    Then, by a simple change of variables (more precisely, $u = \mathcal{I}(\overline{\tau})$ and so $du  = \mathcal{A}_n(\overline{\tau}) d \overline{\tau}$), 
    \begin{align}\label{P7.7.6}
 - \int_{\tau}^{t} \mathcal{A}_n(\overline{\tau})\sin(2\mathcal{B}(\overline{\tau})) \, d\overline{\tau} & =  \int_{\mathcal{I}(\tau)}^{\mathcal{I}(t)}  \frac{2 u}{1+u^2} \, du = \log \bigg( \frac{1+\mathcal{I}(t)^2}{1+\mathcal{I}(\tau)^2} \bigg)
    \end{align}
    which leads to 
    \begin{align*}
         \exp \left(-\int_{\tau}^{t} \mathcal{A}_n(\overline{\tau})\sin(2\mathcal{B}(\overline{\tau})) \, d\overline{\tau}\right) = \frac{1+\mathcal{I}(t)^2}{1+\mathcal{I}(\tau)^2}. 
    \end{align*}
    Inserting now this into \eqref{P1.7.1} and using \eqref{cotaA}, we obtain 
    \begin{align}\label{P4.7.2}
        |\mathcal{E}(t)| \leq CN_{n-1}^{\mu} \int_{t_{n+1}}^{t} \frac{1+\mathcal{I}(t)^2}{1+\mathcal{I}(\tau)^2}  \, d\tau\leq
        CN_{n-1}^{\mu} \left(1+ \mathcal{I}(t)^2
        \right)(t-t_{n+1}) 
    \end{align}
    for all $t \in [t_{n+1}, t_{n+1}+t_\ast]$ and, in particular, 
    \begin{align} 
    \label{bootstrapineq}
        |\mathcal{E}(t)| &\leq C N_{n-1}^{\mu} (1+ \mathcal{I}(t_{n+1}+t_\ast)^2) t_\ast   \leq C N_{n-1}^{\mu} (1+ (N_n^{\varepsilon} t_\ast)^2 ) t_\ast, 
    \end{align}
    where we have used \eqref{cotaA} in the last estimate.

    Let $\tilde{t}_\ast > 0$ be such that 
    \begin{align}
\label{A1}
\mathcal{I}(t_{n+1}+\tilde{t}_\ast)=1.
    \end{align}
    The existence of $\tilde{t}_\ast$ is guaranteed by the fact that $\mathcal{A}_n (\tau) \geq  N_n^\varepsilon/C$ for all $\tau \geq t_{n+1}$ (see \eqref{cotaA}). In addition, by \eqref{cotaA}, it is easy to see that
    \begin{align}\label{P2.7.2}
            \frac{N_n^{-\varepsilon}}{C}\leq \tilde{t}_\ast\leq CN_n^{-\varepsilon}.
    \end{align}
   Next we shall prove that $\tilde{t}_\ast\leq t_\ast$. Assume the opposite (that is, $\tilde{t}_\ast> t_\ast$) and proceed by contradiction. According to (\ref{P2.7.2}), 
   \begin{align*}
       t_\ast<\tilde{t}_\ast\leq CN_{n}^{-\varepsilon}, 
   \end{align*}
   so, by (\ref{bootstrapineq}) and (\ref{des}), we have that
   \begin{align}\label{P4.7.1}
       |\mathcal{E}(t_{n+1}+t_\ast)|\leq CN_{n-1}^{\mu}N_n^{-\varepsilon}\leq N_n^{-\varepsilon+\mu} N_{n-1}^{-\frac{\gamma}{5}\mu},
   \end{align}
   where we have chosen $N_0$ big enough so that the last inequality holds. However,  notice that the bound for $|\mathcal{E}(t_{n+1}+t_\ast)|$ exhibited in \eqref{P4.7.1} is strictly smaller than the corresponding bound given by (\ref{cotaerror}), so this contradicts the condition of supremum involved in the definition of $t_\ast$  and hence we conclude by contradiction that $\tilde{t}_\ast\leq t_\ast$. In particular, $[t_{n+1}, t_{n+1}+\tilde{t}_\ast] \subset [t_{n+1}, t_{n+1}+t_\ast]$ and invoking \eqref{P4.7.2}, \eqref{A1} and \eqref{P2.7.2}, 
    \begin{equation}\label{P2.7.3}
        |\mathcal{E}(t)| \leq C N_{n-1}^{\mu} N_n^{-\varepsilon}   \leq  N_n^{-\varepsilon + \mu} N_{n-1}^{-\frac{\gamma}{5} \mu}, \qquad \forall t \in [t_{n+1}, t_{n+1}+\tilde{t}_\ast]. 
    \end{equation}


    Next we turn our attention to $\mathcal{E}$ on the remaining interval $[t_{n+1}+\tilde{t}_\ast,1]$. More precisely, we wish to prove that
    \begin{equation}\label{P5.7.1}
        |\mathcal{E}(t)| \leq N_{n}^{-\varepsilon+\mu}N_{n-1}^{-\frac{\gamma}{10}\mu}, \qquad \forall  t\in [t_{n+1}+\tilde{t}_\ast,1].  
    \end{equation}
     Assume that $t \in [t_{n+1}+\tilde{t}_\ast,1]$. Notice that, because of (\ref{A1}) and the fact that $\mathcal{A}_n\geq 0$,
    \begin{align*}
        \int_{t_{n+1}}^t \mathcal{A}_n(\tau) \, d\tau\in [1, \infty) \implies \mathcal{B}(t) \in \Big[-\frac{\pi}{4}, 0 \Big)
    \end{align*}
    and then
    \begin{align}\label{cotacoseno}
    \cos \mathcal{B}(t)\geq \frac{\sqrt{2}}{2}, \qquad \sin \mathcal{B}(t) < 0.
    \end{align}
   Assume that there exists $t^\ast \in [t_{n+1}+\tilde{t}_\ast,1]$ such that 
     \begin{align}\label{P3.7.3}
        |\mathcal{E}(t^\ast)|=N_{n}^{-\varepsilon+\mu}N_{n-1}^{-\frac{\gamma}{10}\mu};
\end{align}
otherwise,  in light of \eqref{P2.7.3}, the desired claim \eqref{P5.7.1} would trivially follow.  Notice that for such a  $t^\ast$, using the estimates  (\ref{cotacoseno}) and \eqref{cotaA}, 
    \begin{align}\label{P3.7.2}
        |2\mathcal{A}_n(t^\ast)\mathcal{E}(t^\ast)\cos \mathcal{B}(t^\ast)|\geq \frac{1}{C}N_n^{\mu}N_{n-1}^{-\frac{\gamma}{10}\mu}
    \end{align}
    and so, by  \eqref{cotasg} and \eqref{des}, 
    \begin{align}\label{P5.7.3}
        |2\mathcal{A}_n(t^\ast)\mathcal{E}(t^\ast)\cos \mathcal{B}(t^\ast)|  \geq \frac{N_{n-1}^{\frac{2}{5}\gamma\mu}}{C}|\tilde{g}_1(t^\ast)|\geq \frac{N_{0}^{\frac{2}{5}\gamma\mu}}{C}|\tilde{g}_1(t^\ast)| \geq 4|\tilde{g}_1(t^\ast)|
    \end{align}
    if $N_0$ is large enough. 
   Also using \eqref{cotaA} and  the fact that (see \eqref{tiempos}) $$1-t_{n+1} = N_n^{-\mu},$$ we derive, for every $t\in [ t_{n+1}+\tilde{t}_\ast, 1]$, 
   \begin{align*}
       |\sin \mathcal{B}(t)| \geq \frac{1}{2\int_{t_{n+1}}^t\mathcal{A}_n(\tau) \, d\tau} \geq \frac{1}{2\int_{t_{n+1}}^1\mathcal{A}_n(\tau) \, d\tau}\geq \frac{1}{C}N_n^{-\varepsilon+\mu}. 
   \end{align*}
   As a consequence (see \eqref{P3.7.2})
    \begin{align}\label{P5.7.2} 
        \left|2\mathcal{A}_n(t^\ast)\mathcal{E}(t^\ast)\cos \mathcal{B}(t^\ast)\sin \mathcal{B}(t^\ast)\right|\geq \frac{1}{C}N_n^{-\varepsilon+2\mu}N_{n-1}^{-\frac{\gamma }{10}\mu}.
    \end{align}
    
    Concerning the rest of the terms of the right hand side of (\ref{edoE}): It follows from \eqref{cotaA}, \eqref{P3.7.3},  \eqref{cotasg} and \eqref{P5.7.2}  that
    \begin{align}
&| \mathcal{A}_n(t^\ast)\cos(2\eta_1(t^\ast))\mathcal{E}(t^\ast)^2+\cos (\eta_2(t^\ast))\mathcal{E}(t^\ast)\tilde{g}_{1}(t^\ast)+g_2(t^\ast)| \nonumber \\ 
&\leq C N_n^\varepsilon \left(N_n^{-\varepsilon+\mu}N_{n-1}^{-\frac{\gamma }{10}\mu}\right)^2+CN_{n}^{-\varepsilon+\mu}N_{n-1}^{-\frac{\gamma}{10}\mu} N_{n-1}^{\mu}+N_n^{-2\frac{\varepsilon+\delta}{1+\gamma+\delta}} 
\leq  C N_{n-1}^{-\frac{\gamma}{10} \mu} N_n^{-\varepsilon+2\mu}N_{n-1}^{-\frac{\gamma }{10}\mu} \nonumber\\
&\leq C N_{n-1}^{-\frac{\gamma}{10} \mu}\left|2\mathcal{A}_n(t^\ast)\mathcal{E}(t^\ast)\cos \mathcal{B}(t^\ast)\sin \mathcal{B}(t^\ast)\right| \nonumber\\
&\leq C N_0^{-\frac{\gamma}{10}\mu}\left|2\mathcal{A}_n(t^\ast)\mathcal{E}(t^\ast)\cos \mathcal{B}(t^\ast)\sin \mathcal{B}(t^\ast)\right|
\leq \frac{1}{4}\left|2\mathcal{A}_n(t^\ast)\mathcal{E}(t^\ast)\cos \mathcal{B}(t^\ast)\sin \mathcal{B}(t^\ast)\right|,\label{P5.7.4}
    \end{align}
    provided that $N_0$ is large enough. Then, the ODE (\ref{edoE}) together with the previous bounds \eqref{P5.7.3} and \eqref{P5.7.4} guarantee the existence of $c\in\{\frac{1}{2},2\}$ (depending on the sign of $\mathcal{E}$) such that
    \begin{align*}
    \frac{1}{c}\mathcal{A}_n(t^\ast)\cos \mathcal{B}(t^\ast)\sin \mathcal{B}(t^\ast)\mathcal{E}(t^\ast)
     \leq
        \mathcal{E}'(t^\ast)
        \leq c\mathcal{A}_n(t^\ast)\cos \mathcal{B}(t^\ast)\sin \mathcal{B}(t^\ast)\mathcal{E}(t^\ast).
    \end{align*}
   Moreover, observe that the coefficient $\mathcal{A}_n(t^\ast)\cos \mathcal{B}(t^\ast)\sin \mathcal{B}(t^\ast)$ has a sign, specifically,  
    \begin{align*}
        \mathcal{A}_n(t^\ast)\cos \mathcal{B}(t^\ast)\sin \mathcal{B}(t^\ast)< 0
    \end{align*}
    (see \eqref{cotacoseno}). 
   Hence, since $\mathcal{E}'$ is comparable to a linear term in $\mathcal{E}$ with negative coefficient, $|\mathcal{E}|$ keeps trapped in the interval $[0,N_{n}^{-\varepsilon+\mu}N_{n-1}^{-\frac{\gamma}{10}\mu}]$ on $[t_{n+1}+\tilde{t}_\ast,1]$. We  conclude that \eqref{P5.7.1} holds. 

   The proof of \eqref{estimabeta2} can be achieved by putting together \eqref{P2.7.3} and \eqref{P5.7.1}.  
   
\item \textbf{Control of $r_{n+1}$.} By procedures analogous to those used in the derivation of  the ODEs for the angles $\alpha_{n+1}$ and $\beta_{n+1}$ (see (\ref{edoalpha}) and (\ref{edobeta})), we get
\begin{align}
  &\left(  \log \frac{r_{n+1}}{s_n}\right)'=\frac{r_{n+1}'}{r_{n+1}}-\frac{s_{n}'}{s_{n}} \nonumber\\
  &=-\overline{a}^{(1,1)}_{n+1}\cos^2\alpha_{n+1}- \overline{a}^{(2,2)}_{n+1}\sin^2\alpha_{n+1}-(\overline{a}^{(1,2)}_{n+1}+\overline{a}^{(2,1)}_{n+1})\cos\alpha_{n+1}\sin\alpha_{n+1} \nonumber\\
  &+\overline{a}^{(1,1)}_{n}\cos^2\beta_{n}+ \overline{a}^{(2,2)}_{n}\sin^2\beta_{n}+(\overline{a}^{(1,2)}_{n}+\overline{a}^{(2,1)}_{n})\cos\beta_{n}\sin\beta_{n} \nonumber\\
&= - (\overline{a}_{n+1}^{(1, 1)} - \overline{a}_{n}^{(1, 1)}) \cos 2 \alpha_{n+1}  - \frac{\overline{a}^{(1,2)}_{n+1}+\overline{a}^{(2,1)}_{n+1}}{2} \sin 2\alpha_{n+1} \nonumber \\  
&+2 \overline{a}_n^{(1,1)} \sin (\alpha_{n+1} + \beta_n) \sin(\alpha_{n+1}-\beta_n)  + (\overline{a}^{(1,2)}_{n}+\overline{a}^{(2,1)}_{n})\cos\beta_{n}\sin\beta_{n}  \nonumber \\
&= - \Big(\tilde{a}_n^{(1, 1)} \cos 2 \beta_n - \frac{\tilde{a}_n^{(1,2)} + \tilde{a}_n^{(2, 1)}}{2} \sin 2 \beta_n\Big) \cos 2 \alpha_{n+1} \nonumber\\
&-   \Big( \tilde{a}_n^{(1, 1)} \sin 2 \beta_n + \frac{\tilde{a}_n^{(1, 2)}+\tilde{a}_n^{(2, 1)}}{2}  \cos 2 \beta_n \Big) \sin 2 \alpha_{n+1} \nonumber \\
&- \sin (\alpha_{n+1}-\beta_n) \sum_{j=0}^{n-1} \Big( 2\tilde{a}_j^{(1, 1)} \sin 2 \beta_j \cos (\alpha_{n+1}+\beta_n) +(\tilde{a}_j^{(1, 2)}+\tilde{a}_j^{(2, 1)})  \cos 2 \beta_j \cos (\alpha_{n+1}+\beta_n)\Big) \nonumber \\
&+   \sin(\alpha_{n+1}-\beta_n) \sum_{j=0}^{n-1} \Big(2 \tilde{a}_j^{(1, 1)} \cos 2 \beta_j \sin(\alpha_{n+1}+\beta_n) - ( \tilde{a}_j^{(1, 2)} + \tilde{a}_j^{(2, 1)}) \sin 2 \beta_j \sin(\alpha_{n+1}+\beta_n) \Big) \nonumber\\
  &=-\tilde{a}_{n}^{(1,1)} \cos 2(\alpha_{n+1}-\beta_n)-\frac{\tilde{a}_n^{(1,2)}+\tilde{a}_n^{(2,1)}}{2}\sin 2(\alpha_{n+1}-\beta_n) \nonumber\\
 & +\sin(\alpha_{n+1}-\beta_n)\sum_{j=0}^{n-1}  \left(
2\tilde{a}_{j}^{(1,1)} \sin(\alpha_{n+1}+\beta_n-2\beta_j)
-(\tilde{a}_j^{(1,2)}+\tilde{a}_j^{(2,1)})\cos (\alpha_{n+1}+\beta_n-2\beta_j)
  \right). \label{P8.7.1}
\end{align}
From the assumptions (\ref{cotascoeficientes}) and (\ref{alphaanterior}), we obtain on the time interval $[t_{n+1}, 1]$: 
\begin{align*}
    \left|
 \left(  \log \frac{r_{n+1}}{s_n}\right)'
    \right|&\leq  N_n^{\varepsilon-2\frac{\varepsilon+\delta}{1+\gamma+\delta}+\sigma} +C N_n^\varepsilon |\alpha_{n+1}-\beta_n|
    +C|\alpha_{n+1}-\beta_n|
    \sum_{j=0}^{n-1} N_j^{\varepsilon} \\
    & \leq  N_n^{\varepsilon-2\frac{\varepsilon+\delta}{1+\gamma+\delta}+\sigma} +C N_n^\varepsilon |\alpha_{n+1}-\beta_n| \leq C N_n^{\varepsilon-2\frac{\varepsilon+\delta}{1+\gamma+\delta}+\sigma},
\end{align*} 
where we have also used \eqref{estimaalpha2} (which was already shown in the first step of this proof) for the last estimate. Then, integrating the last expression over  $[t_{n+1}, t]$ with $t < 1$: 
\begin{align*}
   \left|\log \frac{r_{n+1}(t)}{s_n(t)}- \log \frac{r_{n+1}(t_{n+1})}{s_n(t_{n+1})}\right|&\leq C N_n^{\varepsilon-2\frac{\varepsilon+\delta}{1+\gamma+\delta}+\sigma}(t-t_{n+1}) \\
   &\hspace{-4cm}\leq C N_n^{\varepsilon-2\frac{\varepsilon+\delta}{1+\gamma+\delta}+\sigma}(1-t_{n+1}) \leq C  N_n^{\varepsilon-2\frac{\varepsilon+\delta}{1+\gamma+\delta}+\sigma-\mu},
\end{align*}
where again the definition of $t_{n+1}$ as given by \eqref{tiempos} is applied to achieve the last estimate. As a byproduct, we have 
\begin{align}\label{P5.7.5}
   \exp {(-C N_n^{\varepsilon-2\frac{\varepsilon+\delta}{1+\gamma+\delta}+\sigma-\mu})}\leq \frac{r_{n+1}(t)}{r_{n+1}(t_{n+1})}\frac{s_{n}(t_{n+1})}{s_{n}(t)}\leq \exp {(C N_n^{\varepsilon-2\frac{\varepsilon+\delta}{1+\gamma+\delta}+\sigma-\mu})}.
\end{align}
In addition, we claim that 
\begin{align}\label{P6.7.1}
 1 \leq  \frac{s_n(t)}{s_n(t_{n+1})} \leq 1 + N_n^{-\frac{\gamma \mu}{5}}, \qquad \forall t \geq t_{n+1}. 
\end{align}
Indeed, the left-hand side estimate is a simple consequence of the fact that $s_n$ is increasing (see \ref{sestatico}). On the other hand, the proof of the right-hand side estimate in \eqref{P6.7.1} can be obtained as follows: By (\ref{sestatico}), the definition of $t_{n+1}$ as given in \eqref{tiempos}, \eqref{des} and (\ref{parametros1}), 
\begin{align*}
   \frac{s_n(t)}{s_n(t_{n+1})}\leq \exp{(N_{n-1}^{\mu+\sigma}(t-t_{n+1}))} \leq \exp{(N_{n-1}^{\mu +\sigma} N_n^{-\mu})} \leq 1+CN_{n-1}^{\mu+\sigma}N_n^{-\mu}\leq 1+N_{n}^{-\frac{\gamma \mu}{5}}.
\end{align*}
This concludes the proof of \eqref{P6.7.1}. 

Applying the second inequalities contained in both \eqref{P5.7.5} and \eqref{P6.7.1}, we have
\begin{align*}
    \frac{r_{n+1}(t)}{r_{n+1}(t_{n+1})}-1 &\leq \frac{s_n(t)}{s_n(t_{n+1})} \exp {(C N_n^{\varepsilon-2\frac{\varepsilon+\delta}{1+\gamma+\delta}+\sigma-\mu})} -1  \\
    & \leq (1+N_{n}^{-\frac{\gamma \mu}{5}}) \exp {(C N_n^{\varepsilon-2\frac{\varepsilon+\delta}{1+\gamma+\delta}+\sigma-\mu})}-1 \\
    & \leq (1+N_{n}^{-\frac{\gamma \mu}{5}}) (1+ C N_n^{\varepsilon-2\frac{\varepsilon+\delta}{1+\gamma+\delta}+\sigma-\mu}) - 1 \\
    & \leq  C N_n^{\varepsilon-2\frac{\varepsilon+\delta}{1+\gamma+\delta}+\sigma-\mu} \leq  N_n^{-\frac{\gamma \mu}{10}}. 
\end{align*}
In a similar fashion, but now applying the first inequalities from   \eqref{P5.7.5} and \eqref{P6.7.1}, one can  show that
$$
    1-\frac{r_{n+1}(t)}{r_{n+1}(t_{n+1})} \leq N_n^{-\frac{\gamma \mu}{10}}.
$$
Hence \eqref{estimar2} holds, that is, 
\begin{align*}
   \left| \frac{r_{n+1}(t)}{r_{n+1}(t_{n+1})}-1\right|\leq N_n^{-\frac{\gamma \mu}{10}}.
\end{align*}

\item \textbf{Control of $s_{n+1}$.} Observe that $s_{n+1}$ verifies an analogous ODE to the one that verifies $r_{n+1}$  after replacing the role played by $\alpha_{n+1}$ in \eqref{P8.7.1} by $\beta_{n+1}$. More precisely, 
\begin{align}\label{P7.7.4}
    \left(  \log \frac{s_{n+1}}{s_n}\right)'=g_3+g_4,
\end{align}
with (recall \eqref{acal})
\begin{align*}
g_3:=&-\sin 2(\beta_{n+1}-\beta_n)\sum_{j=0}^{n}\frac{\tilde{a}_j^{(1,2)}+\tilde{a}_j^{(2,1)}}{2}=-\frac{\mathcal{A}_n}{2}\sin 2(\beta_{n+1}-\beta_n),\\
   g_4:=&-\tilde{a}_{n}^{(1,1)} \cos 2(\beta_{n+1}-\beta_n)
  +\sin(\beta_{n+1}-\beta_n)\sum_{j=0}^{n-1} 
2\tilde{a}_{j}^{(1,1)} \sin(\beta_{n+1}+\beta_n-2\beta_j)\\
  &+\sin(\beta_{n+1}-\beta_n)\sum_{j=0}^{n-1} (\tilde{a}_j^{(1,2)}+\tilde{a}_j^{(2,1)}) \left(\cos(\beta_{n+1}-\beta_n)
-\cos (\beta_{n+1}-\beta_n+2(\beta_n-\beta_j))
  \right). 
\end{align*}

Using \eqref{cotascoeficientes} and \eqref{betaanterior}, the term $g_4$ can be estimated on $[t_{n+1}, 1]$ as follows:  
\begin{align}
    |g_4|&\leq N_n^{\varepsilon-2\frac{\varepsilon+\delta}{1+\gamma+\delta}+\sigma}+C|\sin(\beta_{n+1}-\beta_n)|\left(\sum_{j=0}^{n-1}N_j^{\varepsilon-2\frac{\varepsilon+\delta}{1+\gamma+\delta}+\sigma}+\sum_{j=0}^{n-1}N_j^{\varepsilon}\sum_{l=j}^{n-1}|\beta_{l+1}-\beta_l|\right) \nonumber\\
    &\leq N_n^{\varepsilon-2\frac{\varepsilon+\delta}{1+\gamma+\delta}+\sigma}+C
    |\sin(\beta_{n+1}-\beta_n)|\left(1+\sum_{j=0}^{n-1}N_j^{\varepsilon}N_{j}^{-\varepsilon+\mu}\right) \nonumber\\
   & \leq N_n^{\varepsilon-2\frac{\varepsilon+\delta}{1+\gamma+\delta}+\sigma}+C
    |\sin(\beta_{n+1}-\beta_n)|N_{n-1}^{\mu}.\label{cotag4}
\end{align}
Integrating the last expression over $[t_{n+1}, t]$, making use of the definition of $t_{n+1}$ (see \eqref{tiempos}) together with  \eqref{P7.7.1} and \eqref{defaux}, lead to 
\begin{align}
     \int_{t_{n+1}}^1 |g_4(\tau)|\, d\tau & \leq C N_n^{\varepsilon-2\frac{\varepsilon+\delta}{1+\gamma+\delta}+\sigma-\mu} 
     +CN_{n-1}^{\mu}N_{n}^{-\frac{\varepsilon}{2}} \nonumber\\
    & \hspace{1cm}+CN_{n-1}^{\mu}\int_{t_{n+1}+N_n^{-\frac{\varepsilon}{2}}}^1 \left(\frac{1}{\int_{t_{n+1}}^\tau \mathcal{A}_n(\overline{\tau}) \, d\overline{\tau}}+|\mathcal{E}(\tau)|\right) \, d\tau. \label{P7.7.2}
\end{align}
Next we estimate the last integral by using \eqref{cotaA} and  \eqref{estimabeta2}
\begin{align*}
    \int_{t_{n+1}+N_n^{-\frac{\varepsilon}{2}}}^1 \left(\frac{1}{\int_{t_{n+1}}^\tau \mathcal{A}_n(\overline{\tau}) \, d\overline{\tau}}+|\mathcal{E}(\tau)|\right) \, d\tau & \leq \left(\frac{1}{\int_{t_{n+1}}^{t_{n+1}+N_n^{-\frac{\varepsilon}{2}}} \mathcal{A}_n(\overline{\tau}) \, d \overline{\tau}} + N_n^{-\varepsilon+\mu-\frac{\gamma \mu}{20}} \right) (1-t_{n+1}-N_n^{-\frac{\varepsilon}{2}}) \\
    & \hspace{-5cm}\leq  \left(\frac{C}{ N_n^{\frac{\varepsilon}{2}}} + N_n^{-\varepsilon+\mu-\frac{\gamma \mu}{20}} \right) (1-t_{n+1}-N_n^{-\frac{\varepsilon}{2}}) \leq C N_n^{-\frac{\varepsilon}{2}} + C  N_n^{-\varepsilon+\mu-\frac{\gamma \mu}{20}} N_n^{-\mu} \leq C N_n^{-\frac{\varepsilon}{2}}
\end{align*}
and inserting this into \eqref{P7.7.2} we obtain
\begin{align}\label{P7.7.3}
    \int_{t_{n+1}}^1 |g_4(\tau)|\, d\tau \leq CN_n^{\varepsilon-2\frac{\varepsilon+\delta}{1+\gamma+\delta}+\sigma-\mu}+CN_{n-1}^{\mu}N_n^{-\frac{\varepsilon}{2}} \leq C (N_n^{\varepsilon-2\frac{\varepsilon+\delta}{1+\gamma+\delta}} + N_{n-1}^{\mu}N_n^{-\frac{\varepsilon}{2}}).
\end{align}

In view of \eqref{P7.7.4}, for $t\geq t_{n+1}$,
\begin{align*}
    \log \frac{s_{n+1}(t)}{s_n(t)}-\log \frac{s_{n+1}(t_{n+1})}{s_n(t_{n+1})}=\int_{t_{n+1}}^t \left(g_3(\tau)+g_4(\tau)\right) \, d\tau
\end{align*}
and then, by \eqref{P7.7.3}, 
\begin{align}
    \left|\log \left(\frac{s_{n+1}(t)}{s_{n+1}(t_{n+1})}\frac{s_{n}(t_{n+1})}{s_{n}(t)}\right)-\int_{t_{n+1}}^t g_3(\tau) \, d\tau\right| & \leq \int_{t_{n+1}}^t |g_4(\tau)| \, d \tau \nonumber \\
    &\leq C(N_n^{\varepsilon-2\frac{\varepsilon+\delta}{1+\gamma+\delta}}+N_{n-1}^{\mu}N_n^{-\frac{\varepsilon}{2}}).\label{P7.7.5}
\end{align}
It remains to study the integral of $g_3$: Applying successively  \eqref{P7.7.1}, \eqref{defaux}, \eqref{P7.7.6} and \eqref{P2.7.4}, 
\begin{align}
    \int_{t_{n+1}}^t g_3(\tau) \, d\tau &=-\int_{t_{n+1}}^t\frac{\mathcal{A}_n(\tau)}{2}\sin 2(\beta_{n+1}-\beta_n)(\tau) \, d\tau \nonumber \\
    &\hspace{-1.5cm}=-\int_{t_{n+1}}^t\frac{\mathcal{A}_n(\tau)}{2}\sin 2(\mathcal{B}(\tau)+\mathcal{E}(\tau)) \, d\tau \nonumber\\
    &\hspace{-1.5cm}=-\int_{t_{n+1}}^t\frac{\mathcal{A}_n(\tau)}{2}\sin 2\mathcal{B}(\tau) \, d\tau-\int_{t_{n+1}}^t\mathcal{A}_n(\tau)\mathcal{E}(\tau)\cos 2\eta(\tau) \, d\tau \nonumber\\
    &\hspace{-1.5cm} = \log \Big(\sqrt{1+\mathcal{I}(t)^2} \Big)
    -\int_{t_{n+1}}^t\mathcal{A}_n(\tau)\mathcal{E}(\tau)\cos 2\eta(\tau) \, d\tau, \label{P7.7.7}
\end{align}
where $|\eta(t)-\mathcal{B}(t)|\leq |\mathcal{E}(t)|$. Furthermore, using (\ref{cotaA}) and (\ref{estimabeta2}) (which was already established in the second step of this proof),
\begin{align}\label{P7.7.8}
    \left|\int_{t_{n+1}}^t\mathcal{A}_n(\tau)\mathcal{E}(\tau)\cos 2\eta(\tau) \, d\tau \right|\leq C N_n^{\varepsilon} N_n^{-\varepsilon + \mu}N_{n-1}^{-\frac{\gamma \mu}{10 }}N_n^{-\mu} \leq   CN_{n-1}^{-\frac{\gamma}{10}\mu},
\end{align}
where \eqref{des} is also applied in the last step. 
Hence, using \eqref{P7.7.5}, \eqref{P7.7.7} and \eqref{P7.7.8} and the standing assumptions on the parameters, 
\begin{align}\label{P7.7.9}
    \left|\log \left(\frac{s_{n+1}(t)}{s_{n+1}(t_{n+1})}\frac{s_{n}(t_{n+1})}{s_{n}(t)}\right)-\log \Big(\sqrt{1+\mathcal{I}(t)^2}\Big)\right| \leq CN_{n-1}^{-\frac{\gamma}{10}\mu}.  
\end{align} 
In addition, recall that (see \eqref{P6.7.1})
\begin{align}\label{P7.7.10}
    1 \leq \frac{s_n(t)}{s_n(t_{n+1})}\leq 1+N_{n}^{-\frac{\gamma \mu}{5}}, \qquad \forall t \geq t_{n+1}.
\end{align}
Thanks to \eqref{P7.7.9} and \eqref{P7.7.10}, we get
\begin{align*}
     \left|\log \left(\frac{s_{n+1}(t)}{s_{n+1}(t_{n+1}) \sqrt{1+\mathcal{I}(t)^2}}\right)\right| & \leq C N_{n-1}^{-\frac{\gamma}{10} \mu} + \log \big(1+ N_n^{-\frac{\gamma }{5} \mu} \big) \\
     & \leq C N_{n-1}^{-\frac{\gamma}{10} \mu} + C N_n^{-\frac{\gamma \mu}{5}} \leq C  N_{n-1}^{-\frac{\gamma}{10} \mu}.
\end{align*}
Combining this with \eqref{des}, we arrive at the desired estimate (\ref{estimas2}).

Finally, we are going to prove (\ref{sestatico2}). Suppose that $t\in[t_{n+2},1]$. From (\ref{cotaA}), \eqref{P7.7.1}, (\ref{estimabeta2}), \eqref{P2.7.4}, and \eqref{tiempos}, we derive
 \begin{align*}
     |g_3(t)|&\leq |\mathcal{A}_n(t)| |\beta_{n+1}(t)-\beta_n(t)| \leq CN_n^{\varepsilon} (|\mathcal{B}(t)| + |\mathcal{E}(t)|) \\
     & \leq C N_n^{\varepsilon} \left(\arctan \bigg(\frac{1}{\mathcal{I}(t_{n+2})}  \bigg) +   N_n^{-\varepsilon + \mu -\frac{\gamma \mu}{20}} \right) \\
     & \leq C N_n^{\varepsilon} N_n^{-\varepsilon + \mu }  = C N_n^{\mu}
 \end{align*}
 and on the other hand, by (\ref{cotag4}), 
 \begin{align*}
     |g_4(t)|\leq C N_{n-1}^{\mu}. 
 \end{align*}
 Using these bounds for $g_3$ and $g_4$ together with \eqref{P7.7.4} lead to  
 \begin{align*}
     \frac{s'_{n+1}(t)}{s_{n+1}(t)}-\frac{s'_{n}(t)}{s_{n}(t)}=\left(\log \frac{s_{n+1}}{s_{n}}\right)'(t)=g_3(t)+g_4(t)\leq CN_n^{\mu}
\end{align*}
and then, by the assumptions (\ref{sestatico}) and \eqref{des}, 
\begin{align*}
      \frac{s'_{n+1}(t)}{s_{n+1}(t)}\leq CN_n^{\mu}+N_{n-1}^{\mu+\sigma}\leq N_n^{\mu+\sigma}.
 \end{align*}
This gives the upper estimate in \eqref{sestatico2}. Concerning the lower estimate, we can apply \eqref{estimabeta2} and \eqref{cotaA} to establish  that, for $t \geq t_{n+2}$, 
\begin{align*}
    |(\beta_{n+1}-\beta_n)(t)| &\leq N_n^{-\varepsilon+\mu-\frac{\gamma \mu}{20}} + \arctan \bigg( \frac{1}{\mathcal{I}(t)} \bigg) \leq  C N_n^{-\varepsilon+\mu}
\end{align*}
and, in particular, 
$$
    |\cos (\beta_{n+1}-\beta_n)(t)| \geq 1/2. 
$$
Then (see also \eqref{cotag4} and \eqref{cotaA})
\begin{align*}
    |g_4| & \leq N_n^{\varepsilon-2\frac{\varepsilon+\delta}{1+\gamma+\delta}+\sigma}+C
    |\sin(\beta_{n+1}-\beta_n)| |\cos (\beta_{n+1}-\beta_n)| N_{n-1}^{\mu} \\
    & \leq  C
    |\sin2 (\beta_{n+1}-\beta_n)| (N_{n-1}^{\mu} + N_n^{2\varepsilon-2\frac{\varepsilon+\delta} {1+\gamma+\delta}+\sigma-\mu}) \\
    & \leq C |\sin2 (\beta_{n+1}-\beta_n)| N_n^{\varepsilon} (N_{n-1}^{\mu} N_n^{-\varepsilon} + N_n^{-\mu}) \\
    & \leq \frac{1}{2}   |\sin2 (\beta_{n+1}-\beta_n)| N_n^{\varepsilon}  (1-N_{n}^{-\frac{\sigma}{2}}) \\
    & \leq \frac{1}{2} |\sin2 (\beta_{n+1}-\beta_n)|  \mathcal{A}_n = |g_3| = g_3,  
\end{align*}
where we have also used that $\beta_{n+1} < \beta_n$ in the last step (which follows automatically from \eqref{estimabeta2}). As a byproduct, we have $g_3 + g_4 \geq 0$ and together with \eqref{sestatico} we obtain
 \begin{align*}
     \frac{s'_{n+1}(t)}{s_{n+1}(t)} = \frac{s'_{n}(t)}{s_{n}(t)} + g_3(t)+g_4(t)  \geq 0.
 \end{align*}
    \end{enumerate}
\end{proof}

\begin{rmk}\label{Rmk9}
        We claim that Lemma \ref{lemaedos} also holds if $n=0$ modulo obvious  modifications (in particular, the coefficients given in \eqref{defcoeficientes} are not defined in this case, as well as the set of assumptions \eqref{betaanterior}, \eqref{alphaanterior} and \eqref{sestatico} are not needed). The proof can be obtained following line by line the general case $n \in \mathbb{N}$ presented in the method of proof of Lemma \ref{lemaedos}; in fact, the proof is considerably simplified due to the fact that $r_0, \alpha_0, s_0, \beta_0$ do not depend on time (see \eqref{P14.7.2}).  
\end{rmk}

\subsection{Construction of the solution} \label{construccion}

Generally speaking, the strategy for the construction of the solution relies on a  multilayer method, where for each value of time $t \in [0, 1)$ the solution will be defined as a finite sum of layers  that arise successively as $t \nearrow 1$.  The dynamics associated to each layer will be determined by the first order of the Taylor expansion of the velocity generated by all the previous layers. Next we make all these general comments rigorous.  

The first layer of the construction is the stationary function
\begin{align}\label{FirstLayer}
    \theta_0(x,t):=N_0^{-1-\gamma+\varepsilon}\phi(x_2)\phi(N_0x_1), \qquad \forall t\in [0,1],
\end{align}
where the function $\phi$ was already introduced at the beginning of Section \ref{S3.2}. Notice that $\theta_0$ generates a sequence of derivatives in the sense of Definition \ref{DefSD} of amplitude $N_0^{-1-\gamma+\varepsilon}$, principal directions $(\frac{\pi}{2},0)$ and principal frequencies $(1,N_0)$. Namely, if we denote by $f_{0,0}^{(0,0)}=\theta_0$, then 
\begin{align}\label{P13.7.1}
    f_{0,0}^{(j,k)}(x,t)=N_0^{-1-\gamma+\varepsilon}\phi^{(j)}(x_2)\phi^{(k)}(N_0x_1), \qquad \forall t\in [0,1].
\end{align}

The $n$-th layer is defined as follows
\begin{align}
\label{deftheta}
    \theta_n(x,t):=\sum_{i=0}^{I_\gamma}f_{n,i}(x,t), \qquad  \forall n\geq 1,
\end{align}
that is, each layer $\theta_n$ consists of a  finite sum of a principal layer (say $f_{n,0}$) and $I_\gamma$ perturbations (say  $f_{n, i}$ with $1 \leq i \leq I_\gamma$). Here, we stress that  $I_\gamma$ will only depend on $\gamma$. Each layer will be introduced smoothly in time with the help of a function $h_n \in C^\infty ([0,1])$ to be defined such that
\begin{align}
    h_n(t)&=0, \qquad \forall t \in [0,t_n], \nonumber \\
    h_n(t)&=1, \qquad  \forall t \in [t_{n}+N_{n-1}^{-10\varepsilon},1], \label{P2.8.3}\\
   0\leq h'_n(t)& \leq 2N_{n-1}^{10\varepsilon}, \qquad  \forall t \in [0,1], \nonumber
\end{align}
where one can check that $t_n+N_{n-1}^{-10\varepsilon}\leq t_{n+1}$ in light of 
\begin{align*}
t_0=0, \hspace{10mm}
    t_n=1-N_{n-1}^{-\mu}, \hspace{3mm} \forall n\geq 1,
\end{align*}
and \eqref{parametros1}. We also let $h_0 \equiv 1$. 

To illustrate this mechanism, a sum of the type 
\begin{align}\label{AuxF}
  t \in [0, 1) \mapsto  \sum_{j=0}^\infty h_j(t) \theta_j(x,t),
\end{align}
consists in fact of a finite sum of layers at each time $t$. Indeed, since $\{t_n\}\nearrow 1$, there exists a unique  $n\in \mathbb{N}$ such that $t\in[t_n,t_{n+1})$, and then $h_j(t)=0$ for any $j\geq n+1$, or in other words, the layers $\theta_j$ with $j \geq n+1$ are not involved in \eqref{AuxF} on $[t_n, t_{n+1})$.

We are going to construct the principal layer $f_{n,0}$ in \eqref{deftheta} in such a way that it generates a sequence of derivatives of amplitude $N_n^{-1-\gamma+\varepsilon}$, principal directions $(\alpha_n(t),\beta_n(t))$ and principal frequencies $(r_n(t),s_n(t))$. More precisely, as we will see later, $f_{n, 0}$ will be of type
\begin{align}
\label{capaprincipal}
    f_{n,0}(x,t)=N_n^{-1-\gamma+\varepsilon}\phi(r_n(t)\textbf{n}(\alpha_n(t))\cdot x)\phi(s_n(t)\textbf{n}(\beta_n(t))\cdot x),
\end{align}
where the sequence  $\{N_n\}_{n \in \mathbb{N}_0}$ will be chosen adequately. In this case, it is easy to check that $f_{n, 0}$ generates a sequence of derivatives with desired features by putting $f_{n,0}^{(0,0)}=f_{n,0}$ and 
\begin{align}\label{P16.7.2}
    f_{n,0}^{(j,k)}(x,t)=N_n^{-1-\gamma+\varepsilon}\phi^{(j)}(r_n(t)\textbf{n}(\alpha_n(t))\cdot x)\phi^{(k)}(s_n(t)\textbf{n}(\beta_n(t))\cdot x).
\end{align}
According to  Definition \ref{sod} (of sequence of derivatives), the supports of $f_{n,0}^{(j,k)}$ are described by the principal frequencies and the principal directions. To simplify the notation, we define
\begin{align}\label{P10.8.1}
    \mathcal{S}_n(t):= \mathcal{S}(\alpha_n(t),\beta_n (t),r_n(t),s_n(t)), \qquad   \forall t\in [t_n,1].
\end{align}
The evolution of the principal  frequencies $r_n(t), s_n(t)$ and the principal directions $\alpha_n(t), \beta_n(t)$  (and hence the evolution of $f_{n,0}$, see \eqref{capaprincipal}) will be determined by the transport of the linear part of the Taylor expansion of the velocity $v^\gamma$  generated by the previous layers. Moreover, we prescribe the initial conditions (that is, when $t=t_n$) of the principal frequencies and directions by 
\begin{align}
\label{ic}
\begin{split}
    r_n(t_n)&=s_n(t_n)=N_n^{1-\frac{\varepsilon+\delta}{1+\gamma+\delta}},\\
    \alpha_n(t_n)&=\beta_n(t_n)+\frac{\pi}{2}=\beta_{n-1}(t_n),
\end{split}
\end{align}
for any $n\in \mathbb{N}$. Next we make this construction more precise. 

Let
\begin{align}
\label{sumacapas}
    \Theta_{n-1}(x, t):=\sum_{j=0}^{n-1} h_j(t)\theta_j(x, t), \hspace{3mm}\forall n\in \mathbb{N}.
\end{align}
Observe that  $\Theta_{n-1}$  follows the mechanism illustrated in \eqref{AuxF}, that is, the layers $\theta_j$ are subsequently  introduced in the construction of $\Theta_{n-1}$ as the time $t \nearrow 1$. Let $\overline{v}^\gamma(\Theta_{n-1})$ be the first order of the Taylor expansion (in the spatial coordinates)\footnote{Notice that $\Theta_0 = \theta_0$ is even (see \eqref{FirstLayer}) and then $v^\gamma(\Theta_0)$ is odd and its Taylor expansion is linear at first order. In fact, this argument can be inductively extended to $v^\gamma(\Theta_{n-1})$ with $n > 1$.} of $v^\gamma(\Theta_{n-1})$; however the precise definition of $\overline{v}^\gamma (\Theta_{n-1})$ is postponed. Assume momentarily that $\overline{v}^\gamma(\Theta_{n-1})$ is of the form
\begin{equation}\label{P9.7.1}
    \overline{v}^\gamma(\Theta_{n-1})(x, t) = \overline{A}_n(t) x, \hspace{3mm}\forall n\in \mathbb{N},
\end{equation}
where the entries of the matrix $\overline{A}_n$ are smooth functions. Then we define the principal layer $f_{n,0}$ as  the (unique) solution of the transport equation 
\begin{align}
\label{transporteppal}
    \partial_t f_{n,0}+\overline{v}^\gamma(\Theta_{n-1})\cdot \nabla f_{n,0}=0, \hspace{3mm} \forall n\in\mathbb{N},
\end{align}
with initial condition given 
\begin{equation}\label{P12.7.1}
    f_{n, 0}(x, t_n) = N_n^{-1-\gamma+\varepsilon}\phi(r_n(t_n)\textbf{n}(\alpha_n(t_n))\cdot x)\phi(s_n(t_n)\textbf{n}(\beta_n(t_n))\cdot x),
\end{equation}
where $r_n(t_n), s_n(t_n), \alpha_n(t_n)$ and $\beta_n(t_n)$ satisfy \eqref{ic}. Under  assumption \eqref{P9.7.1}, which will be justified later, one can invoke   Lemma \ref{dinamica} so  that $f_{n, 0}$ can be expressed as in \eqref{capaprincipal},  where  $r_n, \alpha_n$ verify (\ref{dinamicaalpha}) and $s_n, \beta_n$ verify (\ref{dinamicabeta}) with the initial conditions \eqref{ic}. Note also that the corresponding amplitude is given by (see \eqref{P12.7.1})
\begin{align}\label{P7.8.1}
\mathcal{C}_{n ,0} (t) = N_{n}^{-1-\gamma+\varepsilon}, \qquad  \forall t \geq t_n.
\end{align} 

On the other hand, we define the perturbation $f_{n, i}$ in \eqref{deftheta} with $i =1, \ldots, I_\gamma$  as the unique solution to the Cauchy problem
\begin{align}
\label{defpert}
\begin{split}
    &\partial_t f_{n,i}+\overline{v}^\gamma \left(\Theta_{n-1}\right) \cdot \nabla f_{n,i}=-h_n\sum_{(j_1,j_2)\in\mathcal{J}_{i-1}}v^\gamma(f_{n,j_1})\cdot \nabla f_{n,j_2}, \\
   & f_{n,i}(x,t_{n})=0, 
    \end{split}
\end{align} 
for any $n\in\mathbb{N}$. Here, we define
\begin{align*}
    \mathcal{J}_i=\{(j_1,j_2)\in \mathbb{N}_0^2: \max \, \{j_1,j_2\}=i\}, \hspace{3mm}\forall i=0,...,I_\gamma.
\end{align*}
In light of  Lemmas \ref{nuevaperturbacion} and  \ref{nuevaperturbacion2}, each $f_{n,i}$ will generate a sequence of derivatives of amplitudes $\mathcal{C}_{n, i}(t)$,  principal directions $(\alpha_n(t),\beta_n(t))$ and principal frequencies $(r_n(t),s_n(t))$. Note that $\mathcal{C}_{n, i}(t)$ are determined by  (\ref{nuevaamplitud}) and (\ref{intamplitud}) (in sharp contrast to the case $i=0$, the amplitudes $\mathcal{C}_{n, i}$ with $i =1, \ldots, I_\gamma$ are time dependent). It is clear (see \eqref{ic}) that Lemma \ref{nuevaperturbacion} can be applied, with $I=i-1$, $f_l=f_{n,l}$ for $l=0,\ldots,i-1$ and $\mathcal{J}=\mathcal{J}_{i-1}$, to check that the right-hand side in the transport equation \eqref{defpert} generates a sequence of derivatives with the desired amplitudes, principal frequencies and principal directions. Now, we can make use of Lemma \ref{nuevaperturbacion2} to conclude that the solution to \eqref{defpert} (that is, $f_{n,i}$) will also generate a sequence of derivatives. Notice that this result can be applied by taking $t_\ast=t_n$ and $(\alpha_n,\beta_n)$, $(r_n,s_n)$ as  principal directions and frequencies,  respectively. 
The hypothesis regarding the initial conditions is verified in light of \eqref{ic}. The hypothesis \eqref{flowmap} is also true. To prove it, one just have to check that the flow map of the transport equation \eqref{transporteppal} is indeed the one in \eqref{flowmap} with $t_\ast=t_n$, $(\alpha_n,\beta_n)$ and $(r_n,s_n)$. Now, since the Cauchy problem
\begin{align}
\label{flowmapaux}
\begin{split}
     \frac{dX_t(x)}{dt}&=\overline{v}^\gamma (\Theta_{n-1})(X_t(x), t),\\
     X_{t_n}(x)&=x,
\end{split}
\end{align}
has a unique solution, it is enough to check that $X_t$ given by \eqref{flowmap} solves it. Indeed, if one computes the time derivative of the expression given by \eqref{flowmap} taking into account that $(\alpha_n,\beta_n)$, $(r_n,s_n)$ verify \eqref{dinamicaalpha}, \eqref{dinamicabeta} from Lemma \ref{dinamica}, doing some trigonometric transformations and checking the initial conditions using \eqref{ic}, one gets that $X_t(x)$ given by \eqref{flowmap} solves 
\begin{align*}
    \frac{d X_t(x)}{dt}&=\overline{A} X_t(x), \\
    X_{t_n}(x)&=x,
\end{align*}
where the matrix $\overline{A}$ comes from the coefficients in the Lemma \ref{dinamica}. But notice that we chose $(\alpha_n,\beta_n)$, $(r_n,s_n)$ to solve the ODEs \eqref{dinamicaalpha}, \eqref{dinamicabeta} in Lemma \ref{dinamica} with the coefficients given by the relation $\overline{A}x=\overline{v}^\gamma (\Theta_{n-1})(x,t)$ (check \eqref{sumacapas}), so certainly $X_t(x)$ given by \eqref{flowmap} solves \eqref{flowmapaux}. Hence, we can apply Lemma \ref{nuevaperturbacion2}.

At this point, we need to make rigorous what is the definition of $\overline{v}^\gamma(\Theta_n)$ applied in the above construction. 
It is clear that the first order of the velocity generated by $\Theta_n$ can be computed with the help of Lemma \ref{lemavelocidad} (at least on the region $\mathcal{S}_n$). Indeed,  let us first focus on the linear velocity generated by $\theta_n$ (see \eqref{deftheta}), that is,  
\begin{align*}
    \overline{v}^\gamma(\theta_n)=\sum_{i=0}^{I_\gamma} \overline{v}^\gamma(f_{n,i}).
\end{align*}
Note that, for a fixed $n \in \mathbb{N}$,  the principal frequencies and the principal directions of the sequences of derivatives generated by each $f_{n,i}$ with $i = 0, \ldots, I_\gamma$ are the same. Then we define
       \begin{align*}
            \overline{v}^\gamma(f_{n,i})=&s_n^{1+\gamma}  \mathcal{C}_{n,i}\left[ a_{n,i}^{(2,1)} \left(\begin{matrix}
              -\sin \beta_n \cos \beta_n&-\sin^2\beta_n\\
              \cos^2 \beta_n&\sin \beta_n \cos \beta_n 
          \end{matrix}\right) + a_{n,i}^{(1,2)} \left(\begin{matrix}
              -\sin \beta_n \cos \beta_n&\cos^2\beta_n\\
              -\sin^2 \beta_n&\sin \beta_n \cos \beta_n 
          \end{matrix}\right)\right. \\
          &\left. + a_{n,i}^{(1,1)} \left(\begin{matrix}
              \cos 2\beta_n& \sin 2\beta_n \\
              \sin 2\beta_n  &-\cos 2 \beta_n
          \end{matrix}\right)  \right]x, \hspace{4mm}\forall x\in \mathbb{R}^2
        \end{align*}
and also for all $x\in \mathbb{R}^2$,
\begin{align*}
    &\overline{v}^\gamma(\theta_n)=s_n^{1+\gamma} \left[\left(\sum_{i=0}^{I_\gamma}\mathcal{C}_{n,i}a^{(2,1)}_{n,i}\right) \left(\begin{matrix}
              -\sin \beta_n \cos \beta_n&-\sin^2\beta_n\\
              \cos^2 \beta_n&\sin \beta_n \cos \beta_n 
          \end{matrix}\right) \right. \\
         & \left.
          +\left(\sum_{i=0}^{I_\gamma}\mathcal{C}_{n,i}a^{(1,2)}_{n,i}\right) \left(\begin{matrix}
              -\sin \beta_n \cos \beta_n&\cos^2\beta_n\\
              -\sin^2 \beta_n&\sin \beta_n \cos \beta_n 
          \end{matrix}\right)
          +\left(\sum_{i=0}^{I_\gamma}\mathcal{C}_{n,i}a^{(1,1)}_{n,i}\right) \left(\begin{matrix}
              \cos 2\beta_n& \sin 2\beta_n \\
              \sin 2\beta_n  &-\cos 2 \beta_n
          \end{matrix}\right)
          \right]x,
\end{align*}
where $a_{n,i}^{(j,k)}$ are the coefficients associated to  $f_{n,i}$ given by \eqref{coeficientesvelocidad}. Setting 
    \begin{align}
        \label{defatilde}
            \tilde{a}_n^{(j,k)}:=s_n^{1+\gamma}\sum_{i=0}^{I_\gamma} \mathcal{C}_{n,i} a_{n,i}^{(j,k)} \qquad \text{if} \qquad n \in \mathbb{N} 
        \end{align}
        and
        \begin{align}\label{defatildeVar}
\tilde{a}_0^{(j,k)}:=s_0^{1+\gamma}  \mathcal{C}_{0,0} a_{0,0}^{(j,k)} = N_0^{1+\gamma} N_0^{-1-\gamma+\varepsilon} a^{(j, k)}_{0, 0} = N_0^{\varepsilon} a^{(j, k)}_{0, 0},
        \end{align}
one can rewrite $\overline{v}^\gamma(\theta_n)$ in a more compact fashion  as 
         \begin{align}
         \label{vlin}
         \begin{split}
             \overline{v}^\gamma(\theta_{n})=&\left[ \tilde{a}_{n}^{(2,1)} \left(\begin{matrix}
              -\sin \beta_n \cos \beta_n&-\sin^2\beta_n\\
              \cos^2 \beta_n&\sin \beta_n \cos \beta_n 
          \end{matrix}\right) +\tilde{a}_{n}^{(1,2)} \left(\begin{matrix}
              -\sin \beta_n \cos \beta_n&\cos^2\beta_n\\
              -\sin^2 \beta_n&\sin \beta_n \cos \beta_n 
          \end{matrix}\right)\right.\\
          &\left.
          + \tilde{a}_{n}^{(1,1)} \left(\begin{matrix}
              \cos 2\beta_n& \sin 2\beta_n \\
              \sin 2\beta_n  &-\cos 2 \beta_n
          \end{matrix}\right)
          \right]x, \hspace{4mm} \forall x\in \mathbb{R}^2. 
         \end{split}
        \end{align}
    We now introduce the coefficients $\overline{a}_n^{(j, k)}$ as follows
        \begin{align}
        \label{abarra}
        \begin{split}             \overline{a}^{(1,1)}_n
            &= \sum_{j=0}^{n-1}  \left(\tilde{a}_j^{(1,1)}\cos 2\beta_j-\frac{\tilde{a}_j^{(1,2)}+\tilde{a}_j^{(2,1)}}{2}\sin 2\beta_j\right),\\
            \overline{a}^{(1,2)}_n&=\sum_{j=0}^{n-1}  \left(\tilde{a}_j^{(1,1)}\sin 2\beta_j+\tilde{a}_j^{(1,2)}\cos^2 \beta_j-\tilde{a}_j^{(2,1)}\sin^2\beta_j \right),\\
            \overline{a}^{(2,1)}_n&=\sum_{j=0}^{n-1}  \left(\tilde{a}_j^{(1,1)}\sin 2\beta_j-\tilde{a}_j^{(1,2)}\sin^2 \beta_j+\tilde{a}_j^{(2,1)}\cos^2\beta_j \right),\\
            \overline{a}^{(2,2)}_n&=- \overline{a}^{(1,1)}_n.
                    \end{split}
        \end{align}
          Armed with these coefficients,  if 
        \begin{align*}
       \overline{A}_n=\left(\begin{matrix}
                \overline{a}_n^{(1,1)}& \overline{a}_n^{(1,2)}\\
                \overline{a}_n^{(2,1)}&\overline{a}_n^{(2,2)}
            \end{matrix}\right),
        \end{align*}
        by \eqref{sumacapas}, (\ref{vlin}) and using that $h_j \equiv 1$ on $[t_{j+1}, 1]$, one has that 
        \begin{align}\label{P13.7.6}
        \overline{v}^\gamma(\Theta_{n-1})=\sum_{j=0}^{n-1} h_j \overline{v}^\gamma(\theta_j) = \sum_{j=0}^{n-1} \overline{v}^\gamma(\theta_j) =\overline{A}_n x, \qquad  \text{on} \quad  (x,t)\in \mathbb{R}^2\times [t_n,1]. 
        \end{align}
        This proves the desired claim \eqref{P9.7.1} and, as a byproduct, $\alpha_n$, $r_n$ solve (\ref{dinamicaalpha}) and $\beta_n$, $s_n$ solve (\ref{dinamicabeta}) with the  coefficients $\overline{a}_n^{(j,k)}$ given by \eqref{abarra}.

        Next we deal with the coefficients $b_{n, i}^{(j,k)}$ and $\tilde{b}_{n, i}^{(j,k)}$ related to the cubic Taylor remainder of the velocity $v^\gamma(f_{n, i})$, see \eqref{formulavelocidad}. Specifically, 
        \begin{align*}
            v^\gamma(f_{n,i})-\overline{v}^\gamma(f_{n,i})&=-\textbf{n}(\alpha_n)^\perp \frac{r_n}{2s_n^{1-\gamma}} \sum_{l=0}^3b_{n,i}^{(4-l,l)}\left(\begin{matrix}
                    3\\ l
                \end{matrix}\right)
                \left(r_n \textbf{n}(\alpha_n) \cdot x\right)^{3-l} (s_n\textbf{n}(\beta_n)\cdot x)^l \\
&\hspace{-15mm}- \textbf{n}(\beta_n)^\perp  \frac{s_n^\gamma}{2}\sum_{l=0}^3b_{n,i}^{(3-l,l+1)}\left(\begin{matrix}
                    3\\ l
                \end{matrix}\right)
                \left(r_n \textbf{n}(\alpha_n) \cdot x\right)^{3-l} (s_n\textbf{n}(\beta_n)\cdot x)^l, \hspace{4mm}  \forall (x,t)\in \mathcal{S}_n(t)\times [t_n,1]. 
        \end{align*}
        Here, recall that $b_{n, i}^{(j,k)}$ has both spatial and time dependence, see \eqref{Defbjk}. Moreover, in a similar fashion as for the linear coefficients $\tilde{a}_n^{(j, k)}$ (see \eqref{defatilde}), we  define 
        \begin{align}\label{P6.8.2}
            \tilde{b}_{n}^{(j,k)}:=\sum_{i=0}^{I_\gamma} b_{n,i}^{(j,k)} \qquad \text{if} \qquad n \in \mathbb{N}
        \end{align}
        and
        \begin{align*}
        \tilde{b}^{(j, k)}_0 := b_{0, 0}^{(j, k)}, 
        \end{align*}
        in such a way that 
        \begin{align}
            v^\gamma(\theta_n)-&\overline{v}^\gamma(\theta_n)=-\textbf{n}(\alpha_n)^\perp \frac{r_n}{2 s_n^{1-\gamma}} \sum_{l=0}^3\tilde{b}_{n}^{(4-l,l)}\left(\begin{matrix}
                    3\\ l
                \end{matrix}\right)
                \left(r_n \textbf{n}(\alpha_n) \cdot x\right)^{3-l} (s_n\textbf{n}(\beta_n)\cdot x)^l \nonumber \\
- &\textbf{n}(\beta_n)^\perp \frac{s_n^\gamma }{2}\sum_{l=0}^3\tilde{b}_{n}^{(3-l,l+1)}\left(\begin{matrix}
                    3\\ l
                \end{matrix}\right)
                \left(r_n \textbf{n}(\alpha_n) \cdot x\right)^{3-l} (s_n\textbf{n}(\beta_n)\cdot x)^l, \hspace{4mm}  \forall (x,t)\in \mathcal{S}_n(t)\times [t_n,1]. \label{P2.8.1}
        \end{align}


To define the sequence $\{N_n\}_{n \in \mathbb{N}_0}$ in \eqref{capaprincipal}, we choose a set of positive numbers verifying (\ref{des}). To be more precise, we let
 \begin{align}
 \label{defN}
N_{n}^{\varepsilon}=s_n(t_{n+1})^{1+\gamma}\mathcal{C}_{n,0}a^{(2,1)}_{n,0}(t_{n+1}),\qquad  \forall n\in \mathbb{N},
 \end{align}
 and then we will prove that this definition is indeed compatible with (\ref{des}). The careful reader would notice that definitions of the parameters $\mathcal{C}_{n,0}$, $a^{(2, 1)}_{n,0}$ and $s_n$ involve $N_n$ itself. Accordingly, in those expressions, $N_n$ may be initially taken as a free  parameter verifying \eqref{des} that will be eventually fixed in a constructive way (in terms of $N_{n-1}$) via the expression \eqref{defN}. In particular, the following lemma shows that such a sequence $\{N_n\}_{n \in \mathbb{N}_0}$ always exists.

 \begin{lma}\label{Lemma6}
There exists $N_n$ solving (\ref{defN}) such that, if  $\Theta_n$ is defined as in (\ref{sumacapas}) for $n \in \mathbb{N}_0$, then for every   $t\in [t_{n+1},1]$ and $x \in \mathcal{S}_n(t)$, 
     \begin{align}
         \label{velocidadTheta}
         \begin{split}
v^\gamma(\Theta_n)=&\sum_{j=0}^n 
\tilde{a}_j^{(2,1)}\left(\begin{matrix}
              -\frac{1}{2}\sin 2\beta_j&-\sin^2\beta_j\\
              \cos^2 \beta_j&\frac{1}{2}\sin 2\beta_j 
          \end{matrix}\right)x\\
          &+\sum_{j=0}^n \left[   \tilde{a}^{(1,1)}_j \left(\begin{matrix}
              \cos 2\beta_j& \sin 2\beta_j \\
              \sin 2\beta_j  &-\cos 2 \beta_j
          \end{matrix}\right)
          +\tilde{a}_j^{(1,2)} \left(\begin{matrix}
              -\frac{1}{2}\sin 2\beta_j &\cos^2\beta_j\\
              -\sin^2 \beta_j&\frac{1}{2}\sin 2\beta_j 
          \end{matrix}\right)\right]x
           \\&-\sum_{j=0}^n\textbf{n}(\alpha_j)^\perp \frac{r_j}{2 s_j^{1-\gamma}} \sum_{l=0}^3 \tilde{b}_j^{(4-l,l)}\left(\begin{matrix}
                    3\\ l
                \end{matrix}\right)
                \left(r_j \textbf{n}(\alpha_j) \cdot x\right)^{3-l} (s_j\textbf{n}(\beta_j)\cdot x)^l \\
&- \sum_{j=0}^n\textbf{n}(\beta_j)^\perp  \frac{s_j^\gamma}{2}\sum_{l=0}^3 \tilde{b}_{j}^{(3-l,l+1)}\left(\begin{matrix}
                    3\\ l
                \end{matrix}\right)
                \left(r_j \textbf{n}(\alpha_j) \cdot x\right)^{3-l} (s_j\textbf{n}(\beta_j)\cdot x)^l,  
         \end{split}
     \end{align}
     with 
     \begin{align}
     \label{cotascoeficientestilde}
         |\tilde{a}_j^{(1,1)}(t)|\leq N_j^{\varepsilon-2\frac{\varepsilon+\delta}{1+\gamma+\delta}+\sigma}, \qquad  |\tilde{a}_j^{(1,2)}(t)|\leq N_j^{\varepsilon-4\frac{\varepsilon+\delta}{1+\gamma+\delta}+\sigma}, 
     \end{align}
     \begin{align}
     \label{coeficienteppal}
         |\tilde{a}_j^{(2,1)}(t)-N_j^\varepsilon|\leq N_j^{\varepsilon-3\sigma},
     \end{align}
     and 
     \begin{align}
     \label{cotab}
         |\tilde{b}_j^{(l,k)}(x,t)|\lesssim_{\gamma, l, k} N_j^{-1-\gamma+\varepsilon}, 
     \end{align}
     where \eqref{cotascoeficientestilde}, \eqref{coeficienteppal} and \eqref{cotab} hold for all $t\in[t_{j+1},1], \, j =0, \ldots, n$. 
    Moreover, principal directions and frequencies of the layers verify, for all $t \in [t_{n+1}, 1]$, 
     \begin{align}
     \label{estimaalpha}
        |\alpha_{n+1}(t)-\beta_{n}(t)|&\leq  N_n^{\varepsilon-4\frac{\varepsilon+\delta}{1+\gamma+\delta}}, \\
        \label{estimar}
        \left|
        \frac{r_{n+1}(t)}{r_{n+1}(t_{n+1})}-1\right|&\leq N_n^{-\frac{\gamma\mu}{10}},\\
        \label{estimabeta}
       \left| \beta_{n+1}(t)-\beta_{n}(t)+\arctan\left(\frac{1}{N_n^\varepsilon(t-t_{n+1})}\right)\right|&\leq \frac{N_n^{-2\sigma}}{1+N_n^\varepsilon(t-t_{n+1})},\\
       \label{estimas}
       \left|\frac{s_{n+1}(t)}{s_{n+1}(t_{n+1})}-\sqrt{1+N^{2\varepsilon}_n(t-t_{n+1})^2}\right|&\leq N_n^{-2
\sigma}\sqrt{1+N^{2\varepsilon}_n(t-t_{n+1})^2},
    \end{align} 
    and if in addition $t \in [t_{n+2}, 1]$ then
    \begin{align}
    \label{estimabetatlargo}
        \left| \beta_{n+1}(t)-\beta_{n}(t)+N_n^{-\varepsilon+\mu}\right|&\leq  N_n^{-\varepsilon+\mu-\sigma}, \\
         \label{estimastlargo}\left|\frac{s_{n+1}(t)}{s_{n+1}(t_{n+1})}-N^{\varepsilon-\mu}_n\right|&\leq N_n^{\varepsilon-\mu-\sigma},\qquad  0\leq  \frac{s'_{n+1}(t)}{s_{n+1}(t)}\leq N_n^{\mu+\sigma}.
    \end{align}
     Furthermore, $\{N_n\}_{n \in \mathbb{N}_0}$ verifies 
      \begin{align}
    \label{crecimientoN}
         \frac{1}{2}  N_n^{(1+\gamma+\delta)\frac{\varepsilon-\mu}{\varepsilon+\delta}} 
        \leq 
        N_{n+1}
        \leq 2 N_n^{(1+\gamma+\delta)\frac{\varepsilon-\mu}{\varepsilon+\delta}}, \qquad \forall n\in \mathbb{N}_0.
    \end{align}
    \end{lma}
    \begin{rmk}\label{Remark25}
        It is easy to see that  (\ref{crecimientoN}) implies (\ref{des}). Indeed, the fact that 
        \begin{align*}
            \frac{\varepsilon-\mu}{\varepsilon+\delta}\leq 1-\frac{\mu}{\varepsilon+\delta} \qquad \text{and} \qquad  \delta \leq \frac{\gamma}{2}
        \end{align*}
        yields that
        \begin{align*}
            (1+\gamma+\delta) \frac{\varepsilon-\mu}{\varepsilon+\delta}\leq 1+\frac{3\gamma}{2}-\frac{\mu}{2\varepsilon}, 
        \end{align*}
        so  
        \begin{align*}
            N_{n+1}\leq 2N_0^{-\frac{\mu}{2\varepsilon}}N_n^{1+\frac{3\gamma}{2}}\leq N_n^{1+\frac{3\gamma}{2}}
        \end{align*}
        choosing $N_0$ large enough. On the  other hand, from
        \begin{align*}
            \frac{\varepsilon-\mu}{\varepsilon+\delta}\geq \frac{\varepsilon-\frac{\gamma\varepsilon}{50}}{\varepsilon+\frac{\gamma \varepsilon}{5}}=1-\frac{\frac{1}{5}+\frac{1}{50}}{1+\frac{\gamma}{5}}\gamma\geq 1-\frac{11}{50}\gamma \qquad \text{and} \qquad  \delta<1-\gamma,
        \end{align*}
        we derive
        \begin{align*}
            (1+\gamma+\delta)\frac{\varepsilon-\mu}{\varepsilon+\delta}\geq 1+\gamma-\frac{22}{50}\gamma,
        \end{align*}
        and then
        \begin{align*}
            N_{n+1}\geq \frac{N_0^{\frac{3 \gamma}{50}}}{2}N_n^{1+\frac{\gamma}{2}}\geq N_n^{1+\frac{\gamma}{2}},
        \end{align*}
        if  $N_0$ is sufficiently large. 
    \end{rmk}
     \begin{proof}[Proof of Lemma \ref{Lemma6}]
     We will proceed by induction on $n$. For each $n$, the proof will be divided into three  steps: Firstly, we  estimate the velocity of a layer; secondly, we describe the dynamics of the next layer; thirdly, we  prove that there exists $N_{n+1}$ satisfying  \eqref{defN} (and in particular \eqref{crecimientoN}, see Remark \ref{Remark25}).  
     
     We start with the base case $n=0$.
      \begin{enumerate}
         \item \textbf{Computing the velocity generated by $\Theta_0$.} Recall that (see \eqref{FirstLayer} and  \eqref{sumacapas})
        \begin{align*}
    \Theta_0(x,t)=\theta_0(x,t)=N_0^{-1-\gamma+\varepsilon}\phi(x_2)\phi(N_0x_1), \qquad \forall t\in [0,1],
        \end{align*}
        generates the sequence of derivatives $\{f_{0,0}^{(j,k)}\}$ given by \eqref{P13.7.1} with  amplitude $N_0^{-1-\gamma+\varepsilon}$, principal directions $(\frac{\pi}{2},0)$ and principal frequencies $(1,N_0)$. Then, using Lemma \ref{lemavelocidad} we obtain that, for every $x\in \mathcal{S}_0 =  \left[-\frac{1}{N_0},\frac{1}{N_0}\right]\times[-1,1]$,  
        \begin{align*}
            v^\gamma(\Theta_0)(x,t)=&N_0^{1+\gamma}N_0^{-1-\gamma+\varepsilon}\left[ a^{(1,1)}_{0, 0} \left(\begin{matrix}
              1& 0 \\
              0  &-1
          \end{matrix}\right)+a^{(1,2)}_{0, 0} \left(\begin{matrix}
             0&1\\
              0&0 
          \end{matrix}\right)
          +a^{(2,1)}_{0, 0} \left(\begin{matrix}
              0&0\\
              1&0 
          \end{matrix}\right)
          \right]x
           \\&-\textbf{n}\left(\frac{\pi}{2}\right)^\perp \frac{1}{2 N_0^{1-\gamma}} \sum_{l=0}^3 b_{0, 0}^{(4-l,l)}\left(\begin{matrix}
                    3\\ l
                \end{matrix}\right)
                x_2^{3-l} (N_0 x_1)^l \\
&- \textbf{n}(0)^{\perp} \frac{N_0^\gamma}{2}\sum_{l=0}^3b_{0, 0}^{(3-l,l+1)}\left(\begin{matrix}
                    3\\ l
                \end{matrix}\right)
                x_2^{3-l} (N_0x_1)^l, 
        \end{align*}
    where it is easy to compute, using (\ref{coeficientesvelocidad}) and \eqref{cotadb}, that
    \begin{align}\label{P13.7.3}
        |a^{(1,1)}_{0,0}|\leq \frac{C}{N_0},\qquad  |a_{0,0}^{(1,2)}|\leq \frac{C}{N_0^2},
    \end{align}
    and, by \eqref{P13.7.2},  
   \begin{align}
        a_{0,0}^{(2,1)}&=-\frac{1}{N_0^{-1-\gamma+\varepsilon}} \frac{2^{\gamma}\Gamma(\frac{1+\gamma}{2})}{\pi \Gamma(\frac{1-\gamma}{2})}\int_{\mathbb{R}_+\times\mathbb{R}}\frac{N_0^{-1-\gamma+\varepsilon}\phi(\frac{y_2}{N_0})\phi''(y_1)}{|y|^{1+\gamma}} \, dy \nonumber \\
        &=-\frac{2^{\gamma}\Gamma(\frac{1+\gamma}{2})}{\pi \Gamma(\frac{1-\gamma}{2})}\left(\int_{\mathbb{R}_+\times\mathbb{R}}\frac{\phi(0)\phi''(y_1)}{|y|^{1+\gamma}} \, dy
        +\int_{\mathbb{R}_+\times\mathbb{R}}\frac{(\phi(\frac{y_2}{N_0})
        -\phi(0))\phi''(y_1)}{|y|^{1+\gamma}} \, dy\right) \nonumber\\
        &=1-\frac{2^{\gamma}\Gamma(\frac{1+\gamma}{2})}{\pi \Gamma(\frac{1-\gamma}{2})}
       \int_{\mathbb{R}_+\times\mathbb{R}}\frac{(\phi(\frac{y_2}{N_0})
        -\phi(0))\phi''(y_1)}{|y|^{1+\gamma}} \, dy. \label{P13.7.4}
\end{align}
Furthermore, the last integral can be estimated as follows: Applying the second order Taylor expansion of $\phi$ centered at the origin (in particular,  taking into account that $\phi'(0)=0$), 
\begin{align*}
&\left|\int_{\mathbb{R}_+\times\mathbb{R}}\frac{(\phi(\frac{y_2}{N_0})
        -\phi(0))\phi''(y_1)}{|y|^{1+\gamma}} \, dy
        \right|\\
        &=
        \left|\int_{\mathbb{R}_+ \times \{|y_2|\leq \sqrt{N_0}\}}\frac{\phi(\frac{y_2}{N_0})
        -\phi(0)}{|y|^{1+\gamma}} \, \phi''(y_1) \, dy
        +\int_{\mathbb{R}_+ \times \{|y_2|\geq \sqrt{N_0} \}}\frac{\phi(\frac{y_2}{N_0})
        -\phi(0)}{|y|^{1+\gamma}} \, \phi''(y_1) \, dy
        \right|
        \\
&\leq   \frac{1}{2}   \int_{|y_2|\leq \sqrt{N_0}}\frac{y_2^2}{N_0^2}\frac{\|\phi''
        \|_{L^\infty}}{|y|^{1+\gamma}} \, |\phi''(y_1)| \, dy
        +\frac{1}{N_0^{\frac{\gamma}{4}}}\int_{|y_2|\geq \sqrt{N_0}}\frac{|\phi(\frac{y_2}{N_0})
        -\phi(0)|}{|y|^{1+\frac{\gamma}{2}}} \, |\phi''(y_1)| \, dy
        \\ &\leq
        \frac{C}{N_0}\int_{|y_1|\leq 1}\frac{1}{|y|^{1+\gamma}} \, dy+\frac{C}{N_0^{\frac{\gamma}{4}}} \int_{|y_1|\leq 1}\frac{1}{|y|^{1+\frac{\gamma}{2}}}dy \leq CN_0^{-\frac{\gamma}{4}}.
\end{align*}
Inserting this into \eqref{P13.7.4}, we arrive at
\begin{align}\label{P13.7.5}
    |a_{0, 0}^{(2, 1)}-1| \leq C N_0^{-\frac{\gamma}{4}}. 
\end{align}

According to  (\ref{defatildeVar}) and \eqref{P13.7.3}, we have
    \begin{align}
|\tilde{a}_0^{(1,1)}|=|s_0^{1+\gamma}\mathcal{C}_{0,0}a_{0,0}^{(1,1)}|\leq \frac{CN_0^\varepsilon}{N_0}\leq N_0^{\varepsilon-2\frac{\varepsilon+\delta}{1+\gamma+\delta}}, \label{P14.7.3}\\
|\tilde{a}_0^{(1,2)}|=|s_0^{1+\gamma}\mathcal{C}_{0,0}a_{0,0}^{(1,2)}|\leq  \frac{CN_0^\varepsilon}{N_0^2}\leq N_0^{\varepsilon-4\frac{\varepsilon+\delta}{1+\gamma+\delta}},\label{P14.7.4}
    \end{align}
where we have also assumed that $N_0$ is large enough. On the other hand, in view of \eqref{P13.7.5}, 
\begin{align}\label{P14.7.5}
|\tilde{a}_0^{(2,1)}-N_0^\varepsilon|\leq CN_0^{\varepsilon-\frac{\gamma}{4}}\leq N_0^{\varepsilon-\frac{\gamma}{8}},
\end{align}
provided that $N_0$ is sufficiently large. Then we have shown the validity of (\ref{velocidadTheta}), (\ref{cotascoeficientestilde}) and (\ref{coeficienteppal}). Finally, the bound (\ref{cotab}) follows immediately from  \eqref{cotadb}.
\item \textbf{Dynamics of $\theta_1$.} Recall that, by (\ref{deftheta}), \eqref{transporteppal}, \eqref{P12.7.1} and \eqref{ic},  
\begin{align*}
    \theta_1=\sum_{i=0}^{I_\gamma} f_{1,i},
\end{align*}
where the corresponding  principal layer $f_{1, 0}$ solves the following Cauchy problem: 
\begin{align*}
    &\partial_t f_{1,0}+\overline{v}^\gamma(\Theta_0)\cdot \nabla f_{1,0}=0,\\
   & f_{1,0}(x,t_1)=N_1^{-1-\gamma+\varepsilon} \phi(N_1^{1-\frac{\varepsilon+\delta}{1+\gamma+\delta}}x_1)\phi(N_1^{1-\frac{\varepsilon+\delta}{1+\gamma+\delta}}x_2).
\end{align*}
Here, we have also used that $\phi$ is an even function. As already discussed above (see \eqref{capaprincipal} and especially the text just after \eqref{transporteppal}-\eqref{P12.7.1}), it turns out that $f_{1, 0}$ is given by 
\begin{align*}
    f_{1,0}(x,t)=N_1^{-1-\gamma+\varepsilon} \phi(r_1(t) \textbf{n}(\alpha_1(t))\cdot x)\phi(s_1(t) \textbf{n}(\beta_1(t))\cdot x), \qquad  \forall t\in [t_1,1], 
\end{align*}
where $(r_1,\alpha_1)$ verify (\ref{dinamicaalpha}) and $(s_1, \beta_1)$ verify (\ref{dinamicabeta}) for $t\in[t_1,1]$, with constant coefficients (see \eqref{abarra}) 
\begin{align}\label{P14.7.1}
    \overline{a}_1^{(1,1)}(t)=
    \tilde{a}_0^{(1,1)},
    \qquad 
    \overline{a}_1^{(1,2)}(t)=
    \tilde{a}_0^{(1,2)},
     \qquad 
    \overline{a}_1^{(2,1)}(t)=
    \tilde{a}_0^{(2,1)},
    \qquad 
     \overline{a}_1^{(2,2)}(t)=-
    \tilde{a}_0^{(1,1)}. 
\end{align}
Here, we used that $r_1, s_1, \alpha_1, \beta_1$ may be extended to the endpoint $t=1$, which easily follows from the fact that  the coefficients \eqref{P14.7.1} involved in the ODEs (\ref{dinamicaalpha}) and (\ref{dinamicabeta})  are constant, and so the solution $f_{1, 0}$ can also be continued until $t=1$ as an application of Lemma \ref{dinamica}. 

Next we turn our attention to (\ref{estimaalpha}), (\ref{estimar}), (\ref{estimabeta}) and (\ref{estimas}). For this purpose, we aim to  apply Lemma \ref{lemaedos} with $n=0$ (more precisely, see Remark \ref{Rmk9}). Indeed, by \eqref{P14.7.3}--\eqref{P14.7.5},  we conclude that, for all $t\in [t_1,1]$, 
 \begin{align*}
        |\alpha_{1}(t)-\beta_{0}|\leq  N_0^{\varepsilon-4\frac{\varepsilon+\delta}{1+\gamma+\delta}}, \qquad  \left|
        \frac{r_{1}(t)}{r_{1}(t_{1})}-1\right|\leq N_0^{-\frac{\gamma\mu}{10}},
    \end{align*}
    \begin{align}\label{P15.7.1}
        \left|\beta_1(t)-\beta_0+\arctan \left(
        \frac{1}{(\tilde{a}_0^{(2, 1)} + \tilde{a}_0^{(1, 2)}) (t-t_1)}
        \right)
        \right|\leq N_0^{-\varepsilon+\mu-\frac{\gamma}{20}\mu}, 
    \end{align} 
    and
    \begin{align}\label{P16.7.1}   
        \left|\frac{s_1(t)}{s_1(t_1)} - \sqrt{1+ (\tilde{a}_0^{(2, 1)} + \tilde{a}_0^{(1, 2)})^2 (t-t_1)^2 }\right| \leq N_0^{-\frac{\gamma \mu}{40}} \sqrt{1+ (\tilde{a}_0^{(2, 1)} + \tilde{a}_0^{(1, 2)})^2(t-t_1)^2  }. 
    \end{align}
    In particular,  (\ref{estimaalpha}) and (\ref{estimar}) hold. Concerning \eqref{estimabeta}, from Mean Value Theorem and \eqref{P14.7.4}-\eqref{P14.7.5},
    \begin{align}
        \left|
\arctan \left(
        \frac{1}{(\tilde{a}_0^{(2,1)} +\tilde{a}_0^{(1,2)})(t-t_1)}
        \right)-\arctan\left(\frac{1}{N_0^\varepsilon (t-t_{1})}\right)
        \right| & \nonumber\\
       & \hspace{-8cm}\leq \frac{C}{1+(N_0^{\varepsilon}(t-t_1))^2} \, (|\tilde{a}_0^{(2,1)}-N_0^\varepsilon| + |\tilde{a}_0^{(1, 2)}|) (t-t_1)  \nonumber 
        \\ &\hspace{-8cm}\leq \frac{C N_0^{\varepsilon-3\sigma}(t-t_1)}{1+(N_0^{\varepsilon}(t-t_1))^2}\leq  \frac{N_0^{-2\sigma} }{2(1+N_0^{\varepsilon}(t-t_1))}, \label{P15.7.2}
    \end{align}
    where we may choose $N_0$ to be large enough so that the last step holds. Putting together \eqref{P15.7.1} and \eqref{P15.7.2}, we obtain 
    \begin{align*}
         \left|\beta_1(t)-\beta_0+\arctan \left(\frac{1}{N_0^\varepsilon (t-t_{1})}\right)
        \right| &\leq \frac{N_0^{-2\sigma} }{2(1+N_0^{\varepsilon}(t-t_1))} + N_0^{-\varepsilon+\mu-\frac{\gamma}{20}\mu} \\
        & \hspace{-5cm} \leq \frac{N_0^{-2\sigma} }{2(1+N_0^{\varepsilon}(t-t_1))} + \frac{N_0^{-2\sigma} }{2(1+N_0^{\varepsilon}(t-t_1))} = \frac{N_0^{-2\sigma} }{1+N_0^{\varepsilon}(t-t_1)},
    \end{align*}
    where we have also used that $t-t_1 \leq 1-t_1 \leq C N_{0}^{-\mu}$. This proves \eqref{estimabeta}.

    Assume now that $t \geq t_2$. Throughout this step, $N_1$ denotes a free parameter ranging over the interval
\begin{equation}\label{P18.7.1}
    N_1\in\big[N_0^{1+\frac{\gamma}{2}},\,N_0^{1+\frac{3\gamma}{2}}\big],
\end{equation}
which is the range prescribed by \eqref{des}, and accordingly
$t_2:=1-N_1^{-\mu}$. All the estimates below hold uniformly in that range;
the particular value of $N_1$ is fixed at the end of this step.
    \begin{align*}
        \left|\arctan \bigg(\frac{1}{N_0^\varepsilon (t-t_1)}  \bigg) - N_0^{-\varepsilon+ \mu} \right| &  \\
        &\hspace{-4.5cm}\leq  \left|\arctan \bigg(\frac{1}{N_0^\varepsilon (t-t_1)}  \bigg) - \frac{1}{N_0^\varepsilon (t-t_1)} \right| + \left|\frac{1}{N_0^\varepsilon (t-t_1)}  - N_0^{-\varepsilon + \mu} \right| \\
        &\hspace{-4.5cm}\leq   C N_0^{-3 \varepsilon} (t-t_1)^{-3} + \bigg|\frac{1}{N_0^\varepsilon (t-t_1)}- N_0^{-\varepsilon+\mu} \bigg| \\
        &\hspace{-4.5cm} \leq C N_0^{-3 \varepsilon} (t_2-t_1)^{-3} +  N_0^{-\varepsilon} \Big|N_0^{\mu}-\frac{1}{t_2-t_1} \Big| \\
        & \hspace{-4.5cm} = C N_0^{-3 (\varepsilon - \mu)} +  N_0^{-\varepsilon} \bigg(\frac{1}{N_0^{-\mu} - N_1^{-\mu}} -\frac{1}{N_0^{-\mu}}\bigg) \\
        & \hspace{-4.5cm}  \leq C N_0^{-3 (\varepsilon - \mu)}  + C N_0^{-\varepsilon + 2 \mu}  N_1^{-\mu} \leq  C N_0^{-3 (\varepsilon - \mu)}  + C N_0^{-\varepsilon +  \mu - \frac{\mu \gamma}{2}}.  
    \end{align*}
    Combining this with \eqref{estimabeta}, we obtain 
    \begin{align*}
          \left| \beta_{1}(t)-\beta_{0}+N_0^{-\varepsilon+\mu}\right| & \leq   \frac{N_0^{-2\sigma} }{1+N_0^{\varepsilon-\mu}}  +  C N_0^{-3 (\varepsilon - \mu)}  + C N_0^{-\varepsilon +  \mu - \frac{\mu \gamma}{2}}  \\
          & \hspace{-2cm} \leq N_0^{-2\sigma -\varepsilon + \mu} +  C N_0^{-3 (\varepsilon - \mu)}  + C N_0^{-\varepsilon +  \mu - \frac{\mu \gamma}{2}} \leq N_0^{-\varepsilon + \mu - \sigma},
    \end{align*}
    which gives \eqref{estimabetatlargo}.

    The proof of \eqref{estimas} can be carried out similarly as in \eqref{estimabeta}  but now relying on \eqref{P16.7.1}.  On the other hand, the first statement in \eqref{estimastlargo} (that is, $0 \leq s_1'(t)/s_1(t) \leq N_0^{\mu+\sigma}$ if $t \geq t_2$) follows immediately from \eqref{sestatico2}, while the second statement in  \eqref{estimastlargo} (that is, $|s_{1}(t)/s_{1}(t_{1})-N^{\varepsilon-\mu}_0|\leq N_0^{\varepsilon-\mu-\sigma}$ if $t \geq t_2$) can be obtained from \eqref{estimas} in a similar fashion as \eqref{estimabetatlargo}. Further details are left to the interested reader.
    

      \item \label{existenciaN1}\textbf{Existence of $N_1$ defined by \eqref{defN} and relation between $N_0$ and $N_1$.} 
    According to (\ref{capaprincipal}), 
    \begin{align}\label{P25.7.1}
        \mathcal{C}_{1,0}=N_1^{-1-\gamma+\varepsilon}.
        \end{align}
        For the coefficient $a_{1,0}^{(2,1)}$, using (\ref{coeficientesvelocidad}) and the fact that (see \eqref{DefinitionAlignedRescaled}, \eqref{P16.7.2}) 
        \begin{align*}
            \tilde{f}_{1,0}^{(0,2)}(x,t)=R_{-\beta_1}f^{(0,2)}_{1,0}\left(\frac{x}{s_1},t\right)= N_1^{-1-\gamma+\varepsilon}\phi \left(\frac{r_1}{s_1}\textbf{n}(\alpha_1-\beta_1)\cdot x\right)\phi''(x_1),
        \end{align*}
        we have (see also \eqref{P13.7.2})
        \begin{align}
        a_{1,0}^{(2,1)}=& -\frac{1}{\mathcal{C}_{1,0}}\left[d_{1,0}^{(0,2)}+2\frac{r_1}{s_1}d_{1,0}^{(1,1)}\cos (\alpha_1-\beta_1)+\left(\frac{r_1}{s_1}\right)^2d_{1,0}^{(2,0)}\cos^2(\alpha_1-\beta_1)\right] \nonumber \\
        =&1+\frac{2^{\gamma}\Gamma(\frac{1+\gamma}{2})}{\pi \Gamma(\frac{1-\gamma}{2})}\int_{\mathbb{R}_+\times \mathbb{R}} \left(\phi(0)-\phi \left(\frac{r_1}{s_1}\textbf{n}(\alpha_1-\beta_1)\cdot y\right)\right)\frac{\phi''(y_1)}{|y|^{1+\gamma}} \, dy \nonumber 
       \\
        &-\frac{1}{\mathcal{C}_{1,0}}\left[2\frac{r_1}{s_1}d_{1,0}^{(1,1)}\cos (\alpha_1-\beta_1)+\left(\frac{r_1}{s_1}\right)^2d_{1,0}^{(2,0)}\cos^2(\alpha_1-\beta_1)\right]. \label{P16.7.3}
        \end{align}
    Next we deal with the first term of the previous expression. Applying basic properties of $\phi$, we have  
    \begin{align}
       &\hspace{-10mm} \left|
\int_{\mathbb{R}_+\times \mathbb{R}} \left(\phi(0)-\phi \left(\frac{r_1}{s_1}\textbf{n}(\alpha_1-\beta_1)\cdot y\right)\right)\frac{\phi''(y_1)}{|y|^{1+\gamma}} \, dy
        \right| \nonumber \\
        \leq& \frac{\|\phi''\|_{L^\infty}}{2} \int_{|\textbf{n}(\alpha_1-\beta_1)\cdot y|\leq \sqrt{s_1/r_1}} \left|
        \frac{r_1}{s_1}\textbf{n}(\alpha_1-\beta_1)\cdot y
        \right|^2\frac{|\phi''(y_1)|}{|y|^{1+\gamma}} \, dy \nonumber\\
        &+2 \|\phi\|_{L^\infty} \left(\frac{r_1}{s_1}\right)^{\frac{\gamma}{4}}\int_{|\textbf{n}(\alpha_1-\beta_1)\cdot y|\geq \sqrt{s_1/r_1}}  \frac{|\phi''(y_1)|}{|y|^{1+\frac{\gamma}{2}}} \, dy \nonumber \\
        \leq &C \frac{r_1}{s_1} \int_{(-1, 1) \times \mathbb{R}} |y|^{-1-\gamma} \, dy + C \left(\frac{r_1}{s_1}\right)^{\frac{\gamma}{4}} \int_{(-1, 1) \times \mathbb{R}} |y|^{-1-\frac{\gamma}{2}} \, dy  \leq C\left(\frac{r_1}{s_1}\right)^{\frac{\gamma}{4}}, \label{P16.7.4}
    \end{align}
    where we have also used in the last estimate that $r_1/s_1$ is  bounded on $[t_{1}, 1]$. To justify the latter statement, one may use  (\ref{estimar}) and \eqref{estimas} (recall also that $r_1(t_1) = s_1(t_1)$, see \eqref{ic}) so that
    \begin{align*}
        r_1(t) &\leq (1+N_0^{-\frac{\gamma \mu}{10}}) r_1(t_1) = (1+N_0^{-\frac{\gamma \mu}{10}}) s_1(t_1) \\
        & \leq (1+N_0^{-\frac{\gamma \mu}{10}}) \, \frac{s_1(t)}{(1-N_0^{-\sigma}) \sqrt{1+N_0^{2 \varepsilon} (t-t_1)^2}} \\
        & \leq \frac{1+N_0^{-\frac{\gamma \mu}{10}}}{1-N_0^{-\sigma}}  \, s_1(t) \leq C s_1(t). 
    \end{align*}
    In fact, these computations can be improved provided that $t \geq t_2$  via (\ref{estimastlargo}). To be more precise,  one can show that, at least for $N_0$ sufficiently large, the following holds 
    \begin{align}\label{P16.7.5}
        \frac{r_1(t)}{s_1(t)}\leq N_0^{-\varepsilon+\mu+\frac{\sigma}{10}}, \qquad \forall t \geq t_2. 
    \end{align}
    Putting together \eqref{P16.7.3}, \eqref{P16.7.4} and \eqref{P16.7.5} (see also \eqref{cotadb}), we derive
    \begin{align}\label{P5.8.4}
        |a_{1, 0}^{(2, 1)}(t)-1| \leq  N_0^{-\frac{(\varepsilon-\mu)}{5}\gamma+\frac{\sigma\gamma}{40}}, \qquad \forall t \geq t_2. 
    \end{align}
    In particular,  taking $N_0$ sufficiently large,  
    \begin{align}\label{P24.7.1}
\left(\frac{9}{10}\right)^{(\varepsilon+\delta)\frac{1+\gamma}{1+\gamma+\delta}}\leq a_{1,0}^{(2,1)}(t_2)\leq \left(\frac{11}{10}\right)^{(\varepsilon+\delta)\frac{1+\gamma}{1+\gamma+\delta}}.
    \end{align}

  On the other hand, by \eqref{P18.7.1} and \eqref{des},
   \begin{align*}
       N_0^{-\mu}\big(1-N_0^{-\frac{\gamma\mu}{2}}\big)\leq t_2-t_1=N_0^{-\mu}-N_1^{-\mu}\leq N_0^{-\mu},
   \end{align*}
   so \eqref{estimas} with $n=0$ and $t=t_2$, together with
   $0\leq\sqrt{1+x^2}-x\leq\frac{1}{2x}$, gives
   \begin{align*}
    \left|\frac{s_1(t_2)}{s_1(t_1)} N_0^{-\varepsilon +\mu}- 1\right|
    \leq C N_0^{-2\sigma}+CN_0^{-2(\varepsilon-\mu)}+N_0^{-\frac{\gamma\mu}{2}}
    \leq CN_0^{-\sigma},
   \end{align*}
   where we used $100\sigma<\gamma\mu$ (see \eqref{parametros1}). Hence, for $N_0$ large enough,
   \begin{equation}\label{P24.7.2}
       \left(\frac{9}{10}\right)^{\frac{\varepsilon+\delta}{1+\gamma+\delta}} \, N_0^{\varepsilon-\mu}\leq \frac{s_1(t_2)}{s_1(t_1)}\leq \left(\frac{11}{10}\right)^{\frac{\varepsilon+\delta}{1+\gamma+\delta}} \,  N_0^{\varepsilon-\mu}.
   \end{equation}
    As a byproduct of \eqref{P24.7.1} and \eqref{P24.7.2}, we obtain that
    \begin{equation}\label{P24.7.3}
         \left(\frac{1}{2}\right)^{(\varepsilon+\delta)\frac{1+\gamma}{1+\gamma+\delta}} N_0^{(1+\gamma)(\varepsilon-\mu)}\leq
        \left(\frac{s_1(t_2)}{s_1(t_1)}\right)^{1+\gamma}a_{1,0}^{(2,1)}(t_2)\leq 2^{(\varepsilon+\delta)\frac{1+\gamma}{1+\gamma+\delta}} N_0^{(1+\gamma)(\varepsilon-\mu)}.
    \end{equation}
    In particular, although the number $(s_1(t_2)/s_1(t_1))^{1+\gamma}a_{1,0}^{(2,1)}(t_2)$ depends on $N_1$ trough $t_2$, the uniform bound displayed in \eqref{P24.7.3} shows that such a number can be estimated in terms of $N_0$. Accordingly, one can ensure the existence\footnote{Uniqueness is not guaranteed. However, it is enough to pick one of the possible solutions.} (by a standard Bolzano-type argument) of $N_1$           such that 
    \begin{equation}\label{P7.8.4}
      N_1^{\frac{1+\gamma}{1+\gamma+\delta}(\varepsilon+\delta)}   =\left(\frac{s_1(t_2)}{s_1(t_1)}\right)^{1+\gamma}a_{1,0}^{(2,1)}(t_2). 
    \end{equation}  
    To rigorously  formalize this argument one can consider the map 
    \begin{align}\label{P7.8.3}
        N_{1}\mapsto  N_1^{\frac{1+\gamma}{1+\gamma+\delta}(\varepsilon+\delta)}   -\left(\frac{s_1(t_2)}{s_1(t_1)}\right)^{1+\gamma}a_{1,0}^{(2,1)}(t_2),
    \end{align}
    that is continuous  due to the continuity of the involved ODEs  with respect to time and initial conditions. By  \eqref{P24.7.3}, the map given in \eqref{P7.8.3} is negative at $N_{1}=N_0^{1+\frac{\gamma}{2}}$ and positive when $N_1=N_0^{1+\frac{3\gamma}{2}}$. Hence existence of $N_1$ with \eqref{P7.8.4} is shown and it is in the desired interval according to \eqref{des}.
    
   Notice that for such a choice of $N_1$ (recall \eqref{ic} and \eqref{P25.7.1}):  
    \begin{align*}
 \frac{N_1^\varepsilon}{s_1(t_1)^{1+\gamma}\mathcal{C}_{1,0}}=\left(\frac{s_1(t_2)}{s_1(t_1)}\right)^{1+\gamma}a_{1,0}^{(2,1)}(t_2),
    \end{align*}
    or equivalently 
    $$
          N_1^\varepsilon= s_1(t_2)^{1+\gamma}\mathcal{C}_{1,0}a_{1,0}^{(2,1)}(t_2). 
    $$
     This proves   existence of $N_1$ satisfying \eqref{defN}. 

Next, we check the right-hand side estimate in \eqref{crecimientoN} via \eqref{P24.7.3}:
\begin{align*}
    N_1^\varepsilon &=s_1(t_1)^{1+\gamma} \mathcal{C}_{1,0}\left(\frac{s_1(t_2)}{s_1(t_1)}\right)^{1+\gamma}a_{1,0}^{(2,1)}(t_2)  \\
    & \leq N_1^{(1+\gamma)(1-\frac{\varepsilon+\delta}{1+\gamma+\delta})} N_1^{-1-\gamma+\varepsilon} 2^{(\varepsilon+\delta)\frac{1+\gamma}{1+\gamma+\delta}} N_0^{(1+\gamma)(\varepsilon-\mu)} 
\end{align*}
and so 
$
    N_1 \leq 2 N_0^{(1+\gamma+\delta)\frac{\varepsilon-\mu}{\varepsilon+\delta}}. 
$
Furthermore, the  left-hand side in \eqref{crecimientoN} can be computed in an analogous way.

In view of the three steps above,  the proof of the desired result with $n=0$ is finished. Next, we assume the validity of the result  for any positive integer that is less than or equal to $n-1$ and so we will prove that it also remains true for the following step $n$. 

      \item \textbf{Computing the velocity generated by  $\Theta_n$.} Recall that, by  \eqref{sumacapas}, 
      \begin{align*}
        \Theta_n = h_n\theta_n +\Theta_{n-1}
      \end{align*}
      and so
      \begin{align}\label{P2.8.4}
          v^\gamma(\Theta_n) =h_n v^\gamma(\theta_n) +v^\gamma(\Theta_{n-1}). 
      \end{align}
      On the one hand, the term $v^\gamma(\Theta_{n-1})$  evaluated at $t \geq t_{n+1}$ (and in particular, $t > t_n$) and $x \in \mathcal{S}_n(t)$ is given by the inductive hypothesis. In this step, we make use of the fact that the family $\{\mathcal{S}_n(t)\}$ is nested. In fact we will prove
      \begin{align}
      \label{soportesencajados}
            \mathcal{S}_n(t)\subset B_{0}(CN_n^{-1+2\frac{\varepsilon+\delta}{1+\gamma+\delta}}) \subset B_0((C' N_{n-1})^{-1})\subset\mathcal{S}_{n-1}(t), \qquad  \forall t\in[t_n,1], 
      \end{align}
      for every $n \in \mathbb{N}$, where $B_{0}(CN_n^{-1+2\frac{\varepsilon+\delta}{1+\gamma+\delta}})$ is the ball of radius  $CN_n^{-1+2\frac{\varepsilon+\delta}{1+\gamma+\delta}}$ centered at the origin.
      Indeed, given any $x \in \mathcal{S}_n(t)$, we can write $$x = a \mathbf{n}(\alpha_n(t)) + b \mathbf{n}(\alpha_n(t))^\perp$$ for certain  $a, b \in \mathbb{R}$. In particular, we have
      $$
            \mathbf{n}(\alpha_n(t)) \cdot x = a 
      $$
    and
    $$
        \mathbf{n}(\beta_n(t)) = \cos (\beta_n(t)-\alpha_n(t)) \mathbf{n}(\alpha_n(t)) + \sin (\beta_n(t)-\alpha_n(t)) \mathbf{n}(\alpha_n(t))^\perp. 
    $$
    From the assumption  $x \in \mathcal{S}_n(t)$, we derive that
    $$
        |a| \leq \frac{1}{r_n(t)} \qquad \text{and} \qquad |a \cos (\beta_n(t)-\alpha_n(t)) + b \sin (\beta_n(t)-\alpha_n(t)) | \leq \frac{1}{s_n(t)}
    $$
    and then
    \begin{align}\label{P4.8.1}
        |x|^2 &= a^2 + b^2   \leq 2 \bigg[ \frac{1}{\sin (\alpha_n(t)-\beta_n(t))} \bigg(\frac{1}{r_n(t)} + \frac{1}{s_n(t)} \bigg) \bigg]^2.
    \end{align}
    Moreover,  it follows from   \eqref{estimar} and \eqref{estimas} in the case $n-1$ that $r_n(t) \leq Cs_n(t)$ if $t \geq t_{n}$ (this assertion is guaranteed if $N_0$ is large enough). Putting this together with \eqref{P4.8.1}, we obtain 
      \begin{align*}
\sup_{x\in\mathcal{S}_{n}(t)}|x|
\le\frac{2\sqrt{2}}{r_{n}(t)\,|\sin(\alpha_{n}(t)-\beta_{n}(t))|},
\end{align*}
where $\sin(\alpha_n(t)-\beta_n(t)) \neq 0$ for all $t \geq t_n$, in fact, basic conservation laws yield that (recall also the initial conditions \eqref{ic})
\begin{equation}\label{P12.8.2}
r_{n}(t)s_{n}(t)\sin (\alpha_{n}(t)-\beta_{n}(t))
=r_{n}(t_{n})s_{n}(t_{n})
=N_{n}^{2 (1-\frac{\varepsilon+\delta}{1+\gamma+\delta})}. 
\end{equation}
As a consequence, if $t \geq t_{n}$ then  
\begin{align*}
\sup_{x\in\mathcal{S}_{n}(t)}|x| \leq 2 \sqrt{2} s_{n}(t) N_{n}^{-2 (1-\frac{\varepsilon+\delta}{1+\gamma+\delta})}\leq CN_{n}^{-1+2\frac{\varepsilon+\delta}{1+\gamma+\delta}},
\end{align*}
where we have also used that 
\begin{align*}
    s_{n}(t)\leq CN_{n},
\end{align*}
    which follows  from  the inductive hypothesis \eqref{estimas}  and \eqref{crecimientoN}. This proves the first inclusion in \eqref{soportesencajados}. 

      Regarding the second and third inclusions in \eqref{soportesencajados}, it is easy to see that if $|x|\le 1/s_{n-1}(t)$ then $x \in \mathcal{S}_{n-1}(t)$. Then, using again $s_{n-1}(t)\leq C'N_{n-1}$, it is enough to check that 
\begin{align}\label{P4.8.3}
     CN_n^{-1+2\frac{\varepsilon+\delta}{1+\gamma+\delta}}\leq (C' N_{n-1})^{-1}
\end{align}
in order to conclude that $B_{0}(CN_n^{-1+2\frac{\varepsilon+\delta}{1+\gamma+\delta}}) \subset B_0((C' N_{n-1})^{-1}) \subset  \mathcal{S}_{n-1}(t)$ (see \eqref{soportesencajados}). It remains to show \eqref{P4.8.3}: From \eqref{des} (see also Remark \ref{Remark25})
\begin{align*}
    N_{n}^{1-2\frac{\varepsilon+\delta}{1+\gamma+\delta}}\geq N_{n-1}^{(1+\frac{\gamma}{2}) (1-2\frac{\varepsilon+\delta}{1+\gamma+\delta})}\geq N_{n-1}^{1+\frac{\gamma}{8}}
\end{align*}
for the special choice of $\varepsilon$ given by  $\varepsilon=\frac{\gamma(1-\gamma)}{4(4+\gamma)}$. As an immediate consequence, \eqref{P4.8.3} is satisfied provided that $N_0$ is large enough.

      Next we turn our attention to the term  $v^\gamma(\theta_n)$ in \eqref{P2.8.4}.  
      It follows from \eqref{vlin} and \eqref{P2.8.1} that 
      \begin{align*}
          v^\gamma(\theta_n)=& \tilde{a}_n^{(2, 1)}  \left(\begin{matrix}
              -\frac{1}{2}\sin 2\beta_n&-\sin^2\beta_n\\
              \cos^2 \beta_n&\frac{1}{2}\sin 2\beta_n 
          \end{matrix}\right)x\\
          &+\left[\tilde{a}_n^{(1, 1)} \left(\begin{matrix}
              \cos 2\beta_n& \sin 2\beta_n \\
              \sin 2\beta_n  &-\cos 2 \beta_n
          \end{matrix}\right)
          +\tilde{a}_n^{(1, 2)} \left(\begin{matrix}
              -\frac{1}{2}\sin 2\beta_n &\cos^2\beta_n\\
              -\sin^2 \beta_n&\frac{1}{2}\sin 2\beta_n 
          \end{matrix}\right)\right]x
           \\&-\textbf{n}(\alpha_n)^\perp \frac{r_n}{2 s_n^{1-\gamma}}\sum_{l=0}^3  \tilde{b}_n^{(4-l, l)}    \left(\begin{matrix}
                    3\\ l
                \end{matrix}\right)
                \left(r_n \textbf{n}(\alpha_n) \cdot x\right)^{3-l} (s_n\textbf{n}(\beta_n)\cdot x)^l \\
&-\textbf{n}(\beta_n)^\perp \frac{s_n^\gamma}{2}\sum_{l=0}^3  \tilde{b}_n^{(3-l,l+1)}  \left(\begin{matrix}
                    3\\ l
                \end{matrix}\right)
                \left(r_n \textbf{n}(\alpha_n) \cdot x\right)^{3-l} (s_n\textbf{n}(\beta_n)\cdot x)^l.
      \end{align*}
    Inserting this into \eqref{P2.8.4}, and recalling that $h_n(t) = 1$ if $t \geq t_{n+1}$ (see \eqref{P2.8.3}), the desired formula \eqref{velocidadTheta} is achieved.    
    
    The involved coefficients $\tilde{a}_{j}^{(l,k)}$ and $\tilde{b}_{j}^{(l,k)}$ with $j < n$ can be estimated by using the induction hypothesis. Next we focus on the case $j=n$.  Assume that $t \geq t_{n+1}$. 
 Applying (\ref{estimaalpha}) and \eqref{estimabetatlargo} in the case $n-1$, it is easy to see that 
    \begin{align}\label{P5.8.3}
        |\alpha_n(t)-\beta_n(t)|\leq N_{n-1}^{-\varepsilon+\mu+\frac{\sigma}{10}}
    \end{align}
    
    Here, we are assuming that $N_0$ is large enough.

Suppose now that $t \in [t_n, 1]$. As a combination of (\ref{nuevaamplitud}) and (\ref{intamplitud}), 
$$
    \mathcal{C}_{n,1}(t)= \int_{t_{n}}^t \mathcal{C}_{n,0}^2s_n(\tau)^\gamma r_n(\tau)|\sin(\alpha_n-\beta_n)(\tau)| \, d\tau.    
$$
From \eqref{estimar}, \eqref{ic},     \eqref{estimaalpha}, \eqref{estimabeta}, \eqref{estimas} and \eqref{crecimientoN} for the case $n-1$, we have
   \begin{align*}
            \mathcal{C}_{n,1}(t) 
            &\leq C\mathcal{C}_{n,0}^2 N_n^{(1+\gamma)(1-\frac{\varepsilon+\delta}{1+\gamma+\delta})}\int_{t_{n}}^1 \left(\frac{s_n(\tau)}{s_n(t_n)}\right)^\gamma 
            \arctan \left(\frac{1}{N_{n-1}^\varepsilon (\tau-t_n)} \right)
             \, d \tau\\
           &\leq C\mathcal{C}_{n,0}^2 N_n^{(1+\gamma)(1-\frac{\varepsilon+\delta}{1+\gamma+\delta})} \left(N_{n-1}^{-\varepsilon}+
            \int_{t_{n}+N_{n-1}^{-\varepsilon}}^1 \left(\frac{s_n(\tau)}{s_n(t_n)}\right)^\gamma 
            \arctan \left(\frac{1}{N_{n-1}^\varepsilon (\tau-t_n)} \right) \, d\tau
            \right)\\
            &\leq C\mathcal{C}_{n,0}^2 N_n^{(1+\gamma)(1-\frac{\varepsilon+\delta}{1+\gamma+\delta})}
            \left(N_{n-1}^{-\varepsilon}+
            \int_{t_{n}+N_{n-1}^{-\varepsilon}}^1 \left(N_{n-1}^\varepsilon(\tau-t_n)\right)^{\gamma-1} \,  d\tau
            \right)\\
           &\leq C\mathcal{C}_{n,0}^2 N_n^{(1+\gamma) (1-\frac{\varepsilon+\delta}{1+\gamma+\delta})} N_{n-1}^{(\gamma-1)\varepsilon}
           =C\mathcal{C}_{n,0}N_n^{\varepsilon-\frac{1+\gamma}{1+\gamma+\delta}(\varepsilon+\delta)} N_{n-1}^{-(1-\gamma)(\varepsilon-\mu)}N_{n-1}^{-(1-\gamma)\mu}
           \\&\leq C \mathcal{C}_{n,0}N_{n}^{\varepsilon-(1+\gamma)\frac{\varepsilon+\delta}{1+\gamma+\delta}-(1-\gamma)\frac{\varepsilon+\delta}{1+\gamma+\delta}} N_{n-1}^{-(1-\gamma)\mu}
           \leq C\mathcal{C}_{n,0}N_{n}^{\varepsilon-2\frac{\varepsilon+\delta}{1+\gamma+\delta}}. 
        \end{align*}
        More generally, an analogous reasoning gives that
    \begin{align}\label{P5.8.1}
        \mathcal{C}_{n,i}(t)\leq C\mathcal{C}_{n,0}N_{n}^{i(\varepsilon-2\frac{\varepsilon+\delta}{1+\gamma+\delta})} \qquad  \text{if} \qquad  i \in \{ 0, \ldots, I_\gamma \}.
    \end{align}

    From   \eqref{estimar} and \eqref{estimastlargo} in the case $n-1$, we obtain
    \begin{align}\label{P5.8.2}
     \frac{r_n(t)}{s_n(t)}\leq  N_{n-1}^{-\varepsilon+\mu+\frac{\sigma}{10}} \qquad \text{if} \qquad t \geq t_{n+1}.
    \end{align}
    
    Using the definitions \eqref{coeficientesvelocidad}, \eqref{defatilde} and the estimates  (\ref{cotadb}), (\ref{estimastlargo}) (specifically, $s_n \leq C N_n^{1-\frac{\varepsilon + \delta}{1+\gamma +\delta}} N_{n-1}^{\varepsilon-\mu}$), \eqref{P5.8.3},  \eqref{P5.8.1} and \eqref{P5.8.2},  we get for every $t \in [t_{n+1},1]$  (to simplify the notation, we shall omit the variable $t$ in the forthcoming computations)
    \begin{align*}
|\tilde{a}_{n}^{(1,1)}|&=\left|s_n^{1+\gamma}\sum_{i=0}^{I_\gamma}
        \mathcal{C}_{n,i}\sin (\alpha_{n}-\beta_n) \left[\left(\frac{r_n}{s_n}\right)^2\frac{d_{n,i}^{(2,0)}}{\mathcal{C}_{n,i}}\cos (\alpha_n-\beta_n)+\frac{r_n}{s_n}\frac{d_{n,i}^{(1,1)}}{\mathcal{C}_{n,i}}\right]
        \right|\\
       & \leq C N_n^{1+\gamma-\frac{1+\gamma}{1+\gamma+\delta}(\varepsilon+\delta)}N_{n-1}^{(1+\gamma)(\varepsilon-\mu)} N_n^{-1-\gamma+\varepsilon} \left(N_{n-1}^{-\varepsilon+\mu+\frac{\sigma}{10}}\right)^2 =  C N_n^{\varepsilon-\frac{1+\gamma}{1+\gamma+\delta} (\varepsilon+\delta)}N_{n-1}^{-(1-\gamma)(\varepsilon-\mu)+\frac{\sigma}{5}}\\
        &\leq C N_n^{\varepsilon-\frac{1+\gamma}{1+\gamma+\delta} (\varepsilon+\delta)-\frac{1-\gamma}{1+\gamma+\delta}(\varepsilon+\delta)}N_n^{\frac{2\sigma}{5}}\leq N_n^{\varepsilon-\frac{2(\varepsilon+\delta)}{1+\gamma+\delta}+\sigma},
    \end{align*} 
    where the relation $N_n \leq C N_{n-1}^{(1+\gamma + \delta) \frac{\varepsilon-\mu}{\varepsilon + \delta}}$ (see \eqref{crecimientoN}) is also applied in the penultimate step. This gives the first estimate in \eqref{cotascoeficientestilde}. 
    
    In a similar fashion to the argument above for the coefficient $\tilde{a}_n^{(1, 1)}$, one can show that  
    \begin{align*}
        |\tilde{a}_n^{(1,2)}|=\left|
        s_n^{1+\gamma}\sum_{i=0}^{I_\gamma} \mathcal{C}_{n,i} \sin^2 (\alpha_n-\beta_n)\left(\frac{r_n}{s_n}\right)^2\frac{d^{(2,0)}_{n,i}}{\mathcal{C}_{n,i}}
        \right|\leq N_n^{\varepsilon-\frac{4(\varepsilon+\delta)}{1+\gamma+\delta}+\sigma}, \qquad  \forall t \in [t_{n+1},1],
    \end{align*}
    or in other words, the second estimate in \eqref{cotascoeficientestilde} is satisfied.

 Next we deal with  (\ref{coeficienteppal}). For the coefficient $a_{n,0}^{(2,1)}$, proceeding similarly as for  $a_{1,0}^{(2,1)}$ in \eqref{P5.8.4} and using \eqref{des}, we get that  
    \begin{align}
    \label{cotaan21}
       |a_{n,0}^{(2,1)}(t)-1| \leq N_{n-1}^{-\frac{(\varepsilon-\mu)}{5}\gamma+\frac{\sigma\gamma}{40}}\leq N_{n}^{-\frac{(\varepsilon-\mu)}{20}\gamma+\sigma}, \qquad  \forall t\in [t_{n+1},1].
    \end{align}
    In the case  $i\in \{1,\ldots, I_\gamma\}$, one can make  use of \eqref{cotadb} and \eqref{P5.8.2} to get
    \begin{align}\label{5.8.6}
        |a^{(2,1)}_{n,i}|=\left|
        \frac{d_{n,i}^{(0,2)}}{\mathcal{C}_{n,i}}+2\frac{r_n}{s_n}\frac{d_{n,i}^{(1,1)}}{\mathcal{C}_{n,i}}\cos (\alpha_n-\beta_n)+\left(\frac{r_n}{s_n}\right)^2\frac{d_{n,i}^{(2,0)}}{\mathcal{C}_{n,i}}\cos^2(\alpha_n-\beta_n)\right|\leq C.
    \end{align}
 We also apply \eqref{estimastlargo} in the $n-1$ case and the relation \eqref{des} that lead to
    \begin{align}\label{5.8.5}
        1 \leq \left(\frac{s_n(t)}{s_n(t_{n+1})}\right)^{1+\gamma}\leq \left(e^{C N_{n-1}^{\mu+\sigma} N_n^{-\mu}}\right)^{1+\gamma}\leq e^{(1+\gamma)N_n^{-\frac{\gamma\mu}{10}}}\leq 1+N_n^{-\frac{\gamma \mu}{20}}
        ,\qquad  \forall t\geq t_{n+1}.
    \end{align}
    Assume that $t\geq t_{n+1}$. According to \eqref{defatilde}, the monotonicity properties of $s_n$ on $[t_{n+1}, 1]$ given by the inductive hypothesis (see \eqref{estimastlargo}), \eqref{5.8.5}, \eqref{5.8.6},  \eqref{P5.8.1} and \eqref{cotaan21}: 
    \begin{align*}
    |\tilde{a}_n^{(2,1)}(t)-&s_n(t_{n+1})^{1+\gamma}\mathcal{C}_{n,0}a_{n,0}^{(2,1)}(t)|=
    \left|
s_n(t)^{1+\gamma}\sum_{i=0}^{I_\gamma} \mathcal{C}_{n,i}(t) a_{n,i}^{(2,1)}(t)-s_n(t_{n+1})^{1+\gamma}\mathcal{C}_{n,0}a_{n,0}^{(2,1)}(t)
    \right|
    \\
&\leq \left|\mathcal{C}_{n,0}a_{n,0}^{(2,1)}(t)s_n(t_{n+1})^{1+\gamma}\left(1-\left(\frac{s_n(t)}{s_{n}(t_{n+1})}\right)^{1+\gamma}\right)\right|+\left|
s_n(t)^{1+\gamma}\sum_{i=1}^{I_\gamma} \mathcal{C}_{n,i}(t) a_{n,i}^{(2,1)}(t)\right|\\
    & \leq N_n^{-\frac{\gamma\mu}{20}}\mathcal{C}_{n,0}a_{n,0}^{(2,1)}(t)s_n(t_{n+1})^{1+\gamma} + CN_n^{\varepsilon-2\frac{\varepsilon+\delta}{1+\gamma+\delta}} \mathcal{C}_{n, 0} s_n(t)^{1+\gamma} \\
    & \leq C N_n^{-\frac{\gamma\mu}{20}}\mathcal{C}_{n,0} s_n(t_{n+1})^{1+\gamma} + CN_n^{\varepsilon-2\frac{\varepsilon+\delta}{1+\gamma+\delta}} \mathcal{C}_{n, 0} s_n(t)^{1+\gamma}.
    \end{align*}
    Recalling the formula involving $N_n^{\varepsilon}$ (validated by the inductive hypothesis) given by \eqref{defN}, using the previous computation, \eqref{cotaan21} and the estimate $s_n(t)\leq CN_n$,
    \begin{align*}
     |\tilde{a}_n^{(2,1)}(t)-N_n^{\varepsilon}|&=
    |\tilde{a}_n^{(2,1)}(t)-s_n(t_{n+1})^{1+\gamma}\mathcal{C}_{n,0}a_{n,0}^{(2,1)}(t_{n+1})| \\
   & \leq  |\tilde{a}_n^{(2,1)}(t)-s_n(t_{n+1})^{1+\gamma}\mathcal{C}_{n,0}a_{n,0}^{(2,1)}(t)|
    +  \mathcal{C}_{n,0} s_n(t_{n+1})^{1+\gamma}|a_{n,0}^{(2,1)}(t)-a_{n,0}^{(2,1)}(t_{n+1})|
    \\& \leq  |\tilde{a}_n^{(2,1)}(t)-s_n(t_{n+1})^{1+\gamma}\mathcal{C}_{n,0}a_{n,0}^{(2,1)}(t)|\\
    &\hspace{5mm}+ \mathcal{C}_{n,0} s_n(t_{n+1})^{1+\gamma}\left(|a_{n,0}^{(2,1)}(t)-1|+|a_{n,0}^{(2,1)}(t_{n+1})-1|
    \right) \\
    & \leq C N_n^{-\frac{\gamma\mu}{20}}\mathcal{C}_{n,0} s_n(t_{n+1})^{1+\gamma} + CN_n^{\varepsilon-2\frac{\varepsilon+\delta}{1+\gamma+\delta}} \mathcal{C}_{n, 0} s_n(t)^{1+\gamma} + C \mathcal{C}_{n,0} s_n(t_{n+1})^{1+\gamma} N_{n}^{-\frac{(\varepsilon-\mu)}{20}\gamma+\sigma} \\
    & \leq C \mathcal{C}_{n, 0} N_n^{1+\gamma} (N_n^{-\frac{\gamma\mu}{20}} + N_n^{\varepsilon-2\frac{\varepsilon+\delta}{1+\gamma+\delta}} + N_{n}^{-\frac{(\varepsilon-\mu)}{20}\gamma+\sigma})  \leq N_n^{\varepsilon-3 \sigma}. 
    \end{align*}
    This provides the desired result \eqref{coeficienteppal}. 
    
    Observe that  \eqref{cotab} is a simple consequence of \eqref{P6.8.2}, \eqref{cotadb} and \eqref{P5.8.1}.

    \item \textbf{Dynamics of $\theta_{n+1}$.} The goal of this step is to show (\ref{estimaalpha})--\eqref{estimastlargo}. By \eqref{P2.8.3}, $h_j(t)=1$ for $t\geq t_{n+1}$ and $j\leq n$, so according to (\ref{velocidadTheta}) the linear part of the velocity generated by $\Theta_n$ is given by
    \begin{align*}
        \overline{v}^\gamma(\Theta_n)=&\sum_{j=0}^n 
\tilde{a}_j^{(2,1)}\left(\begin{matrix}
              -\frac{1}{2}\sin 2\beta_j&-\sin^2\beta_j\\
              \cos^2 \beta_j&\frac{1}{2}\sin 2\beta_j 
          \end{matrix}\right)x\\
          &+\sum_{j=0}^n \left[   \tilde{a}^{(1,1)}_j \left(\begin{matrix}
              \cos 2\beta_j& \sin 2\beta_j \\
              \sin 2\beta_j  &-\cos 2 \beta_j
          \end{matrix}\right)
          +\tilde{a}_j^{(1,2)} \left(\begin{matrix}
              -\frac{1}{2}\sin 2\beta_j &\cos^2\beta_j\\
              -\sin^2 \beta_j&\frac{1}{2}\sin 2\beta_j 
          \end{matrix}\right)\right]x,
    \end{align*}
    for any $t\in[t_{n+1},1]$; see also \eqref{P13.7.6}. Recall that the principal layer  $f_{n+1,0}$ is transported by $\overline{v}^\gamma(\Theta_n)$ (more precisely, see \eqref{transporteppal}). Then according to Lemma \ref{dinamica}, $\alpha_{n+1}$, $r_{n+1}$ verify \eqref{dinamicaalpha} and $\beta_{n+1}$, $s_{n+1}$ verify \eqref{dinamicabeta} with coefficients $\overline{a}_n^{(j,k)}$ given by \eqref{abarra} and initial conditions \eqref{ic}. Note that assumptions of Lemma \ref{lemaedos} are verified \footnote{Note that all terms in the sequence $\{N_n\}$ should satisfy \eqref{des}. However, at this point, we have only constructed $N_0, \ldots, N_n$ satisfying the stronger relations \eqref{crecimientoN} (see also Remark \ref{Remark25}), while the terms $N_j$ with $j > n$ remain unknown. This is merely a technical obstruction that can be easily overcome. Indeed,  we may apriori work  with any $\{N_n\}$ verifying \eqref{des} and then we will construct the sequence $\{N_n\}$ given by \eqref{defN} for which \eqref{des} holds.}: \eqref{cotascoeficientes} follows automatically from \eqref{cotascoeficientestilde} and \eqref{coeficienteppal}; while \eqref{betaanterior},  \eqref{alphaanterior} and \eqref{sestatico} are consequences of the inductive hypothesis applied to \eqref{estimabetatlargo}, \eqref{estimaalpha} and \eqref{estimastlargo}, respectively. Hence Lemma \ref{lemaedos} yields  \eqref{estimaalpha}, \eqref{estimar} and the second part of \eqref{estimastlargo}. To get \eqref{estimabeta} and \eqref{estimabetatlargo} with the help of \eqref{estimabeta2},  one may follow line by line the arguments carried out in the second step of this proof for the case $\beta_1 -\beta_0$.   We can proceed analogously as in the second step of this proof to achieve (\ref{estimas}), namely, it follows from \eqref{estimas2} and an application of the Mean Value Theorem. The behavior after $t_{n+2}$ of $s_{n+1}$ expressed in the first part of (\ref{estimastlargo}) is a consequence of (\ref{sestatico2}) and (\ref{estimas}). We leave further details to the reader.


    \item \textbf{Existence of $N_{n+1}$ defined by \eqref{defN} and relation between $N_{n}$ and $N_{n+1}$.} We shall follow the line of reasoning proposed in the third step of this proof related to $N_0$ and $N_1$. For convenience of the reader, we next sketch the proof.

    The coefficient $a_{n+1,0}^{(2,1)}$ is given by Lemma \ref{lemavelocidad}: 
    \begin{align*}
         a^{(2,1)}_{n+1,0}=&-\frac{1}{\mathcal{C}_{n+1,0}}\left[d^{(0,2)}_{n+1,0}+2\frac{r_{n+1}}{s_{n+1}}d^{(1,1)}_{n+1,0}\cos (\alpha_{n+1}-\beta_{n+1})\right. \\
         &\left. \hspace{1cm}+\left(\frac{r_{n+1}}{s_{n+1}}\right)^2d^{(2,0)}_{n+1,0}\cos^2(\alpha_{n+1}-\beta_{n+1})\right]. 
    \end{align*}
    Proceeding similarly as in \eqref{cotaan21} for $a_{n,0}^{(2,1)}$, we obtain 
    \begin{align*}
       |a_{n+1,0}^{(2,1)}(t)-1| \leq N_{n}^{-\frac{\varepsilon-\mu}{5}\gamma+\frac{\sigma\gamma}{40}}\leq N_{n+1}^{-\frac{\varepsilon-\mu}{20}\gamma+\sigma}, \qquad  \forall t\in [t_{n+2},1],
    \end{align*}
    where $N_{n+1}$ is any number satisfying \eqref{des}, and in particular 
    \begin{align}\label{P6.8.3}
        \left(\frac{9}{10}\right)^{(\varepsilon+\delta)\frac{1+\gamma}{1+\gamma+\delta}}\leq a_{n+1,0}^{(2,1)}(t_{n+2})\leq \left(\frac{11}{10}\right)^{(\varepsilon+\delta)\frac{1+\gamma}{1+\gamma+\delta}}
    \end{align}
    under the assumption that $N_0$ is large enough. 
    
    Notice that the behavior of  $s_{n+1}$ was already obtained in  (\ref{estimas}) and (\ref{estimastlargo}). Using that $N_0$ is large  enough, one can obtain (see \eqref{P24.7.2}) 
    \begin{align}\label{P6.8.4}
         \left(\frac{9}{10}\right)^{\frac{\varepsilon+\delta}{1+\gamma+\delta}} \, N_n^{\varepsilon-\mu}\leq \frac{s_{n+1}(t_{n+2})}{s_{n+1}(t_{n+1})}\leq \left(\frac{11}{10}\right)^{\frac{\varepsilon+\delta}{1+\gamma+\delta}} \,  N_n^{\varepsilon-\mu}.
    \end{align}
    
  As a byproduct of \eqref{P6.8.3} and \eqref{P6.8.4}, we get that (see \eqref{P24.7.3})
\begin{align}\label{P7.8.2}
\left(\frac{1}{2}\right)^{(\varepsilon+\delta)\frac{1+\gamma}{1+\gamma+\delta}}N_n^{(1+\gamma)(\varepsilon-\mu)}\leq 
    \left(\frac{s_{n+1}(t_{n+2})}{s_{n+1}(t_{n+1})}\right)^{1+\gamma}a_{n+1,0}^{(2,1)}(t_{n+2})\leq 2^{(\varepsilon+\delta)\frac{1+\gamma}{1+\gamma+\delta}}N_n^{(1+\gamma)(\varepsilon-\mu)}.
\end{align}
 Now, proceeding by the same Bolzano argument that we implemented in the initial case of the induction (see Step \ref{existenciaN1}), we can guarantee the existence of $N_{n+1}$ such that
 $$
        N_{n+1}^{\frac{1+\gamma}{1+\gamma+\delta} (\varepsilon+\delta)} = \left(\frac{s_{n+1}(t_{n+2})}{s_{n+1}(t_{n+1})}\right)^{1+\gamma}a_{n+1,0}^{(2,1)}(t_{n+2})
 $$
 or equivalently (see \eqref{ic} and \eqref{P7.8.1})
 $$
    \frac{N_{n+1}^\varepsilon}{s_{n+1}(t_{n+1})^{1+\gamma} \mathcal{C}_{n+1, 0}} = \left(\frac{s_{n+1}(t_{n+2})}{s_{n+1}(t_{n+1})}\right)^{1+\gamma}a_{n+1,0}^{(2,1)}(t_{n+2}).
 $$
 Accordingly, there exists $N_{n+1}$ such that  \eqref{defN} holds. In addition, using \eqref{P7.8.2} it is easy to check the validity of \eqref{crecimientoN}.
     \end{enumerate}
     \end{proof}

 \section{Control of the forces}\label{SF}

The goal of this section is to estimate the Sobolev norms in the well-posedness regime of the forces associated to the solution constructed in Section \ref{construccion}. Specifically, we wish to study the following forces 
 \begin{align}\label{P14.8.10}
     F_n:=h_n\left[
(v^\gamma-\overline{v}^\gamma)(\Theta_{n-1})\cdot \nabla \theta_n+v^\gamma(\theta_n)\cdot \nabla\Theta_{n-1}+h_n\sum_{(j_1,j_2)\in \mathcal{J}_{I_\gamma}} v^\gamma(f_{n,j_1})\cdot \nabla f_{n,j_2}
     \right]+h'_n \theta_n
 \end{align}
for  $n\in \mathbb{N}$. The zero term is  
 \begin{align}\label{P14.8.11}
     F_0:=v^\gamma(\theta_0)\cdot\nabla\theta_0,
 \end{align}
 which is clearly smooth.
 
 \begin{lma}
 \label{fuerzaanteriores}
 Let $\kappa \geq 0$. Then
     \begin{align}\label{P11.8.4}
         \|v^\gamma(\theta_{n})\cdot \nabla \Theta_{n-1}\|_{H^\kappa}\leq CN_{n-1}^{1+\varepsilon} N_n^{-3-\gamma+2\frac{\varepsilon+\delta}{1+\gamma+\delta}+\varepsilon+\kappa},         \qquad \forall t\in [t_{n},1].
     \end{align}    
 \end{lma}
 \begin{proof}
     First, let us prove that 
     \begin{align}
     \label{centroplano}
         \nabla \Theta_{n-1}(x,t)=0, \qquad  \forall x \in B_0((10C'N_{n-1})^{-1}),
     \end{align}
     for any $t\in[t_n,1]$, where $C'>0$ is the same constant as in  \eqref{soportesencajados}. Let $k\in\{1,\ldots ,n-1\}$. Since $\phi$ is constant on the interval $(-1/10,1/10)$ and $f_{k,0}$ solves \eqref{transporteppal}, then $f_{k,0}$ is also constant on 
     \begin{align*}
        \mathcal{S}_{k}(t)/10:= \{x\in\mathbb{R}^2 :  10x\in\mathcal{S}_{k}(t)\};
     \end{align*}
     see also \eqref{P10.8.1}. 
     Clearly, 
     \begin{align*}
         \nabla f_{k,0}(x,t)=0, \qquad  \forall x\in\mathcal{S}_{k}(t)/10
     \end{align*}
     and, in addition, this also holds with $k=0$ (see \eqref{P13.7.1}). 
   As a consequence, since $f_{k,1}$ solves \eqref{defpert} with zero initial condition and a force that vanishes on $\mathcal{S}_k(t)/10$ then 
   \begin{align*}
         f_{k,1}(x,t)=\nabla f_{k,1}(x,t)=0, \qquad  \forall x\in\mathcal{S}_{k}(t)/10.
     \end{align*}
     In addition, the same reasoning can be done with the rest of corrections $f_{k, j}$ and, by an inductive argument, one has that
     \begin{align}\label{P11.8.1}
         f_{k,j}(x,t)=\nabla f_{k,j}(x,t)=0, \qquad  \forall x\in\mathcal{S}_{k}(t)/10,
     \end{align}
     for $j=1,\ldots ,I_\gamma$. With this, using that the supports are nested in the sense of \eqref{soportesencajados}, one concludes \eqref{centroplano} (see also \eqref{deftheta} and \eqref{sumacapas}).

     On the other hand, it is clear that 
     $$
        \nabla \Theta_{n-1}(x, t) =0, \qquad \forall x \not \in \mathcal{S}_0, 
     $$
     for any $t \in [t_n, 1]$.  

    It remains to study $v^\gamma(\theta_{n})\cdot \nabla \Theta_{n-1}$ in $\mathcal{S}_{0}\setminus B_0((10C'N_{n-1})^{-1})$. Recall that $f_{k,j}$ generates a sequence of derivatives, so we put $f_{k,j}^{(0,0)} = f_{k,j}$, and $\text{supp } f_{k, j}^{(0, 0)} (\cdot, t) \subset \mathcal{S}_k(t) \subset B_0(C N_k^{-1+2 \frac{\varepsilon + \delta}{1+\gamma + \delta}})$ (see \eqref{soportesencajados}).  Let $x\in \mathbb{R}^2\setminus B_0((10C'N_{n-1})^{-1})$. In particular, $|y| \leq |x|/2$ if $y \in \text{supp } f_{n, i}^{(0, 0)} (\cdot, t)$. Then
	\begin{align}\label{P11.8.2}
		| v^\gamma(f_{n,i}^{(0,0)})(x,t)| &\leq C\int_{\text{supp} f_{n, i}^{(0, 0)} (\cdot, t)}\frac{|f_{n,i}^{(0,0)}(y,t)|}{|x-y|^{2+\gamma}} \, dy 
        \leq C \, \frac{\mathcal{C}_{n,i}(t) N_{n}^{-2(1-\frac{\varepsilon+\delta}{1+\gamma+\delta})}}{|x|^{2+\gamma}}.
	\end{align}
    	
        Using that $r_k(t) \leq C s_k(t) \leq C N_k$ if $t \geq t_k$,  
        there exist $\eta_1, \eta_2\in \mathbb{R}^2$ such that
	\begin{align}
		|\nabla f_{k,j}^{(0,0)}(x,t)|
		&\leq r_k(t)|f_{k,j}^{(1,0)}(x, t)|+s_k(t)|f_{k,j}^{(0,1)}(x, t)| \nonumber \\
        &\leq C N_k(|\nabla f_{k,j}^{(1,0)}(\eta_1,t)|+|\nabla f_{k,j}^{(0,1)}(\eta_2,t)|)|x| \nonumber
		\\
        &\leq C\mathcal{C}_{k,j}(t) N_k^{2}|x|\ind_{\mathbb{R}^2\setminus(\mathcal{S}_k(t)/10)}, \qquad  \forall x\in \mathbb{R}^2, \label{P11.8.3}
	\end{align}
	where we have applied that $f^{(l,m)}_{k,j}(0,t)=0$ for $l+m>0$ and $t\in[t_k,1]$ (see \ref{derivadas} in Definition \ref{sod}) in the penultimate step and \eqref{P11.8.1} and \ref{amplitud} from Definition \ref{sod} in the ultimate step.

    As a combination of \eqref{P11.8.2} and \eqref{P11.8.3}, together with \eqref{P7.8.1} and \eqref{P5.8.1}, we get
	\begin{align*}
		\|v^\gamma(f_{n,i}^{(0,0)})(\cdot,t)\cdot \nabla f_{k,j}^{(0,0)}(\cdot,t)\|_{L^2} 
		& \leq C \mathcal{C}_{k,j}(t) \mathcal{C}_{n,i}(t) N_k^2 N_n^{-2(1-\frac{\varepsilon+\delta}{1+\gamma+\delta})}\left(\int_{|x|\geq (10C'N_{n-1})^{-1}}\frac{dx}{|x|^{2+2\gamma}}\right)^{\frac{1}{2}}
		\\&\leq C N_{k}^{1-\gamma+\varepsilon} N_n^{-3-\gamma+\varepsilon+2\frac{\varepsilon+\delta}{1+\gamma+\delta}}N_{n-1}^\gamma.
	\end{align*} 
    This leads to 
    \begin{align*}
         \|v^\gamma(\theta_{n})\cdot \nabla \Theta_{n-1}\|_{L^2}\leq C N_{n-1}^\gamma N_n^{-3-\gamma+\varepsilon+2\frac{\varepsilon+\delta}{1+\gamma+\delta}} \sum_{k=0}^{n-1} N_k^{1-\gamma+\varepsilon}  \leq C N_{n-1}^{1+\varepsilon}  N_n^{-3-\gamma+\varepsilon+2\frac{\varepsilon+\delta}{1+\gamma+\delta}},   
    \end{align*}
    that is, \eqref{P11.8.4} is satisfied with $\kappa =0$.

 Next, we focus on the $\dot{H}^\kappa$ semi-norm with $\kappa \in \mathbb{N}$ of the functions $v^\gamma(f_{n, i}^{(0, 0)}) \cdot \nabla f_{k, j}^{(0, 0)}$. Keeping  in mind that $f_{k,j}^{(0,0)}$ generates a sequence of derivatives, we observe that  any partial derivative of order $\kappa$ of $\nabla f_{k,j}^{(0,0)}$ consists of a finite sum of functions vanishing on $\mathcal{S}_k(t)/10$ (see \eqref{P11.8.1}), which have the same principal frequencies $r_k, s_k$ and amplitudes  $N_k^\kappa \mathcal{C}_{k, j}$. Furthermore, given any multi-index $(\kappa_1, \kappa_2) \in \mathbb{N}_0^2$,  we have that $D^{(\kappa_1, \kappa_2)} v^\gamma(f_{n, i}^{(0, 0)}) = v^\gamma (D^{(\kappa_1, \kappa_2)} f_{n, i}^{(0, 0)})$, where the amplitudes of  $D^{(\kappa_1, \kappa_2)} f_{n, i}^{(0, 0)}$ are $N_n^{\kappa_1 + \kappa_2}\mathcal{C}_{n, i}$. Hence, for each  $(\kappa_1, \kappa_2) \in \mathbb{N}_0^2$ with $\kappa_1 + \kappa_2 = \kappa$,  by the same reasoning as applied above for the case $\kappa=0$ it is easy to see that 
        \begin{align*}
            \big|D^{(\kappa_1,\kappa_2)} \big(v^\gamma(f_{n,i}^{(0,0)})\cdot \nabla f_{k,j}^{(0,0)} \big)(x, t) \big|\leq C \mathcal{C}_{k,0} \mathcal{C}_{n,0} N_k^2 N_n^{-2(1-\frac{\varepsilon+\delta}{1+\gamma+\delta})}\frac{1}{|x|^{1+\gamma}}N_n^\kappa \ind_{\mathbb{R}^2\setminus(\mathcal{S}_k(t)/10)}.
        \end{align*}
As an immediate consequence, we derive 
\begin{align*}
    \|v^\gamma(\theta_{n})\cdot \nabla \Theta_{n-1}\|_{\dot{H}^{\kappa}} \leq  C N_{n-1}^{1+\varepsilon} N_n^{-3-\gamma+\varepsilon+2\frac{\varepsilon+\delta}{1+\gamma+\delta} + \kappa}
\end{align*}
and so
\begin{align*}
    \|v^\gamma(\theta_{n})\cdot \nabla \Theta_{n-1}\|_{H^{\kappa}} \leq C ( \|v^\gamma(\theta_{n})\cdot \nabla \Theta_{n-1}\|_{L^2} + \|v^\gamma(\theta_{n})\cdot \nabla \Theta_{n-1}\|_{\dot{H}^{\kappa}}) \leq  C N_{n-1}^{1+\varepsilon} N_n^{-3-\gamma+\varepsilon+2\frac{\varepsilon+\delta}{1+\gamma+\delta} + \kappa}.
\end{align*}
This proves the desired result \eqref{P11.8.4} in the case $\kappa \in \mathbb{N}$. 

    The non-integer case  $\kappa>0$ can easily be obtained from the integer case via classical interpolation. Further details are left to the interested reader. 
 \end{proof}
 
To estimate the Sobolev norms of the term $(v^\gamma-\overline{v}^\gamma)(\Theta_{n-1}) \cdot \nabla \theta_n$ in $F_n$, we will make use of the following auxiliary result. 
 
 \begin{lma}
 \label{corolariocubico}
      Let $\{f^{(j,k)}\}$ be a sequence of derivatives with amplitude $\mathcal{C}(t)$, principal directions $(\alpha(t),\beta(t))$ and principal frequencies $(r(t),s(t))$ such that $r(t)\leq s(t)$. Let $(\kappa_1, \kappa_2) \in \mathbb{N}_0^2$ with $\kappa_1 +\kappa_2 \leq 3$. Then
    \begin{align}
    \label{derivadasvelocidad}
               | D^{(\kappa_1, \kappa_2) } (v^\gamma-\overline{v}^\gamma)(f^{(0,0)})(x,t)|\leq C\mathcal{C}(t)s(t)^{3+\gamma}\left(s(t)^{\kappa_1+\kappa_2}|x|^3+|x|^{3-\kappa_1-\kappa_2} \right)
    \end{align}
    for all $x\in \mathcal{S}(\alpha(t),\beta(t),r(t),s(t)).$
     \end{lma}
    \begin{proof}
        In view of  Lemma \ref{lemavelocidad}, we can express 
            \begin{align*}
                (v^\gamma-\overline{v}^\gamma)(f^{(0,0)})(x,t) &=        -\textbf{n}(\alpha)^\perp \frac{r}{2 s^{1-\gamma}} \sum_{l=0}^3b^{(4-l,l)}\left(\begin{matrix}
                    3\\ l
                \end{matrix}\right)
                \left(r \textbf{n}(\alpha) \cdot x\right)^{3-l} (s\textbf{n}(\beta)\cdot x)^l \\
& \hspace{1cm}- \textbf{n}(\beta)^\perp  \frac{s^\gamma}{2}\sum_{l=0}^3b^{(3-l,l+1)}\left(\begin{matrix}
                    3\\ l
                \end{matrix}\right)
                \left(r \textbf{n}(\alpha) \cdot x\right)^{3-l} (s\textbf{n}(\beta)\cdot x)^l, 
            \end{align*}
            where $|D^{(\kappa_1, \kappa_2)}b^{(j, k)}| \lesssim_{\gamma, j, k} \mathcal{C} s^{\kappa_1 + \kappa_2}$ (which is a simple extension of the computation exhibited in \eqref{P11.8.5} for the case $\kappa_1 + \kappa_2 =0$). Then a simple application of the Leibniz rule gives \eqref{derivadasvelocidad}.  
    \end{proof}

 \begin{lma} \label{fuerzacubico}
 Let $0\leq \kappa\leq 3$. Then 
     \begin{align}\label{P13.8.5}
         \|(v^\gamma-\overline{v}^\gamma)(\Theta_{n-1})\cdot \nabla \theta_n\|_{H^\kappa}\leq CN_{n-1}^{2+\varepsilon}N_n^{-4-\gamma+\kappa+7\frac{\varepsilon+\delta}{1+\gamma+\delta}+\varepsilon}, 
         \qquad  \forall t\in [t_n,1]
     \end{align}
     \end{lma}
     \begin{proof} 
Recall from \eqref{soportesencajados} that 
$$\text{supp } f_{n,i} (\cdot, t) \subset   \mathcal{S}_n(t)\subset B_{0}(CN_n^{-1+2\frac{\varepsilon+\delta}{1+\gamma+\delta}})$$
and
$$
    \mathcal{S}_n(t) \subset \mathcal{S}_k(t)
$$
for all $t\in[t_n,1]$ and $k\in \{0, \ldots ,n-1\}$. Once again, recall that the functions $f_{n,i}$ generate a sequence of derivatives and let  $f_{n,i}^{(0,0)}=f_{n,i}$. Let $(\kappa_1,\kappa_2)\in  \mathbb{N}_0^2$ be such that $\kappa_1+\kappa_2\leq 3$. Invoking Lemma \ref{corolariocubico}, together with the facts  $\mathcal{C}_{k, i} \lesssim N_k^{-1-\gamma + \varepsilon}, \, s_k \lesssim N_k$ and \eqref{P4.8.3},  we infer that  
\begin{align*}
    \|D^{(\kappa_1,\kappa_2)}&\left((v^\gamma-\overline{v}^\gamma)(f_{k,i}^{(0,0)})\cdot \nabla f_{n,j}^{(0,0)}\right)\|_{L^2}\leq |\mathcal{S}_n|^{\frac{1}{2}}\sum_{m=0}^{\kappa_1+\kappa_2} \|(v^\gamma-\overline{v}^\gamma)(f_{k,i}^{(0,0)})\|_{C^m(\mathcal{S}_n)} \|f_{n,j}^{(0,0)}\|_{C^{1+\kappa_1+\kappa_2-m}(\mathcal{S}_n)}\\
    &\leq C (r_n s_n \sin (\alpha_n-\beta_n))^{-\frac{1}{2}} \sum_{m=0}^{\kappa_1+\kappa_2} \mathcal{C}_{k,i}s_k^{3+\gamma}\sup_{x\in \mathcal{S}_n}\left(s_k^{m}|x|^3+|x|^{3-m} \right) \mathcal{C}_{n,j} s_n^{1+\kappa_1+\kappa_2-m}\\
     &\leq CN_{n}^{-1+\frac{\varepsilon+\delta}{1+\gamma+\delta}} \sum_{m=0}^{\kappa_1+\kappa_2} N_k^{-1-\gamma+\varepsilon} N_k^{3+\gamma}  N_n^{3(-1+2 \frac{\varepsilon + \delta}{1+\gamma+\delta})} \Big(N_k^m  + N_n^{-m(-1+2 \frac{\varepsilon + \delta}{1+\gamma+\delta})} \Big) N_n^{-1-\gamma+\varepsilon}N_n^{1+\kappa_1+\kappa_2-m}
    \\
    & \leq C N_n^{-4-\gamma+\kappa_1+\kappa_2+7\frac{\varepsilon+\delta}{1+\gamma+\delta}+\varepsilon}   N_k^{2+\varepsilon}  \sum_{m=0}^{\kappa_1+\kappa_2}  N_n^{-m(-1+2 \frac{\varepsilon + \delta}{1+\gamma+\delta})} N_n^{-m} \leq CN_n^{-4-\gamma+\kappa_1+\kappa_2+7\frac{\varepsilon+\delta}{1+\gamma+\delta}+\varepsilon}N_k^{2+\varepsilon}.
\end{align*}
Summing the last expression over all indices  $0 \leq k\leq n-1$ and $i, j =0, \ldots, I_\gamma$,  we arrive at
$$
    \|(v^\gamma-\overline{v}^\gamma)(\Theta_{n-1}) \cdot \nabla \theta_n\|_{H^\kappa} \leq C N_{n-1}^{2+\varepsilon} N_n^{-4-\gamma+\kappa+7\frac{\varepsilon+\delta}{1+\gamma+\delta}+\varepsilon}.
$$
Finally, interpolating one gets the result for non-integers $0\leq \kappa\leq 3$.
     \end{proof}

In the next step we estimate the Sobolev norm of the self-interaction term in $F_n$.

\begin{lma}\label{fuerzapert} Let $0\leq \kappa\leq 3$. There exists $I_\gamma \in \mathbb{N}$ (depending only on $\gamma$) such that  
    \begin{align}\label{P13.8.2}
        \left\|
\sum_{(j_1,j_2)\in \mathcal{J}_{I_\gamma}} v^\gamma(f_{n,j_1})\cdot \nabla f_{n,j_2}
        \right\|_{H^\kappa}\leq N_n^{-2-\gamma+\kappa-\delta},\qquad  \forall t\in[t_n,1].
    \end{align}
    \end{lma}
    \begin{proof}
        The perturbations $f_{n,j}$ with $j=1, \ldots, I_\gamma$ are defined by (\ref{defpert}). Following  Lemmas \ref{nuevaperturbacion} and \ref{nuevaperturbacion2}, each $f_{n,j}$ generates a sequence of derivatives with principal frequencies $(r_n,s_n)$, principal directions $(\alpha_n,\beta_n)$ and amplitude $\mathcal{C}_{n,j}$. As usual, we let $f_{n,j}^{(0,0)} = f_{n,j}$.  Furthermore, elementary computations show that 
        \begin{align*}
            v^\gamma(f_{n,j_1}^{(0,0)})\cdot \nabla f_{n,j_2}^{(0,0)}=r_n s_n \sin(\alpha_n-\beta_n)
             \left(f_{n,j_2}^{(0,1)} (-\Delta)^{-\frac{1-\gamma}{2}}f_{n,j_1}^{(1,0)}
             -f_{n,j_2}^{(1,0)} (-\Delta)^{-\frac{1-\gamma}{2}}f_{n,j_1}^{(0,1)}
             \right).
        \end{align*}
        
        Recall that, by \eqref{P12.8.1}, \eqref{P5.8.1} and \eqref{P7.8.1},  for each $j=0,\ldots, I_\gamma$, 
          \begin{align*}
            |(-\Delta)^{-\frac{1-\gamma}{2}}f_{n,j}^{(l,m)}|\leq C s^{-1+\gamma}_n \mathcal{C}_{n,j},
        \end{align*}
        \begin{align*}
            \mathcal{C}_{n,j}\leq C \mathcal{C}_{n,0}N_{n}^{-j (2\frac{\varepsilon+\delta}{1+\gamma+\delta}-\varepsilon)} = C N_n^{-1-\gamma+\varepsilon} N_{n}^{-j (2\frac{\varepsilon+\delta}{1+\gamma+\delta}-\varepsilon)},   
        \end{align*}
        and that the support of $f_{n,j}$ is a subset of $\mathcal{S}_n$ with (see \eqref{P12.8.2})
        \begin{align}\label{P12.8.5}
        |\mathcal{S}_n| =4 (r_n s_n |\sin(\alpha_n-\beta_n)|)^{-1} = N_n^{-2(1-\frac{\varepsilon+\delta}{1+\gamma+\delta})}.
        \end{align}
        Then
        \begin{align*}
            \| v^\gamma(f_{n,j_1}^{(0,0)})\cdot \nabla f_{n,j_2}^{(0,0)}\|_{L^2} &\leq  r_ns_n\sin(\alpha_n-\beta_n)
            \left(
        \|f_{n,j_2}^{(0,1)} (-\Delta)^{-\frac{1-\gamma}{2}}f_{n,j_1}^{(1,0)}\|_{L^2} + \|f_{n,j_2}^{(1,0)} (-\Delta)^{-\frac{1-\gamma}{2}}f_{n,j_1}^{(0,1)}\|_{L^2}
            \right) \\
           & \leq C  s_n^{\gamma-1} (r_n s_n \sin(\alpha_n-\beta_n)) |\mathcal{S}_n|^{\frac{1}{2}}\mathcal{C}_{n,j_1} \mathcal{C}_{n,j_2} \\
           &  \leq C N_n^{(\gamma-1)(1-\frac{\varepsilon+\delta}{1+\gamma+\delta})} N_n^{1-\frac{\varepsilon+\delta}{1+\gamma+\delta}}  \mathcal{C}_{n,j_1} \mathcal{C}_{n,j_2}\\
           & \leq CN_n^{-2-\gamma +2\varepsilon-\gamma \frac{\varepsilon+\delta}{1+\gamma+\delta}-(j_1+j_2)(2\frac{\varepsilon+\delta}{1+\gamma+\delta}-\varepsilon)},
        \end{align*}
        where we have also used that $s_n(t) \geq C s_n(t_n) = C N_n^{1-\frac{\varepsilon+\delta}{1+\gamma+\delta}}$ (see \eqref{estimas} and \eqref{ic}).

Repeating the reasoning employed in the proof of Lemma \ref{fuerzaanteriores} (or in other words, for each derivative we pay an extra factor of order $N_n$) and noting  that $j_1+j_2\geq I_\gamma$ (since $(j_1, j_2) \in \mathcal{J}_{I_\gamma}$), the previous computations can be extended to obtain  
\begin{align}\label{P12.8.3}
    \| v^\gamma(f_{n,j_1}^{(0,0)})\cdot \nabla f_{n,j_2}^{(0,0)}\|_{H^\kappa}\leq CN_n^{-2-\gamma+\kappa +2\varepsilon-\gamma \frac{\varepsilon+\delta}{1+\gamma+\delta}-I_\gamma (2\frac{\varepsilon+\delta}{1+\gamma+\delta}-\varepsilon )}.
\end{align}
Let us choose $I_\gamma \in \mathbb{N}$ large enough such that the following inequality obeys
\begin{align}\label{P12.8.4}
    2\varepsilon-\gamma\frac{\varepsilon+\delta}{1+\gamma+\delta}-I_\gamma\left(2\frac{\varepsilon+\delta}{1+\gamma+\delta}-\varepsilon\right)<-2\delta. 
\end{align}
Hence, by  \eqref{P12.8.3} and \eqref{P12.8.4}, 
\begin{align*}
      \| v^\gamma(f_{n,j_1}^{(0,0)})\cdot \nabla f_{n,j_2}^{(0,0)}\|_{H^\kappa}\leq CN_n^{-2-\gamma+\kappa -2 \delta}
\end{align*}
and so
\begin{align*}
     \left\|
\sum_{(j_1,j_2)\in \mathcal{J}_{I_\gamma}} v^\gamma(f_{n,j_1})\cdot \nabla f_{n,j_2}
        \right\|_{H^\kappa} \leq C  N_n^{-2-\gamma + \kappa -2 \delta}  \leq  N_n^{-2-\gamma + \kappa - \delta},
\end{align*}
where the last step is justified provided that $N_0$ is sufficiently large (depending on $\gamma$). 
    \end{proof}

\begin{rmk}
    In view of \eqref{P12.8.4}, the constant $I_\gamma$ depends on $\gamma, \varepsilon$ and $\delta$. However, as we will see in Section \ref{SectionResult} below,  the parameters $\varepsilon$ and $\delta$ will only depend on $\gamma$ in an explicit way.    
\end{rmk}

\begin{lma}\label{fuerzaactivacion}
Let $0\leq \kappa \leq 3$
Then
\begin{align}\label{P13.8.1}
    \int_{t_n}^1h_n'(t)\|\theta_n(\cdot,t)\|_{H^\kappa} \, dt \leq CN_n^{-2-\gamma+\kappa+\varepsilon+(1-\kappa)\frac{\varepsilon+\delta}{1+\gamma+\delta}},
\end{align}
where $h_n$ satisfies \eqref{P2.8.3}. 
\end{lma}
    \begin{proof}
         In light of \eqref{estimas} and \eqref{ic}, we can estimate $s_n$ for $t \in [t_n,t_n+N_{n-1}^{-10\varepsilon}]$ as follows
        $$
            s_n (t) \leq C N_n^{1-\frac{\varepsilon+\delta}{1+\gamma+\delta}}(1+ N_{n-1}^\varepsilon (t-t_n)) \leq C N_n^{1-\frac{\varepsilon+\delta}{1+\gamma+\delta}}. 
        $$
        Assume first that $\kappa \in \mathbb{N}_0$. 
        Then, by \eqref{P12.8.5} and \eqref{P7.8.1}, we have
        \begin{align*}
            \|\theta_n(\cdot,t)\|_{H^\kappa}&\leq C |\text{supp } \theta_n|^{\frac{1}{2}} \|\theta_n\|_{C^\kappa}\leq CN_n^{-1+\frac{\varepsilon+\delta}{1+\gamma+\delta}} \mathcal{C}_{n, 0} s_n^\kappa \leq  CN_n^{-1+\frac{\varepsilon+\delta}{1+\gamma+\delta}}N_n^{-1-\gamma+\varepsilon}N_n^{(1-\frac{\varepsilon+\delta}{1+\gamma+\delta})\kappa} \\
            &=CN_n^{-2-\gamma+\kappa+\varepsilon+(1-\kappa)\frac{\varepsilon+\delta}{1+\gamma+\delta}}, \qquad  \forall t\in[t_n,t_n+N_{n-1}^{-10\varepsilon}]. 
        \end{align*}
        Moreover, this estimate can be extended to arbitrary $\kappa \geq 0$ by a simple interpolation argument. 
        Now, using the properties of $h_n$ stated in \eqref{P2.8.3}, we get
        \begin{align*}
\int_{t_n}^{1} h_n'(t)\|\theta_n(\cdot,t)\|_{H^\kappa} \, dt\leq CN_n^{-2-\gamma+\kappa+\varepsilon+(1-\kappa)\frac{\varepsilon+\delta}{1+\gamma+\delta}}\int_{t_n}^{t_n+N_{n-1}^{-10\varepsilon}}h_n'(t) \, dt\leq CN_n^{-2-\gamma+\kappa+\varepsilon+(1-\kappa)\frac{\varepsilon+\delta}{1+\gamma+\delta}}.
        \end{align*}
    \end{proof}

\section{The result}\label{SectionResult}
We are ready to prove the main result of this paper, namely, Theorem \ref{teorema}. For the sake of clarity,  we repeat its precise   statement below.

\begin{teorema*}
Let $\gamma\in(0,1)$. There exist
\begin{equation}\label{P13.8.4}
\kappa_0>2+\gamma+\frac{\gamma^2(1-\gamma)}{25(4+\gamma)},
\end{equation}
a scalar
\[
\theta:\mathbb{R}^2\times[0,1)\longrightarrow\mathbb{R},
\]
and a force
\[
F:\mathbb{R}^2\times[0,1)\longrightarrow\mathbb{R},
\]
with the following properties:
\begin{enumerate}
\item For every $T\in(0,1)$,
\[
\theta,\,F\in C^{\infty}\!\big([0,T];C^{\infty}_c(\mathbb{R}^2)\big),
\]
and $\theta$ is a classical solution of the forced $\gamma$-gSQG equation on
$[0,1)$: 
\[
\partial_t\theta(x,t)+v^{\gamma}(\theta)(x,t)\cdot\nabla\theta(x,t)=F(x,t)
\]
for every $(x,t)\in\mathbb{R}^2\times[0,1)$.

\item The force $F$ remains in the Sobolev well-posedness class up to and including
the singular time:
\[
F\in L^1\!\big([0,1];H^{\kappa}(\mathbb{R}^2)\big)
\quad\text{for every} \quad \kappa\in[2+\gamma,\kappa_0].
\]

\item For every $\kappa\in[2+\gamma,\kappa_0]$,
\[
\lim_{T\nearrow 1}\int_{0}^{T}\|\theta(\cdot,t)\|_{H^{\kappa}}\,dt=\infty.
\]
\item For every $\kappa\in\big[0,2+\gamma-\frac{\gamma(1-\gamma)}{2(4+\gamma)}\big]$,
\begin{align*}
    \sup_{0\leq t< 1} \|\theta(\cdot,t)\|_{H^\kappa}<\infty.
\end{align*}
\end{enumerate}
\end{teorema*}

\begin{rmk}
At this point, we can choose some specific values for the involved  parameters in the construction of the solution described in previous sections. Namely, 
\begin{equation}\label{eq:eleccion}
\varepsilon=\frac{\gamma(1-\gamma)}{4(4+\gamma)},\qquad
\delta=\frac{9\gamma\varepsilon}{50},\qquad
\mu=\frac{\delta}{20},\qquad
\sigma=\frac{\gamma^{2}\varepsilon}{20000}. 
\end{equation}
In particular, the required assumptions stated from \eqref{parametros1}  are satisfied for this choice.  
    We did not make an attempt to optimize this choice of the parameters and hence we did not optimize the Sobolev exponent $\kappa_0$ where the proof of the above result works. Nevertheless, we believe that such an optimal  range verifies 
   $        \kappa_0-2-\gamma\simeq \gamma^2(1-\gamma).
$
\end{rmk}

   \begin{proof}[Proof of Theorem]
       We choose $\theta$ and $F$ to be the functions defined in  Sections \ref{construccion} and \ref{SF}, respectively. We next review some of the most important features of the construction. Recall from (\ref{tiempos}) that $t_0=0$ and $t_n\nearrow 1$, specifically,  
       \begin{align*}
    t_n=1-N_{n-1}^{-\mu}, \qquad  \forall n\in \mathbb{N}.
\end{align*}
Let
\begin{align*}
    \theta(x,t)=\sum_{j=0}^\infty h_j(t)\theta_j(x,t), \qquad  \forall t\in [0,1), \qquad \forall x \in \mathbb{R}^2. 
\end{align*}
\begin{enumerate}
    \item \textbf{Regularity of the solution and the force for $t\in[0,1)$.} Notice that although $\theta$ is a priori defined  as an infinite sum, for any $t\in [0,1)$, there exists a unique $n\in\mathbb{N}_0$ such that $t\in[t_n,t_{n+1})$ and so
\begin{align}\label{P12.8.6}
    h_j(t)=0, \qquad  \forall j\geq n+1,
\end{align}
and we conclude that for any $t \in [0,1)$, $\theta$ is in fact a finite sum of $C^\infty$ functions. Indeed, $\theta_n$ is $C^\infty$ by construction, since (see \eqref{deftheta})
\begin{align*}
     \theta_n(x,t)=\sum_{i=0}^{I_\gamma} f_{n,i}(x,t),
\end{align*}
where $f_{n,0}$ is defined by (\ref{capaprincipal}) and it is obtained as the solution of a Cauchy problem with a $C^\infty$ initial condition and a transport PDE with a linear velocity (in space) with $C^\infty_t$ coefficients (in time). Namely, $f_{n,0}$ solves (see \eqref{transporteppal}, \eqref{P12.7.1})
\begin{align*}
   & \partial_t f_{n,0}+\overline{v}^\gamma(\Theta_{n-1})\cdot \nabla f_{n,0}=0,\\
    &f_{n,0}(x,t_n)=N_n^{-1-\gamma+\varepsilon}\phi(r_n(t_n)\textbf{n}(\alpha_n(t_n))\cdot x)\phi(s_n(t_n)\textbf{n}(\beta_n(t_n))\cdot x),
\end{align*} 
where $\Theta_{n-1}$ is defined by (\ref{sumacapas}) and $\overline{v}^\gamma$ corresponds to the linear part of the spatial Taylor expansion of $v^\gamma(\Theta_{n-1})$. The coefficients of $\overline{v}^\gamma(\Theta_0)$ are constant in time (and hence $C^\infty_t$). Then $f_{1,0}$ is $C_t^\infty C_x^\infty$ since it is the solution of a transport PDE with smooth initial conditions and linear velocity with constant coefficients. Also $f_{1,i}$ with $i=1, \ldots, I_\gamma$ are smooth because they solve (\ref{defpert}), that are transport PDEs with smooth force and   linear velocity. As a consequence,  $\theta_1$ and $h_1$ are smooth in space and time and so $\Theta_1$ and $\overline{v}^\gamma(\Theta_1)$ as well. With this recursive reasoning one can conclude that $v^\gamma(\Theta_{n-1})$ has smooth coefficients in time, and hence $f_{n,0}$ and its perturbations $f_{n,i}$ are smooth in space and time. Then, one can conclude that, any finite sum of  type 
\begin{align*}
    \sum_{j=0}^n h_j(t)\theta_j(x,t)
\end{align*}
is smooth in space and time and hence $ \theta \in C^\infty([0,T],C^\infty_c(\mathbb{R}^2))$  for any fixed  $T\in(0,1)$. Moreover, by virtue of  \eqref{soportesencajados}, one concludes that $\theta$ is compactly supported.

The force is defined by 
\begin{align*}
    F(x,t)=\sum_{j=0}^\infty F_j(x,t), \qquad \forall t \in [0, 1), \qquad \forall x \in \mathbb{R}^2, 
\end{align*}
where $F_j$ was introduced in \eqref{P14.8.10} and \eqref{P14.8.11}. For the same reasoning as before, for each fixed $t\in[0,1)$, $F(\cdot, t)$ becomes a finite sum. Notice that $F_j$ is a sum of terms that depend only on $f_{j,i}$ in such a way that everything is smooth with compact support. Then, it is clear that $ F \in C^\infty([0,T],C^\infty_c(\mathbb{R}^2))$  for any fixed  $T\in(0,1)$.

\item \textbf{Checking that $\theta$ solves gSQG with force $F$ for $t\in[0,1)$.} Given any fixed value $t\in [0,1)$, let $n \in \mathbb{N}_0$ be the unique value for which \eqref{P12.8.6} holds. We adopt the convention that 
\begin{align*}
  \Theta_{-1} \equiv 0 \qquad \text{and} \qquad   f_{0,i} \equiv 0 \qquad \forall i\in \mathbb{N}.  
\end{align*}
Then 
\begin{align*}
    \partial_t \theta(x,t)&+v^\gamma (\theta)(x,t)\cdot \nabla \theta(x,t) 
    =\sum_{j=0}^nh_j \left(
         \partial_t \theta_j+v^\gamma (\theta)\cdot \nabla \theta_j
    \right)
    +\sum_{j=0}^n h_j'\theta_j\\
    =&\sum_{j=0}^nh_j \left(
         \partial_t \theta_j+\overline{v}^\gamma (\Theta_{j-1})\cdot \nabla \theta_j
    \right)
    +\sum_{j=0}^nh_j (v^\gamma-\overline{v}^\gamma )(\Theta_{j-1})\cdot \nabla \theta_j\\
   & +\sum_{j=0}^nh_j v^\gamma (\theta-\Theta_{j-1})\cdot \nabla \theta_j
    +\sum_{j=0}^n h_j'\theta_j\\
    =&\sum_{j=0}^nh_j \sum_{i=0}^{I_\gamma}\left(
         \partial_t f_{j,i}+\overline{v}^\gamma (\Theta_{j-1})\cdot \nabla f_{j,i}
    \right)
    +\sum_{j=0}^nh_j (v^\gamma-\overline{v}^\gamma )(\Theta_{j-1})\cdot \nabla \theta_j\\
   & +\sum_{j=0}^nh_j v^\gamma (h_j\theta_j)\cdot \nabla \theta_j
    +\sum_{j=0}^nh_j v^\gamma \left(\sum_{l=j+1}^{n}h_l\theta_l\right)\cdot \nabla \theta_j
    +\sum_{j=0}^n h_j'\theta_j\\
    =&-\sum_{j=0}^nh_j^2 \sum_{i=1}^{I_\gamma}\sum_{(i_1,i_2)\in\mathcal{J}_{i-1}}v^\gamma(f_{j,i_1})\cdot \nabla f_{j,i_2}
    +\sum_{j=0}^nh_j (v^\gamma-\overline{v}^\gamma )(\Theta_{j-1})\cdot \nabla \theta_j\\
   & +\sum_{j=0}^nh_j^2 \sum_{i_1,i_2=0}^{I_\gamma} v^\gamma (f_{j,i_1})\cdot \nabla f_{j,i_2}
    +\sum_{j=0}^nh_j v^\gamma \left(\sum_{l=j+1}^{n}h_l\theta_l\right)\cdot \nabla \theta_j
    +\sum_{j=0}^n h_j'\theta_j\\
    =& \, F_0+\sum_{j=1}^n h_j^2 \sum_{(i_1,i_2)\in\mathcal{J}_{I_\gamma}} v^\gamma(f_{j,i_1})\cdot \nabla f_{j,i_2}+\sum_{j=1}^nh_j (v^\gamma-\overline{v}^\gamma )(\Theta_{j-1})\cdot \nabla \theta_j\\
    &+\sum_{j=1}^n h_j v^\gamma (\theta_j)\cdot \nabla \left( \sum_{l=0}^{j-1} h_l\theta_l\right)
    +\sum_{j=1}^n h_j'\theta_j=\sum_{j=0}^n F_j=F(x,t).
\end{align*}

\item \textbf{Sobolev regularity of the force.} Next we are going to check that the estimates  of each term $F_n$ of the force $F$ given by Lemmas \ref{fuerzaanteriores}, \ref{fuerzacubico}, \ref{fuerzapert} and \ref{fuerzaactivacion} are small enough to make them summable. In particular, we are interested in working with $\|F_n\|_{H^\kappa}$. Notice that the term given by Lemma \ref{fuerzaanteriores} controls the one given by Lemma \ref{fuerzacubico}. Indeed, since 
\begin{align*}
N_{n-1}^{\,2+\varepsilon}\,N_{n}^{\,-4-\gamma+\kappa+7\frac{\varepsilon+\delta}{1+\gamma+\delta}+\varepsilon}
 \le
N_{n-1}^{1+\varepsilon} N_n^{-3-\gamma+2\frac{\varepsilon+\delta}{1+\gamma+\delta}+\varepsilon+\kappa} 
\iff   N_{n-1} \le N_{n}^{1-5\frac{\varepsilon+\delta}{1+\gamma+\delta}}.
\end{align*}
Then, in light of \eqref{crecimientoN} it suffices to prove that 
\begin{align*}
    \frac{\varepsilon+\delta}{(1+\gamma+\delta)(\varepsilon-\mu)} &<  1-5\frac{\varepsilon+\delta}{1+\gamma+\delta}=\frac{1+\gamma-5\varepsilon-4\delta}{1+\gamma+\delta}\\
    & \hspace{-2cm}\iff 
    1+\frac{\delta+\mu}{\varepsilon-\mu} <  1+\gamma-5\varepsilon-4\delta\\
    &\hspace{-2cm} \iff 
   5\varepsilon + 4\delta + \frac{\delta+\mu}{\varepsilon-\mu} < \gamma
\end{align*}
and the latter holds in view of \eqref{eq:eleccion}:  
\begin{align*}
    5\varepsilon+4\delta+\frac{\delta+\mu}{\varepsilon-\mu}
    < 5\varepsilon+\frac{4\gamma}{5}\varepsilon+\frac{\gamma}{4} <  \frac{13}{15}\gamma < \gamma. 
\end{align*}

On the other hand, if we wish that the right-hand side of the inequality  \eqref{P13.8.1} decays as $n \to \infty$, we should choose $\kappa$ such that
\begin{align}\label{P13.8.3}
   -2-\gamma+\kappa+\varepsilon+(1-\kappa)\frac{\varepsilon+\delta}{1+\gamma+\delta} < 0 \iff  \kappa<2+\gamma+\delta.
\end{align}
The same comment also applies to the inequality \eqref{P13.8.2} under the same restriction on $\kappa$ given by \eqref{P13.8.3}. Assume that $\kappa_0$ satisfies  
\begin{align}
\label{kappadelta}
    \kappa_0 <2 + \gamma + \delta,
\end{align}
then the self-interaction and activation terms considered in Lemmas \ref{fuerzapert} and \ref{fuerzaactivacion} become relevant in the forcing term. Note that \eqref{kappadelta} is compatible with \eqref{P13.8.4}, that is, 
$$
  2 + \gamma + \frac{\gamma^2 (1-\gamma)}{25(4+\gamma)} < 2 + \gamma + \delta \iff \frac{\gamma^2 (1-\gamma)}{25(4+\gamma)}  < \frac{9 \gamma^2 (1-\gamma)}{200 (4+\gamma)}. 
$$

Note that the right-hand side of the inequality \eqref{P11.8.4} (and in particular, the right-hand side of \eqref{P13.8.5}) is dominated by (see \eqref{crecimientoN})
$$
    N_{n-1}^{1+\varepsilon} N_n^{-3-\gamma+2\frac{\varepsilon+\delta}{1+\gamma+\delta}+\varepsilon+\kappa} \leq C N_n^{(1+\varepsilon)\frac{\varepsilon+\delta}{(1+\gamma+\delta)(\varepsilon-\mu)} -3-\gamma+2\frac{\varepsilon+\delta}{1+\gamma+\delta}+\varepsilon+\kappa }
$$
and hence its contribution  to the forcing term is subject to the fact that desired parameter  $\kappa_0$  fulfills  
\begin{align}\label{P13.8.6}
    \kappa_0 < 3+\gamma - (1+\varepsilon)\frac{\varepsilon+\delta}{(1+\gamma+\delta)(\varepsilon-\mu)} -2\frac{\varepsilon+\delta}{1+\gamma+\delta}-\varepsilon.
\end{align}
After some basic computations, which rely on the choice of parameters given in  \eqref{eq:eleccion}, one may check that \eqref{P13.8.6} is  compatible with \eqref{P13.8.4}, that is, 
$$
   2 + \gamma + \frac{\gamma^2 (1-\gamma)}{25(4+\gamma)} < 3+\gamma  -(1+\varepsilon)\frac{\varepsilon+\delta}{(1+\gamma+\delta)(\varepsilon-\mu)} -2\frac{\varepsilon+\delta}{1+\gamma+\delta}-\varepsilon. 
$$


If we write 
\begin{align}\label{P13.8.11}
    \kappa_{0}= 2+\gamma +\frac{4\gamma\varepsilon}{25}+\kappa_{\text{pert}},
\end{align}
for some $\kappa_{\text{pert}}> 0$ to be chosen, we have 
\begin{align}
     (1+\varepsilon)\frac{\varepsilon+\delta}{(1+\gamma+\delta)(\varepsilon-\mu)} -3-\gamma+2\frac{\varepsilon+\delta}{1+\gamma+\delta}+\varepsilon+\kappa_0 & \nonumber \\
     &\hspace{-8.7cm}=-1+\frac{\varepsilon+\delta}{(1+\gamma+\delta)(\varepsilon-\mu)} +\varepsilon\frac{\varepsilon+\delta}{(1+\gamma+\delta)(\varepsilon-\mu)} +2\frac{\varepsilon+\delta}{1+\gamma+\delta}+\varepsilon+\frac{4\gamma\varepsilon}{25}+\kappa_{{\text{pert}}}. \label{P13.8.7}
\end{align}
From \eqref{eq:eleccion}, it is easy to see that 
\begin{align}\label{cotacociente}
    \frac{\varepsilon+\delta}{(1+\gamma+\delta)(\varepsilon-\mu)}=\frac{1+\frac{9\gamma}{50}}
{(1+\gamma+\delta)\big(1-\frac{9\gamma}{1000}\big)}\leq \frac{1+\frac{\gamma}{5}}{1+\gamma}
\end{align}
and so the first two terms in \eqref{P13.8.7} can be estimated as follows
\begin{align}\label{P13.8.8}
    -\left(1-\frac{\varepsilon+\delta}{(1+\gamma+\delta)(\varepsilon-\mu)} \right)\leq -\frac{4\gamma}{5(1+\gamma)}\leq-\frac{16(4+\gamma)}{5(1+\gamma)}\varepsilon\leq -6\varepsilon,
\end{align}
while the remaining term in \eqref{P13.8.7}, using that the quantity in the right-hand side of \eqref{cotacociente} is dominated by 1 together with the definitions given in \eqref{eq:eleccion},
\begin{align}
    \varepsilon\frac{\varepsilon+\delta}{(1+\gamma+\delta)(\varepsilon-\mu)} +2\frac{\varepsilon+\delta}{1+\gamma+\delta}+\varepsilon+\frac{4\gamma\varepsilon}{25}+\kappa_{\text{pert}} \nonumber \\
    \leq \varepsilon+2\varepsilon+\frac{2}{5}\varepsilon+\varepsilon+\frac{4}{25}\varepsilon+\kappa_{\text{pert}} \leq 5\varepsilon+\kappa_{\text{pert}}. \label{P13.8.9}
\end{align}
As a combination of \eqref{P13.8.7}, \eqref{P13.8.8} and \eqref{P13.8.9}, we establish 
\begin{align}\label{P13.8.10}
     (1+\varepsilon)\frac{\varepsilon+\delta}{(1+\gamma+\delta)(\varepsilon-\mu)} -3-\gamma+2\frac{\varepsilon+\delta}{1+\gamma+\delta}+\varepsilon+\kappa_0  \leq -\varepsilon +  \kappa_{\text{pert}}. 
\end{align}

In view of \eqref{P13.8.10}, if we choose $\kappa_0$ as in \eqref{P13.8.11} with
$$0<\kappa_{\text{pert}}<\frac{\gamma\varepsilon}{50}$$ then the required assumptions \eqref{kappadelta} and \eqref{P13.8.6} hold, as well as  \eqref{P13.8.4}. Hence for any $\kappa \leq \kappa_0$ (in particular, $\kappa < 3$ by \eqref{kappadelta}) we are in a position to apply Lemmas \ref{fuerzaanteriores}, \ref{fuerzacubico}, \ref{fuerzapert} and \ref{fuerzaactivacion} to get that  
 \begin{align*}
       \int_0^1\|F\|_{H^\kappa} \, dt & \leq \sum_{n=0}^\infty \int_{t_n}^1 \|F_n\|_{H^\kappa} \, dt  \\
       & \hspace{-1.8cm} \leq \sum_{n=1}^\infty \sup_{t \in [t_n, 1]} \left[ \|v^\gamma(\theta_{n})\cdot \nabla \Theta_{n-1}\|_{H^\kappa} + \|(v^\gamma-\overline{v}^\gamma)(\Theta_{n-1})\cdot \nabla \theta_n\|_{H^\kappa} +   \left\|
\sum_{(j_1,j_2)\in \mathcal{J}_{I_\gamma}} v^\gamma(f_{n,j_1})\cdot \nabla f_{n,j_2}
        \right\|_{H^\kappa} \right]   \\
        & + \sum_{n=1}^\infty \int_{t_n}^1h_n'(t)\|\theta_n(\cdot,t)\|_{H^\kappa} \, dt +\int_0^1 \|F_0\|_{H^\kappa}\,dt\leq C < \infty,
   \end{align*}
   where we have also used the properties of $h_n$ given in \eqref{P2.8.3}. 


\item \textbf{Blow-up of the solution.} Let $T\in (t_3,1)$. Then there exists a unique $n\in \mathbb{N}$ such that
   \begin{align}\label{P14.8.8}
    t_{n+1} < T \leq  t_{n+2}.
   \end{align}
   By construction, we have $h_j(t)=0$ for any $t<t_{n+1}$ $j\geq n+1$. Then, noting that $t_n + N_{n-1}^{-2 \mu} < t_{n+1}$ and $h_n(t) = 1$ for $t \geq t_n + N_{n-1}^{-2 \mu}$ since $N_{n-1}^{-2 \mu} > N_{n-1}^{-10 \varepsilon}$, we obtain  
   \begin{align}
   \label{cotainf}
   \begin{split}
       \int_0^T &\|\theta(\cdot,t)\|_{H^\kappa} \, dt
       \geq \int_{t_n+N_{n-1}^{-2\mu}}^{t_{n+1}}\|\theta(\cdot,t)\|_{H^\kappa} \, dt \\
       &\geq \int_{t_{n}+N_{n-1}^{-2\mu}}^{t_{n+1}}\left(\|h_n(t)\theta_n(\cdot,t)\|_{H^\kappa}-\sum_{j=0}^{n-1}\|\theta_j(\cdot,t)\|_{H^\kappa}
       \right) \, dt\\
        & = \int_{t_n+N_{n-1}^{-2\mu}}^{t_{n+1}}\left(\|\theta_n(\cdot,t)\|_{H^\kappa}-\sum_{j=0}^{n-1}\|\theta_j(\cdot,t)\|_{H^\kappa}
       \right) \, dt\\
       &\geq \int_{t_n+N_{n-1}^{-2\mu}}^{t_{n+1}}\left(\|f_{n,0}(\cdot,t)\|_{H^\kappa}
       -\sum_{i=1}^{I_\gamma}\|f_{n,i}(\cdot,t)\|_{H^\kappa}
       -\sum_{j=0}^{n-1}\|\theta_j(\cdot,t)\|_{H^\kappa}
       \right) \, dt.
       \end{split}
   \end{align}
   
   Using that $f_{n,i}$ with $i=1, \ldots, I_\gamma$ generates a sequence of derivatives with amplitude $\mathcal{C}_{n, i}$, principal directions $(\alpha_n, \beta_n)$ and principal frequencies $(r_n, s_n)$ and interpolation in Sobolev spaces, it is easy to prove that 
   $$
     \|f_{n,i}(\cdot,t)\|_{H^\kappa} \leq C \mathcal{C}_{n,i}(t) s_n(t)^{\kappa}|\mathcal{S}_n|^{\frac{1}{2}}, \qquad \forall t \in [t_n, 1].  
   $$
   From this and \eqref{P5.8.1}, \eqref{P12.8.5} and \eqref{P7.8.1}, 
   \begin{align}\label{P14.8.6}
       \|f_{n,i}(\cdot,t)\|_{H^\kappa}& \leq C\mathcal{C}_{n,0}N_{n}^{-i\left(2\frac{\varepsilon+\delta}{1+\gamma+\delta}-\varepsilon\right)} N_n^\kappa N_n^{-1+\frac{\varepsilon+\delta}{1+\gamma+\delta}} \leq CN_n^{-2-\gamma+\varepsilon+\kappa+\frac{\varepsilon+\delta}{1+\gamma+\delta}-i\left(2\frac{\varepsilon+\delta}{1+\gamma+\delta}-\varepsilon\right)}.
   \end{align}
   Similarly, one can also obtain that, for $t \geq t_j$, 
   \begin{align}
   \label{cotasobolevcapa}
       \|\theta_j\|_{H^\kappa}\leq \sum_{i=0}^{I_\gamma} \|f_{j, i}(\cdot, t)\|_{H^\kappa} \leq  C\mathcal{C}_{j,0}s_j(t)^{\kappa}|\mathcal{S}_j|^{\frac{1}{2}}\leq  CN_j^{-2-\gamma+\varepsilon+\kappa+\frac{\varepsilon+\delta}{1+\gamma+\delta}}.
   \end{align}
   Next, we wish to establish a lower bound of the $H^\kappa$ norm of $f_{n,0}$. Applying classical interpolation inequalities for Sobolev semi-norms (noting that, in particular,  $\kappa > 1$)  
   \begin{align*}
       \|f_{n,0}\|_{\dot{H}^1}\leq C\|f_{n,0}\|_{L^2}^{\frac{\kappa-1}{\kappa}}\|f_{n,0}\|_{\dot{H}^\kappa}^{\frac{1}{\kappa}}  
   \end{align*}
   leads to 
   \begin{align}\label{P14.8.3}
    \|f_{n,0}\|_{\dot{H}^\kappa}\geq \frac{\|f_{n,0}\|_{\dot{H}^1}^\kappa}{C^\kappa\|f_{n,0}\|_{L^2}^{\kappa-1}}.
   \end{align}
   The gradient of $f_{n,0}$ (see \eqref{capaprincipal}) is 
   \begin{align*}
       \nabla f_{n,0}(x,t)&= \mathcal{C}_{n, 0} \phi'(r_n(t)\textbf{n}(\alpha_n(t))\cdot 
 x)\phi(s_n(t)\textbf{n}(\beta_n(t))\cdot x)r_n(t)\textbf{n}(\alpha_n(t))\\
  &\hspace{1cm}+\mathcal{C}_{n, 0}  \phi(r_n(t)\textbf{n}(\alpha_n(t))\cdot 
 x)\phi'(s_n(t)\textbf{n}(\beta_n(t))\cdot x)s_n(t)\textbf{n}(\beta_n(t)),
   \end{align*}
   and
   \begin{align}
       \|\nabla f_{n,0} (\cdot, t)\|_{L^2}& \geq \mathcal{C}_{n,0}s_n(t)\|\phi(r_n(t)\textbf{n}(\alpha_n(t))\cdot 
 x)\phi'(s_n(t)\textbf{n}(\beta_n(t))\cdot x)\|_{L^2}\nonumber \\
&\hspace{1cm} -\mathcal{C}_{n,0}r_n(t)\|\phi'(r_n(t)\textbf{n}(\alpha_n(t))\cdot 
 x)\phi(s_n(t)\textbf{n}(\beta_n(t))\cdot x)\|_{L^2}.\label{normaH1}
   \end{align}
   Now, a simple change of variables gives 
   \begin{align*}
       \|\phi(r_n\textbf{n}(\alpha_n)\cdot 
 x)\phi'(s_n\textbf{n}(\beta_n)\cdot x)\|_{L^2}^2
 =&\int_{\mathbb{R}^2} \phi(r_n\textbf{n}(\alpha_n)\cdot 
 x)^2\phi'(s_n\textbf{n}(\beta_n)\cdot 
 x)^2 \, dx\\
 =&\frac{1}{r_ns_n |\sin(\alpha_n-\beta_n)|} \, \|\phi\|_{L^2(\mathbb{R})}^2 \|\phi'\|_{L^2(\mathbb{R})}^2
   \end{align*}
   and analogously 
   \begin{align*}
  \|\phi'(r_n\textbf{n}(\alpha_n)\cdot 
 x)\phi(s_n\textbf{n}(\beta_n)\cdot x)\|_{L^2}^2
 =\frac{1}{r_ns_n |\sin(\alpha_n-\beta_n)|} \, \|\phi\|_{L^2(\mathbb{R})}^2 \|\phi'\|_{L^2(\mathbb{R})}^2. 
   \end{align*}
Using this and (\ref{normaH1}), we have 
\begin{align}\label{P14.8.1}
    \|\nabla f_{n,0}(\cdot,t)\|_{L^2} \geq \mathcal{C}_{n,0}\frac{s_n-r_n}{C\sqrt{r_ns_n| \sin(\alpha_n-\beta_n)|}}.
\end{align}
In addition, it follows from \eqref{estimar} and \eqref{estimas} under  $t\geq t_n+N_{n-1}^{-2\mu}$, \eqref{ic} and \eqref{crecimientoN} that 
\begin{align*}
    s_n(t) - r_n(t) \gtrsim s_n(t) \gtrsim s_n(t_n) N_{n-1}^\varepsilon (t-t_n)  \gtrsim N_n^{1-\frac{\varepsilon + \delta}{1+\gamma+\delta}} N_{n-1}^{\varepsilon-2 \mu} \gtrsim N_n^{1-\frac{\varepsilon + \delta}{1+\gamma+\delta}} N_n^{(\varepsilon - 2 \mu) \frac{\varepsilon + \delta}{(1+\gamma+\delta)(\varepsilon-\mu)} }.
\end{align*}
It is a routine to check that 
$$
    (\varepsilon - 2 \mu) \frac{\varepsilon + \delta}{(1+\gamma+\delta)(\varepsilon-\mu)} \geq \frac{\varepsilon + \delta}{1+\gamma+\delta}-\mu  \iff (\varepsilon -\mu) (1+\gamma+\delta) \geq \varepsilon+\delta
$$
and the validity of the right-hand side inequality (see \eqref{eq:eleccion}). As a byproduct, we derive that
\begin{align}\label{P14.8.2}
    s_n(t)-r_n(t) \gtrsim N_n^{1-\mu}, \qquad t\geq t_n+N_{n-1}^{-2\mu}. 
\end{align}
Inserting \eqref{P14.8.2} and \eqref{P12.8.5} into \eqref{P14.8.1} (recall also the value of $\mathcal{C}_{n, 0}$ given by \eqref{P7.8.1}), we achieve the following 
\begin{align}\label{P14.8.4}
      \|\nabla f_{n,0}(\cdot,t)\|_{L^2}  \geq \frac{1}{C} \, N_n^{-1-\gamma+\varepsilon+\frac{\varepsilon+\delta}{1+\gamma+\delta}-\mu}, \qquad t\geq t_n+N_{n-1}^{-2\mu}.
\end{align}
Similarly, we have 
\begin{align}\label{P14.8.5}
\|f_{n,0}\|_{L^2}\leq  C N_n^{-2-\gamma+\varepsilon + \frac{\varepsilon+\delta}{1+\gamma+\delta}}, \qquad  t \geq t_n.
\end{align}
Putting now together \eqref{P14.8.3}, \eqref{P14.8.4} and \eqref{P14.8.5}, we get
\begin{align}\label{P14.8.7}
    \|f_{n,0}\|_{\dot{H}^\kappa}\geq \frac{1}{C}  N_n^{-2-\gamma+\kappa+\varepsilon+\frac{\varepsilon+\delta}{1+\gamma+\delta}-\kappa\mu}, \qquad t\geq t_n+N_{n-1}^{-2\mu}.
\end{align}

Plugging \eqref{P14.8.6}, \eqref{cotasobolevcapa} and \eqref{P14.8.7} into \eqref{cotainf} (and taking into account that $t_{n+1}-t_n-N_{n-1}^{-2 \mu} \gtrsim N_{n-1}^{-\mu}$ and $2 + \gamma \leq \kappa  \leq 3$) we obtain 
\begin{align}
     \int_0^T \|\theta(\cdot,t)\|_{H^\kappa} \, dt
       & \nonumber \\
       & \hspace{-2cm}\geq \frac{1}{C}N_n^{-2-\gamma+\kappa+\varepsilon+\frac{\varepsilon+\delta}{1+\gamma+\delta}-\kappa \mu} N_{n-1}^{-\mu}  
       -CN_n^{-2-\gamma+\kappa+\varepsilon+\frac{\varepsilon+\delta}{1+\gamma+\delta}-\left(2\frac{\varepsilon+\delta}{1+\gamma+\delta}-\varepsilon\right)} N_{n-1}^{-\mu} -CN_{n-1}^{-2-\gamma+\kappa +\varepsilon+\frac{\varepsilon+\delta}{1+\gamma+\delta}-\mu} \nonumber \\
       & \hspace{-2cm}\geq \frac{1}{C}N_n^{-2-\gamma+\kappa+\varepsilon+\frac{\varepsilon+\delta}{1+\gamma+\delta}-4 \mu} \geq N_n^{\varepsilon},\label{P14.8.9}
\end{align}
where the second inequality may be justified as follows: Since $(\varepsilon+\delta)\big(\frac{2}{1+\gamma+\delta}-1\big)>0$  (see \eqref{eq:eleccion}), we have
\begin{align*}  
2\frac{\varepsilon+\delta}{1+\gamma+\delta}-\varepsilon -\kappa \mu \geq  2\frac{\varepsilon+\delta}{1+\gamma+\delta}-\varepsilon - \frac{3 \delta}{20}> \frac{17 \delta}{20} > 0
\end{align*} 
and then 
\begin{align*}
       \frac{N_n^{-2-\gamma+ \kappa + \varepsilon+\frac{\varepsilon+\delta}{1+\gamma+\delta}-\left(2\frac{\varepsilon+\delta}{1+\gamma+\delta}-\varepsilon\right)}}{N_n^{-2-\gamma+\kappa+\varepsilon+\frac{\varepsilon+\delta}{1+\gamma+\delta}-\kappa \mu}} \to 0 \qquad \text{as} \qquad n \to \infty. 
\end{align*}
On the other hand, using successively $\kappa  \geq 2+\gamma$ and $\frac{1}{1+x}\leq 1-\frac{x}{2}$ for $x = \frac{\gamma}{2} \in[0,1]$,
\begin{align*}
-2-\gamma+\kappa+\varepsilon+\frac{\varepsilon+\delta}{1+\gamma+\delta}-4 \mu &>\left(-2-\gamma+\kappa +\varepsilon+\frac{\varepsilon+\delta}{1+\gamma+\delta}-\mu\right) \left(1-\frac{\gamma}{4} \right)\\
&\geq
    \frac{-2-\gamma+\kappa+\varepsilon+\frac{\varepsilon+\delta}{1+\gamma+\delta}-\mu}{1+\frac{\gamma}{2}}
\end{align*}
and then, by \eqref{des} (see also Remark \ref{Remark25}), 
\begin{align*}
    \frac{N_{n-1}^{-2-\gamma+\kappa+\varepsilon+\frac{\varepsilon+\delta}{1+\gamma+\delta}-\mu}}{N_n^{-2-\gamma+\kappa+\varepsilon+\frac{\varepsilon+\delta}{1+\gamma+\delta}-\kappa \mu} N_{n-1}^{-\mu}} \leq N_n^{-(4-\kappa)\mu} N_{n-1}^\mu \leq N_n^{-(3-\kappa)\mu} \to 0 \qquad \text{as} \qquad n \to \infty. 
\end{align*}


In light of \eqref{P14.8.8},  $n \to \infty$ as $T \nearrow 1$ and then \eqref{P14.8.9} yields  
$$
    \lim_{T \nearrow 1}  \int_0^T \|\theta(\cdot,t)\|_{H^\kappa} \, dt \geq \lim_{n \to \infty} N_n^\varepsilon = \infty. 
$$

\item \textbf{Sobolev regularity of the solution up to $t=1$ in the supercritical setting.} For any $t \in [0, 1)$, choose the unique value $n \in \mathbb{N}_0$ such that $t \in [t_n, t_{n+1})$.  In particular, $h_j(t) = 0$ if $j \geq n+1$ and $\theta(\cdot, t) = \sum_{j=0}^n h_j(t) \theta_j (\cdot, t)$. Assume that $\kappa\in \big[0,2+\gamma-\frac{\gamma(1-\gamma)}{2(4+\gamma)}\big]$. From  \eqref{cotasobolevcapa},  
\begin{align*}
      \|\theta_j(\cdot, t)\|_{H^\kappa}\leq  CN_j^{-2-\gamma+\varepsilon+\kappa+\frac{\varepsilon+\delta}{1+\gamma+\delta}} \leq C N_j^{-\varepsilon+\frac{\varepsilon+\delta}{1+\gamma+\delta}}
\end{align*}
for $j \in \{0, \ldots, n\}$. 
Hence, taking into account that  $-\varepsilon + \frac{\varepsilon+\delta}{1+\gamma+\delta} < 0$ (see  \eqref{eq:eleccion}), 
\begin{align*}
    \|\theta(\cdot, t)\|_{H^\kappa} \leq \sum_{j=0}^n \|\theta_j(\cdot, t)\|_{H^\kappa} \leq C \sum_{j=0}^n N_j^{-\varepsilon+\frac{\varepsilon+\delta}{1+\gamma+\delta}} \leq C \sum_{j=0}^\infty N_j^{-\varepsilon+\frac{\varepsilon+\delta}{1+\gamma+\delta}} < \infty.
\end{align*}
\end{enumerate}
          \end{proof}

\section*{Acknowledgements}

This work is supported in part by the Spanish Ministry of Science
and Innovation, through the “Severo Ochoa Programme for Centres of
Excellence in R\&D (CEX2023-001347-S)” and PID2023-152878NB-I00. O.D.  is also partially supported by the grants RYC2022-037402-I 
and CNS2025-166944 funded by the Spanish Ministry of Science and Innovation.

\end{document}